\documentclass{amsart} 
\usepackage[utf8]{inputenc}
\usepackage{color}
\usepackage{amsmath} 
\usepackage{amssymb} 
\usepackage{amsthm}
\usepackage{graphicx}
\usepackage{esint}
\usepackage[colorlinks=true,linkcolor=blue]{hyperref}
\usepackage[german,english]{babel}

\theoremstyle{plain}
\newtheorem{thm}{Theorem}
\newtheorem{prop}{Proposition}[section]

\newtheorem{conv}[prop]{Convention}
\newtheorem{lem}[prop]{Lemma}
\newtheorem{cor}[prop]{Corollary}
\newtheorem{assum}[prop]{Assumption}
\newtheorem{defi}[prop]{Definition}
\newtheorem{rmk}[prop]{Remark}

\newtheorem{example}[prop]{Example}

\newcommand {\R} {\mathbb{R}} 
 \newcommand {\N} {\mathbb{N}}
 
\newcommand {\p} {\partial}

\newcommand {\dv} {\partial_{\varphi}}
\newcommand {\va} {\varphi}
\newcommand {\D} {\Delta}

\newcommand {\supp} {\text{supp}}
\newcommand {\diam} {\text{diam}}

\newcommand{\LL}{\tilde{L}^2}
\newcommand{\LLw}{\tilde{L}^2_{\bar{\omega}}}

\DeclareMathOperator {\dist} {dist}
\DeclareMathOperator {\Ree} {Re}
\DeclareMathOperator {\Imm} {Im}

\DeclareMathOperator {\inte} {int}
\DeclareMathOperator {\spa} {span}

\usepackage{xspace}
\newcommand{\thmtwo}{2 in \cite{KRSI}\xspace}
\newcommand{\rmkeigen}{16 in \cite{KRSI}\xspace}

\begin{document} 
\title{Higher Regularity for the fractional thin obstacle problem}

\author{Herbert Koch}
\author{Angkana R\"uland }
\author{Wenhui Shi}

\address{
Mathematisches Institut, Universit\"at Bonn, Endenicher Allee 60, 53115 Bonn, Germany }
\email{koch@math.uni-bonn.de}

\address{
Mathematical Institute of the University of Oxford, Andrew Wiles Building, Radcliffe Observatory Quarter, Woodstock Road, OX2 6GG Oxford, United Kingdom }
\email{ruland@maths.ox.ac.uk}

\address{
Mathematisches Institut, Universit\"at Bonn, Endenicher Allee 64, 53115 Bonn, Germany  }
\email{wenhui.shi@hcm.uni-bonn.de}

\begin{abstract}
In this article we investigate the higher regularity properties of the regular free boundary in the fractional thin obstacle problem. Relying on a Hodograph-Legendre transform, we show that for smooth or analytic obstacles the regular free boundary is smooth or analytic, respectively. This leads to the analysis of a fully nonlinear, degenerate (sub)elliptic operator which we identify as a (fully nonlinear) perturbation of the fractional Baouendi-Grushin Laplacian. Using its intrinsic geometry and adapted function spaces, we invoke the analytic implicit function theorem to deduce analyticity of the regular free boundary.
\end{abstract}



\thanks{
H.K. acknowledges support by the DFG through SFB 1060.
A.R. acknowledges a Junior Research Fellowship at Christ Church.
W.S. is supported by the Hausdorff Center of Mathematics.}

\maketitle
\tableofcontents

\section{Introduction}
In this article we study higher regularity properties of the regular free boundary associated with the ``fractional thin obstacle problem''. More precisely, given an \emph{obstacle} $\phi:B'_1\rightarrow \R$ and assuming that $s\in(0,1)$, we consider local minimizers of the functional
\begin{align*}
J(\tilde{w})=\int\limits_{B_1^+}\left(\frac{1}{2}|\nabla \tilde{w}|^2+\tilde{w}\tilde{f}\right) x_{n+1}^{1-2s}  dx,
\end{align*}
in the convex, constrained set 
$$\mathcal{K}:=\{\tilde{w}\in H^1(B_1^+, x_{n+1}^{1-2s} dx): \ \tilde{w} \geq \phi \mbox{ in } B_1' \}.$$ 
Here $B_1^+:=\{x\in \R^{n+1}: |x|\leq 1, x_{n+1}\geq 0\}$ denotes the upper half-ball and $B_1':=B_1^+\cap\{x_{n+1}=0\}$ is the co-dimension one ball on the boundary of $\R^{n+1}_+$. If the obstacle $\phi$ and inhomogeneity $\tilde{f}$ are assumed to be in a suitable class, classical arguments involving variational inequalities ensure the existence of local minimizers.\\
The relation of a minimizer with the co-dimension one (hence ``thin'') set, on which it is constrained to lie above the obstacle, gives rise to three sets which are of importance in the sequel: The \emph{contact set} $\Lambda_{\tilde{w}}:=\{x\in B_1': \tilde{w}=\phi\}$, in which the obstacle is attained by the minimizer, the \emph{non-coincidence set} $\Omega_{\tilde{w}}:=\{x\in B_1': \tilde{w}>\phi\}$, on which the minimizer is strictly larger than the obstacle, and the \emph{free boundary} $\Gamma_{\tilde{w}}:= \partial \Omega_{\tilde{w}}\cap B_1'$, which separates the previous two sets. \\
Carrying out variations of the functional $J$ around minimizers, yields that a minimizer of $J$ in the class $\mathcal{K}$ solves a Signorini problem for the degenerate elliptic operator $\nabla \cdot x_{n+1}^{1-2s}\nabla$:
\begin{equation}
\label{eq:thin_00}
\begin{split}
\nabla \cdot x_{n+1}^{1-2s}\nabla \tilde{w}=x_{n+1}^{1-2s}\tilde{f}&\text{ in } B_1^+,\\
\tilde{w}\geq \phi &\text{ on } B'_1,\\
 \lim _{x_{n+1}\rightarrow 0_+} x_{n+1}^{1-2s}\p_{n+1}\tilde{w}\leq 0 &\text{ on } B'_1,\\
\lim_{x_{n+1}\rightarrow 0_+}x_{n+1}^{1-2s}\p_{n+1}\tilde{w}=0 &\text{ on }  B'_1\cap \{u>\phi\}.
\end{split}
\end{equation}
Relying on previous work on the the fractional thin obstacle problem (in particular on \cite{CSS, Si}) and on regularity assumptions for the inhomogeneity $\tilde{f}$, these equations will be understood in a pointwise sense in the sequel. In particular this holds for the \emph{complementary} (or \emph{Signorini}) boundary conditions on $B'_1$.\\

In investigating the higher regularity properties of solutions to the fractional thin obstacle problem, we build on the seminal work on the obstacle problem for the fractional Laplacian by Caffarelli, Silvestre and Salsa \cite{CSS}. As explained in Section \ref{sec:back} there is a close connection between the above Signorini problem \eqref{eq:thin_00} and the obstacle problem for the fractional Laplacian (c.f. \cite{CaS, CSS} and Section~\ref{sec:back}). Due to this reason we refer to the problem (\ref{eq:thin_00}) as the ``fractional thin obstacle problem''. This close relationship also allows us to exploit the results from \cite{CSS} in our context. \\
Let us briefly recall the, to us, most relevant results from \cite{CSS} (for a more detailed summary we refer to Section~\ref{sec:back}). Firstly, optimal regularity of the solutions is established: Assuming that $\tilde{f}\in C^{0,1}(B_1^+)$ and that $\phi\in C^{2,1}(B'_1)$, and assuming that the free boundary $\Gamma_{\tilde{w}}$ is compactly contained in $B'_{1/2}$ (which permits us to extend the local problem for the fractional Laplacian \eqref{eq:thin_00} into a global problem, c.f. Section \ref{sec:back}), the solution $\tilde{w}$ to \eqref{eq:thin_00} has the optimal regularity (up to the boundary $B_1'$)
\begin{equation}
\label{eq:opt_reg}
\begin{split}
&\p_i \tilde{w}\in C^{0,s}(B_{1/2}^+) \mbox{ for } i\in \{1,\dots, n\},\\
&x_{n+1}^{1-2s}\p_{n+1}\tilde{w} \in C^{0,1-s}(B_{1/2}^+).
\end{split}
\end{equation} 
The optimality of this can be seen by noting that the function 
\begin{align}
\label{eq:w1s}
w_{1,s}(x):= \frac{1}{s^2-1} \left(\sqrt{x_n^2 + x_{n+1}^2}+x_n\right)^s\left(s\sqrt{x_n^2 + x_{n+1}^2}- x_n\right), 
\end{align}
which will play the role of a model solution for us, satisfies \eqref{eq:thin_00} for $\tilde{f}=\phi=0$.\\
Furthermore, the free boundary $\Gamma_w$ decomposes as
\begin{align}
\label{eq:Gamma_reg}
\Gamma_{\tilde{w}} = \Gamma_{1+s}(\tilde{w})\cup \bigcup_{\kappa\geq 2}\Gamma_{\kappa}(\tilde{w}),
\end{align}
where $\Gamma_{\kappa}(\tilde{w}):=\left\{x_0 \in \Gamma_{\tilde{w}}: \frac{\Phi_{u,x_0}(0_+)-n-2-2s}{2}=\kappa\right\}$ and $\Phi_{\tilde{w},x_0}(r)$ denotes a truncated frequency function associated with the function $\tilde{w}(x)-\phi(x)-\frac{\Delta'\phi(x_0)+\tilde{f}(x_0)}{2(2-2s)}x_{n+1}^2$ at the point $x_0\in \Gamma_{\tilde{w}}$ (c.f. Section \ref{sec:back} for more details). The set $\Gamma_{1+s}(\tilde{w})$, which is denoted as the \emph{regular free boundary}, is an open subset of $\Gamma_{\tilde{w}}$ and it is locally a $C^{1,\alpha}$ graph for some $\alpha>0$. 

\subsection{Main result}
In this article we seek to derive an improved understanding of the higher regularity properties of the regular free boundary $\Gamma_{s+1}(\tilde{w})$. Similar as in \cite{CSS} we first reduce the setting to the zero obstacle problem by considering the equation for $w=\tilde{w}-\phi$. This function then solves the Signorini problem 
\begin{equation}
\label{eq:fracLa}
\begin{split}
\nabla \cdot x_{n+1}^{1-2s} \nabla w &= x_{n+1}^{1-2s} \tilde{f} \mbox{ in } B_1^+,\\
w &\geq 0 \text{ on } B'_1,\\
 \lim\limits_{x_{n+1}\rightarrow 0_+} x_{n+1}^{1-2s}\p_{n+1} w&\leq 0 \mbox{ on } B_1',\\
\lim\limits_{x_{n+1}\rightarrow 0_+} x_{n+1}^{1-2s}\p_{n+1} w&=0 \mbox{ on } B_1'\cap\{w>0\}.
\end{split}
\end{equation}
Considering obstacles and inhomogeneities of suitably high regularity, we may assume that the resulting inhomogeneity $\tilde{f}$ is at least $C^{3,1}(B_1^+)$ regular.\\
In this set-up our main result asserts the smoothness and even analyticity of the regular free boundary.

\begin{thm}
\label{thm:ana1}
Let $w:B_1^+ \rightarrow \R$  be a solution to (\ref{eq:fracLa}) with inhomogeneity $\tilde{f}$. Assume that $\p_iw\in C^{0,s}_{loc}(B_1^+)$ for $i\in \{1,\dots, n\}$ and $x_{n+1}^{1-2s}\p_{n+1}w\in C^{0,1-s}_{loc}(B_1^+)$. 
Then, if $\tilde{f}$ is smooth, the regular free boundary $\Gamma_{s+1}(w)$ is locally smooth. If $\tilde{f}$ is real analytic, the regular free boundary $\Gamma_{s+1}(w)$ is locally real analytic.
\end{thm}

Here the assumption on the smoothness of the inhomogeneity is due to the desire of avoiding technicalities as far as possible. We however emphasize that the situation of smooth inhomogeneities is not the only case in which our arguments hold (c.f. Remark \ref{rmk:inhom} in Section \ref{sec:IFT}). Also, the assumption on the regularity of the solution $w$ is not a major restriction. For example, it covers the problem studied in \cite{CSS}.

\subsection{Strategy of the proof} In order to infer the result of Theorem \ref{thm:ana1}, we rely on a partial Hodograph-Legendre transform, a precise analysis of the resulting fully nonlinear, degenerate (sub)elliptic equation, and the implicit function theorem. We discuss these ingredients in greater detail in the sequel.\\ 

\emph{Definition of the Hodograph-Legendre transform.} Seeking to fix and straighten the free boundary, we carry out a (partial) Hodograph-Legendre transform \cite{KN77} of the problem at hand. In this context, the choice of the dependent and independent variables requires certain care: In contrast to the case $s=\frac{1}{2}$, which corresponds to the classical thin obstacle problem, the choice
\begin{align*}
y'' = x'', \ y_n = \p_n w, \ y_{n+1} = x_{n+1}^{1-2s}\p_{n+1}w,
\end{align*}
of which one would hope that it suffices to fix the free boundary, is not ideal. Indeed, considering the model solution $w_{1,s}$ from \eqref{eq:w1s} yields
\begin{align*}
\p_nw_{1,s}(x)&=\left(\sqrt{x_n^2 + x_{n+1}^2}+x_n\right)^s, \\
 x_{n+1}^{1-2s}\p_{n+1}w_{1,s}(x)&=\frac{s}{s-1}\left(\sqrt{x_n^2+x_{n+1}^2}-x_n\right)^{1-s}. 
\end{align*}
This indicates that
\begin{align}
\label{eq:LH}
y'' = x'', \ y_n^{2s} = \p_n w, \ y_{n+1}^{2(1-s)} = -c_s x_{n+1}^{1-2s}\p_{n+1}w,
\end{align}
for some $c_s>0$ provides a better choice of dependent and independent variables. Simplifying, we see that this change of coordinates then corresponds to the square root mapping
\begin{align*}
y''=x'', \ y_n = \Ree(x_n + ix_{n+1})^{1/2}, \ y_{n+1} = \Imm(x_n + i x_{n+1})^{1/2},
\end{align*}
which was already used in the analysis of \cite{KPS14}. The Hodograph-Legendre transformation hence maps the upper half-plane into the upper quarter space $Q_+:=\{y\in \R^{n+1}: y_n \geq 0, y_{n+1}\geq 0\}$ and maps the free boundary into the co-dimension two hyperplane $P:=\{y\in \R^{n+1}:y_n=y_{n+1}=0\}$. \\
Indeed, this heuristic argument for using (\ref{eq:LH}) is made rigorous by a careful analysis of solutions to (\ref{eq:fracLa}), for which we prove a leading order asymptotic expansion around the free boundary in terms of the model solution $w_{1,s}$ (c.f. Proposition \ref{prop:asymp1}).
As in \cite{KPS14} this analysis plays a central role, since general solutions to (\ref{eq:fracLa}) are not regular enough to prove the invertibility of the Hodograph-Legendre transform (\ref{eq:LH}) by means of the classical implicit function theorem. Instead, we use the asymptotics at the free boundary combined with elliptic estimates in annuli around it to deduce the invertibility of the transformation (c.f. Proposition \ref{prop:invert}).\\

\emph{Fractional fully nonlinear subelliptic equation.} As in \cite{KPS14} a second main step consists of analyzing the transformed equation. Defining the Legendre function as 
\begin{align*}
v(y):= w(x)- x_n y_n^{2s} + \frac{1}{2(1-s)} x_{n+1}^{2s} y_{n+1}^{2(1-s)},
\end{align*}
where $w$ is a solution to (\ref{eq:fracLa}), we note that the free boundary is parametrized as
\begin{align*}
x_n = - \frac{1}{2s}y_n^{1-2s}\p_n v(y)|_{y=(y'',0,0)}.
\end{align*}
Hence, seeking to deduce regularity of the regular free boundary, we study the regularity of the Legendre function $v$.
However, while the Hodograph-Legendre transform allows us to fix the free boundary, it comes at the expense of transforming our linear equation (\ref{eq:fracLa}) into a fully nonlinear, degenerate (sub)elliptic Monge-Amp\`ere type equation. Yet, as in the case of the thin obstacle problem, it is possible to deduce a certain structure for this equation and to view it as a perturbation of a \emph{fractional Baouendi-Grushin Laplacian} $\Delta_{G,s}=\sum\limits_{i=1}^{n+1}Y_i\omega(y)Y_i$, where $Y_i$, $i\in\{1,\dots,n+1\}$, denote the classical Baouendi-Grushin vector fields (c.f. Definition \ref{defi:BGgeo}) and where the weight $\omega(y)=(y_ny_{n+1})^{1-2s}$ for $s\in (0,1)$ belongs to Muckenhoupt class $A_2$.\\
To avoid a bootstrap argument in proving the higher (partial) regularity result, we apply the implicit function theorem as in \cite{KRSIII} (relying on the observation that the subelliptic structure is translation invariant in the tangential variables $y''$). Here the definition of suitable function spaces (such that conditions of the Banach implicit function theorem are satisfied) plays a pivotal role. These function spaces can be viewed as weighted generalizations of the generalized H\"older spaces from \cite{KRSIII}. Compared with \cite{KRSIII} the fractional character of the equation poses additional difficulties in constructing the spaces. The correct choice of the weights is of central importance (c.f. next point below). \\

\emph{Analyticity of the functional, function spaces.} Compared to the situation $s=\frac{1}{2}$, we encounter a further complication related to the additional ``fractional weight'' in our fully nonlinear operator: Due to our choice of dependent and independent coordinates in (\ref{eq:LH}), the fully nonlinear equation for the Legendre function $v$ involves \emph{non-integer} powers of (derivatives) of $v$. Seeking to prove analyticity of the Legendre function by means of the analytic implicit function theorem as in \cite{KRSIII}, therefore requires a careful choice of the function spaces to ensure that the resulting operator still yields an \emph{analytic} mapping from the domain into the image space.\\ To this end, we introduce \emph{weighted} versions of the generalized Hölder spaces from \cite{KRSIII}. We recall that the generalized Hölder spaces in \cite{KRSIII} were constructed in order to mimic the asymptotic expansion of our Legendre functions at the straightened boundary $P=\{y\in\R^{n+1}:y_n=y_{n+1}=0\}$ and were motivated by Campanato type norms \cite{Ca64}. 
In the setting of the fractional Baoeundi-Grushin operator, spaces which only reflect the asymptotic behavior at $P$ do not suffice: Due to the presence of the Muckenhoupt weight $(y_ny_{n+1})^{1-2s}$, our spaces also have to capture the asymptotic behavior at the planes $\{y_n=0\}$ and $\{y_{n+1}=0\}$ where the weights degenerate. By interpolating between the asymptotics at $P$ and the asymptotics at $\{y_n=0\}\cup \{y_{n+1}=0\}$, we construct weighted H\"older spaces with respect to the intrinsic Baouendi-Grushin geometry which are adapted to our problem. We show that with these choices the nonlinear operator is an analytic map from its domain into the image space. Moreover, by deducing ``Schauder type'' apriori estimates for the fractional Baoeundi-Grushin operator in our generalized H\"older spaces, we prove that the linearization of the nonlinear operator at $v$ is invertible. Thus, the analytic implicit function theorem can be applied in our spaces, which then yields our main result. 

\subsection{Context and literature}
In this section we relate the fractional thin obstacle problem to the obstacle problem for the fractional Laplacian and provide some background on the literature on these problems.

\subsubsection{Relation to the fractional obstacle problem}
\label{sec:back}
Let us consider the \emph{obstacle problem for the fractional Laplacian} $(-\Delta)^s$, $s\in (0,1)$: Given a function $\va:\R^n \rightarrow \R$ with rapid decay at infinity, one seeks a function $u$ with $\lim_{|x|\rightarrow \infty}u(x)=0$ which satisfies
\begin{align}
\label{eq:obst_frac}
\min\{(-\Delta)^s u(x), u(x)-\va(x)\}=0, \quad x\in \R^n.
\end{align}
Here $(-\Delta)^s$ is the fractional Laplacian, which for $s\in(0,1)$ can be defined as an integral operator
$$(-\Delta)^su(x):=c_{n,s}\ p.v.\int_{\R^n}\frac{u(x)-u(y)}{|x-y|^{n+2s}}dy,$$
and $c_{n,s}$ denotes a universal constant depending on $n,s$. In \cite{Si} Silvestre considered the existence and regularity of the solution to the obstacle problem \eqref{eq:obst_frac}. For $\phi \in C^{2}(\R^n)$ he proved that there exists a solution $u$ which is $C^{1,\beta}(\R^n)$ regular for all $\beta \in (0,s)$ and $(-\Delta )^s u \in C^{0,\gamma}$ for all $\gamma\in (0,1-s)$.\\

The relationship between the obstacle problem for the fractional Laplacian \eqref{eq:obst_frac} and the Signorini problem (\ref{eq:fracLa}) is established by the Dirichlet-to-Neumann map for the degenerate elliptic operator $L_s:=\nabla \cdot x_{n+1}^{1-2s}\nabla $. More precisely, given $u\in C_c(\R^n)$, we extend it to the upper half space $\R^{n+1}_+=\R^n\times \R_+$ by solving the Dirichlet problem
\begin{align*}
L_s \tilde{w}(x)=0 \text{ in } \R^{n+1}_+,\quad 
\tilde{w}(x',0)=u(x') \text{ on } \R^n\times\{0\}.
\end{align*} 
Then $\tilde{w}$ satisfies
\begin{align*}
\lim_{x_{n+1}\rightarrow 0_+} c_{n,s} x_{n+1}^{1-2s}\p_{n+1} \tilde{w}(x',x_{n+1})=-(-\Delta)^s u(x'),
\end{align*}
where $c_{n,s}>0$ is an only dimension and $s$ dependent constant (c.f. \cite{CaS}). Using this characterization, the obstacle problem for the fractional Laplacian $(-\Delta)^s$ can be reformulated as a Signorini problem for the degenerate elliptic operator $L_s$:
\begin{equation*}
\begin{split}
L_s \tilde{w} =0 &\text{ in } \R^{n+1}_+,\\
\tilde{w} \geq \va, \ \lim\limits_{x_{n+1}\rightarrow 0_+} x_{n+1}^{1-2s}\p_{n+1}\tilde{w}\leq 0, \ (\tilde{w}-\va)(\lim\limits_{x_{n+1}\rightarrow 0_+} x_{n+1}^{1-2s}\p_{n+1}\tilde{w})=0 &\text{ on } \R^{n}\times\{0\}.
\end{split}
\end{equation*}
Localizing the above problem by considering $w:=\tilde{w}\eta$, where $\eta$ is a radial cut-off function which is equal to one in $B_1^+$, we obtain the problem \eqref{eq:thin_00} with obstacle $\phi=\va \eta$.\\

Conversely, assume that $\tilde{w}\in L^{\infty}(B_1^+)$ is a solution to \eqref{eq:thin_00} with sufficiently regular obstacle $\phi:B_1'\rightarrow \R$ and sufficiently regular inhomogeneity $\tilde{f}$. Following the argument of Lemma 4.1. in \cite{CSS}, we can transform the local problem (\ref{eq:thin_00}) into a global problem of the form (\ref{eq:obst_frac}). 
Let us explain this reduction: As in Lemma~4.1 of \cite{CSS}, we consider the function $\tilde{w}-\phi$ and extend it globally by defining $\overline{w}:=(\tilde{w}-\phi)\eta$, where $\eta$ denotes a radial cut-off function which is supported in $B_{3/4}^+$ and which is equal to one in $B_{1/2}^+$.  
Then, $\overline{w}$ satisfies $L_s\overline{w}=x_{n+1}^{1-2s}\tilde{g}$ for a compactly supported function $\tilde{g}$ (which is computed in terms of $\tilde{f}, \eta, \phi$) with unchanged Neumann data (which is a consequence of the radial dependence of $\eta$). In particular, the inhomogeneity is non-trivial in general. To remedy this and to reduce the situation to that of the Caffarelli-Silvestre extension, we consider an auxiliary function $w$ which solves the equation $L_s w=x_{n+1}^{1-2s}\tilde{g}$ in $\R^{n+1}_+$ with $w=0$ on $\R^n\times\{0\}$. Then the function $\overline{w}-w$ satisfies
\begin{align*}
L_s (\overline{w}-w) &= 0 \mbox{ in } \R^{n+1}_+,\\
\overline{w}-w &= \eta(\tilde{w}-\phi) \mbox{ on } \R^{n}\times\{0\},\\
\lim\limits_{x_{n+1}\rightarrow 0_+} x_{n+1}^{1-2s} \p_{n+1} (\overline{w}-w) &\leq -\lim\limits_{x_{n+1}\rightarrow 0_+} x_{n+1}^{1-2s} \p_{n+1} w \mbox{ on } \R^{n}\times\{0\},\\
\lim\limits_{x_{n+1}\rightarrow 0_+} x_{n+1}^{1-2s} \p_{n+1} (\overline{w}-w) &=- \lim\limits_{x_{n+1}\rightarrow 0_+} x_{n+1}^{1-2s} \p_{n+1} w \mbox{ on } (\R^{n}\times\{0\})\cap\{\tilde{w}-\phi>0\}.
\end{align*}
Thus, the function
 $\tilde{u}(x'):=(\overline{w}-w)(x',0)$ solves the following problem in $\R^n$: 
 \begin{equation}\label{eq:nonlocal_reg}
 \begin{split}
 \tilde{u}\geq 0, \quad (-\Delta)^{s}\tilde{u}&\geq \psi \text{ in }\R^n,\\
 (-\Delta)^{s}\tilde{u}&=\psi\text{ in } \R^n\cap\{\tilde{u}(x')>0\},
 \end{split}
 \end{equation}
Here $\psi(x'):=c_{n,s}\lim_{x_{n+1}\rightarrow 0_+}x_{n+1}^{1-2s}\p_{n+1}w$. 
Setting $\va:=(-\D)^{-s} \psi$ and $u(x'):=\tilde{u}(x')-\va(x')$ then turns this into an obstacle problem (\ref{eq:obst_frac}) for the fractional Laplacian:
\begin{align*}
\min\{(-\D)^s u(x'), u(x')+\va(x')\}\geq 0, \quad  x'\in \R^{n}.
\end{align*}
In this (slightly restricted) sense the fractional thin obstacle problem (\ref{eq:thin_00}) and the thin obstacle problem for the fractional Laplacian (\ref{eq:obst_frac}) can be regarded as equivalent.\\

Motivated by the available regularity results for the obstacle problem for the fractional Laplacian (c.f. \cite{Si}, \cite{CaS}) and the described (slightly restricted) equivalence of the local and nonlocal problems (\ref{eq:thin_00}) and (\ref{eq:obst_frac}), it can be expected that solutions to (\ref{eq:thin_00}) enjoy analogous optimal regularity results as the ones described in (\ref{eq:opt_reg}). Indeed, using the (generalized) frequency function, the characterizations of global homogeneous solutions in two-dimensions (and a reduction to this following the argument of Remark \rmkeigen) and regularity estimates as in \cite{U87} allows us to prove this optimal regularity result by purely local means. As in the sequel we are however mainly interested in \emph{higher} regularity properties, we do not further elaborate on the details of this point, but will instead always assume some initial regularity (c.f. assumption (A2) in Section \ref{sec:setup}).

\subsubsection{Almgren frequency function and blow-ups} 
\label{sec:blowup}
In analyzing solutions to the fractional thin obstacle problem, a key tool in \cite{CSS} consists of a truncated frequency function: Assuming that
$w$ is a solution to \eqref{eq:fracLa} with $w(0)=0$, we reflect $w$ evenly about $x_{n+1}$ and set $\tilde{w}(x):=w(x)-\frac{\tilde{f}(0)}{2(2-2s)}x_{n+1}^2$. Setting
\begin{align*}
F_{w,0}(r):=\int_{\p B_r} |\tilde{w}(x)|^2 |x_{n+1}|^{1-2s} d\sigma,
\end{align*}
then allows us to define the modified frequency function at $0$
\begin{align*}
r\mapsto \Phi_{w,0}(r):=(r+C_0r^2)\frac{d}{dr}\log \max (F_w(r), r^{n+(1-2s)+4}).
\end{align*}
Its relevance stems from the fact that it is a monotone quantity.
Moreover, it satisfies the following dichotomy (Lemma 6.1 in \cite{CSS}):
$$\text{ Either } \Phi_{w,0}(0_+)=n+(1-2s)+2(1+s) \text{ or } \Phi_{w,0}(0_+)\geq n+(1-2s)+4.$$
In particular, this yields the decomposition into the regular free boundary $\Gamma_{1+s}(w)$ and the remaining free boundary (c.f. (\ref{eq:Gamma_reg})). Furthermore, for each $x_0\in \Gamma_{1+s}(w)$, there exists a blow-up sequence $w_{r_j,x_0}(x)=w(x_0+r_jx)/(r^{-(n+1-2s)}F_{w,x_0}(r_j))^{1/2}$, such that
\begin{align*}
w_{r_j,x_0}&\rightarrow w_{x_0} \text{ uniformly  in }B_{1/2}^+,\\
\nabla'w_{r_j,x_0}&\rightarrow \nabla' w_{x_0} \text{ uniformly  in }B_{1/2}^+,\\
x_{n+1}^{1-2s}\p_{n+1}w_{r_j,x_0}&\rightarrow x_{n+1}^{1-2s}\p_{n+1}w_{x_0} \text{ uniformly  in }B_{1/2}^+.
\end{align*}
Here $w_{x_0}(x):=c_{n,s}w_{1+s}(Qx)$, $Q\in SO(n+1)$ is a rotation matrix which might depend on the choice of the converging subsequence, $c_{n,s}$ is a normalization constant and $w_{1,s}$ is the function from (\ref{eq:w1s}). This in particular exemplifies the role of $w_{1,s}$ as a model solution: It is unique blow-up profile at the regular free boundary. It has a flat free boundary (c.f. Proposition 6.3 in \cite{CSS}). With this at hand, regularity of the regular free boundary was shown in \cite{CSS} by means of the comparison principle and a boundary Harnack inequality.\\

Building on the these results in \cite{CSS}, our main statement, Theorem \ref{thm:ana1}, translates into the analyticity (smoothness) of the regular free boundary of the obstacle problem for the fractional Laplacian:
\begin{thm}
\label{thm:ana2}
Let $u:\R^{n} \rightarrow \R$ be a solution of the obstacle problem for the fractional Laplacian (\ref{eq:obst_frac}) with obstacle $\va: \R^{n}\rightarrow \R$. Then if $\va$ is smooth, the regular free boundary $\Gamma_{1+s}(u)$ is locally smooth. If moreover $\va$ is real analytic, the regular free boundary $\Gamma_{1+s}(u)$ is locally real analytic.
\end{thm}

\subsubsection{Literature}
After the seminal articles of Silvestre \cite{Si} and of Caffarelli, Salsa and Silvestre \cite{CSS}, the thin obstacle problem has been studied by various authors with different focuses: For instance, Barrios, Figalli and Ros-Oton \cite{BFRO} study the regularity of the free boundary in the obstacle problem for the fractional Laplacian under the assumption that the obstacle $\va$ satisfies $\Delta\va\leq 0$ near the contact region. Petrosyan and Pop \cite{PP15} investigate the effects of the presence of drift terms on the the optimal regularity of solutions to the fractional obstacle problem. A further analysis of the free boundary regularity in this situation including drifts has been carried out in \cite{GPPSVG15}. Recently, fully nonlinear versions of the fractional obstacle have been addressed by Caffarelli, Ros-Oton and Serra \cite{CROS16}. \\

In spite of these activities to the best of our knowledge the higher regularity of the regular free boundary has not yet been addressed in the case of the fractional thin obstacle problem with general $s\in(0,1)$, but has up to now been restricted to the case $s=1/2$: In the case that $s=1/2$ the analyticity of the regular free boundary was proved by Koch, Petrosyan and Shi in \cite{KPS14} by relying on the Legendre-Hodograph transform. Simultaneously, but building on higher order boundary Harnack estimates, De Silva and Savin \cite{DSS14} showed the $C^{\infty}$ smoothness of the free boundary. Finally, in \cite{KRSIII} the higher regularity properties of the regular free boundary are studied depending on the (potentially low regularity) of the present variable coefficient metrics and inhomogeneities.

\subsection{Organization of the article}
The remainder of the article is organized as follows: After briefly summarizing and explaining our main assumptions and notations in Section \ref{sec:pre}, we deduce the asymptotic behavior of solutions (around the regular free boundary) in Section \ref{sec:asymp}. Here we argue in two steps and first construct barrier functions, prove a comparison result and a boundary Harnack inequality. Then we apply these tools to infer a leading order asymptotic expansion of solutions around the free boundary (Proposition \ref{prop:asymp1}) and a priori regularity estimates around the free boundary (Proposition \ref{prop:asymp2}). Building on these, in Section \ref{sec:HL} we then introduce the Hodograph-Legendre transform, show its invertibility (c.f. Proposition \ref{prop:invert}) and deduce the fully nonlinear equation which is satisfied by the Legendre function (c.f. Proposition \ref{prop:eqnonlin}). In Section \ref{sec:asymp1}, we translate the asymptotic behavior which was deduced in Section \ref{sec:asymp} in the original variables into the Legendre variables  (c.f. Propositions \ref{prop:asymp_v1}, \ref{prop:asymp_v2}). Motivated by the structure of the model solution in Legendre variables (c.f. Example \ref{ex:model}), we then define a suitable intrinsic geometry adapted to the nonlinear operator in Section \ref{sec:function_spaces}. Based on this, we introduce the function spaces which we are using to describe the mapping properties of the nonlinear equation and its linearization (c.f. Definition \ref{defi:XY}). With the aid of the new geometry we in particular conclude that the Legendre function lies in these function spaces (c.f. Corollary \ref{cor:reg_v}). Relying on this observation, in Section \ref{sec:mapping_prop} we discuss the mapping properties of the nonlinear and linearized operators (c.f. Propositions \ref{prop:analyticity}, \ref{prop:perturb}), which in Section \ref{sec:IFT} is used to invoke the implicit function theorem and to prove Theorem \ref{thm:ana1}.\\
Finally in two appendices, we discuss various auxiliary results. Here the appendices are structured such that in Appendix A, c.f. Section \ref{sec:AppA}, various regularity results are collected and proved which might be of independent interest (c.f. Propositions \ref{prop:Dirichlet}, \ref{prop:Neumann}, \ref{prop:approx_main} and the eigenfunction characterization in Section \ref{sec:Eigenf}). This in particular includes the explicit computation of the higher order eigenfunctions to the fractional Laplacian in the (flat) slit domain with mixed Dirichlet-Neumann data (c.f. Lemma \ref{lem:2Deigen} and Proposition \ref{prop:nD_mDN} in Section \ref{sec:Eigenf}). In Appendix B (c.f. Section \ref{sec:AppB}), we provide the proofs of various results which are used in the main body of the text (e.g. Propositions \ref{prop:decomp}, \ref{prop:Banach} and \ref{prop:map}), but which we decided to prove later, in order to clarify the structure of our main argument in Sections \ref{sec:asymp}-\ref{sec:mapping_prop}.

\section{Preliminaries}
\label{sec:pre}

\subsection{Set-up}
\label{sec:setup}
In this paper, we will study the higher regularity of the free boundary around regular free boundary points. We start with the following observation:

\begin{prop}
\label{prop:inhomo_0}
Let $\tilde{w}$ be a solution of \eqref{eq:fracLa} with $\tilde{f}\in C^{3,1}(B_1^+)$. Then, 
\begin{align*}
\bar{w}(x):=\tilde{w}(x)-\frac{1}{2(2-2s)}\tilde{f}(x',0)x_{n+1}^2-\frac{1}{3(3-2s)}\p_{n+1}\tilde{f}(x',0)x_{n+1}^3
\end{align*}
is a solution to the Signorini problem
\begin{equation*}
\begin{split}
\nabla \cdot x_{n+1}^{1-2s} \nabla \bar{w} = x_{n+1}^{3-2s} f &\mbox{ in } B_1^+,\\
\bar{w} \geq 0,\
 \lim\limits_{x_{n+1}\rightarrow 0_+} x_{n+1}^{1-2s}\p_{n+1}\bar{w}\leq 0,\ 
\tilde{w} \lim\limits_{x_{n+1}\rightarrow 0_+} x_{n+1}^{1-2s}\p_{n+1}\bar{w}=0 &\mbox{ on } B_1',
\end{split}
\end{equation*}
where $f(x)\in C^{0,1}(B_1^+)$ and 
\begin{align*}
f(x)&:=\left(\tilde{f}(x)-\tilde{f}(x',0)-\p_{n+1}\tilde{f}(x',0)x_{n+1}\right)x_{n+1}^{-2}\\
&-\frac{1}{2(2-2s)}\Delta' \tilde{f}(x',0)-\frac{1}{3(3-2s)}\Delta'\p_{n+1}f(x',0)x_{n+1}.
\end{align*}
In particular, the free boundary of $w$ remains unchanged, i.e. $\Gamma_{\bar{w}}=\Gamma_{\tilde{w}}.$
\end{prop}

\begin{proof}
The statement follows from a direct computation and a Taylor expansion of $\tilde{f}$ at $\{x_{n+1}=0\}$. 
\end{proof}

Compared to the problem (\ref{eq:fracLa}), the change from $\tilde{w}$ to $\bar{w}$ provides additional decay of the order $x_{n+1}^2$ for the inhomogeneity. This has the advantage that we can treat the cases $s\in(0,1/2]$ and $s\in(1/2,1)$ simultaneously in our analysis (c.f. Remark \ref{rmk:inhom}). In particular, we can work with the same function spaces (c.f. Section \ref{sec:function_spaces}) in both cases. As in this article we are primarily interested in smooth or analytic inhomogeneities, the loss in the derivatives which is involved in this reformulation does not pose any restrictions onto our framework. Since we are primarily interested in the regularity of the free boundary, and since $\Gamma_{\bar{w}}=\Gamma_{\tilde{w}}$, in the sequel we mainly consider \eqref{eq:fracLa_2} instead of \eqref{eq:fracLa}.\\

We recall that our equation enjoys the following scaling and multiplication symmetries:
\begin{lem}[Scaling and multiplication symmetries]
\label{lem:sym}
Let $w:B_1^+ \rightarrow \R$ be a solution to (\ref{eq:fracLa}) and consider constants $c>0$, $\lambda>0$ and a point $x_0\in B_{1}'$. Then in $B_{r}^+$ with $r\in(0,\lambda^{-1}(1-|x_0|))$ the function
\begin{align*}
x\mapsto w_{c,\lambda,x_0}(x):=c w(x_0 + \lambda x),
\end{align*}
is a solution of
\begin{align*}
\nabla \cdot x_{n+1}^{1-2s} \nabla  w_{c,\lambda,x_0} = x_{n+1}^{1-2s} f_{c,\lambda, x_0},
\end{align*}
with Signorini boundary conditions. Here $f_{c,\lambda,x_0}(x):= c \lambda^2 f(x_0 + \lambda x)$.
\end{lem}

\begin{proof}
This follows from a simple computation.
\end{proof}
 
Relying on these properties, throughout the paper we will assume that: 
\begin{itemize}
\item[(A1)] $w \in C^{2}_{loc}(B_1^+\cap\{x_{n+1}>0\})$ is a solution of the Signorini problem
\begin{equation}
\label{eq:fracLa_2}
\begin{split}
\nabla \cdot x_{n+1}^{1-2s} \nabla w &= x_{n+1}^{3-2s} f \mbox{ in } B_1^+,\\
w &\geq 0 \text{ on } B'_1,\\
 \lim\limits_{x_{n+1}\rightarrow 0_+} x_{n+1}^{1-2s}\p_{n+1}w&\leq 0 \mbox{ on } B_1',\\
\lim\limits_{x_{n+1}\rightarrow 0_+} x_{n+1}^{1-2s}\p_{n+1}w&=0 \mbox{ on } B_1'\cap\{\tilde{w}>0\},
\end{split}
\end{equation}
with $f:B_1^+ \rightarrow \R$ satisfying assumption (A4). The boundary conditions are attained in a pointwise sense (c.f. (A2)).
\item[(A2)] $w$ is sufficiently close to the blow-up limit $w_{1,s}$ in the sense that 
\begin{align*}
\| \nabla' w- \nabla' w_{1,s}\|_{C^0(B_1^+)}+\| x_{n+1}^{1-2s}\p_{n+1} w- x_{n+1}^{1-2s} \p_{n+1} w_{1,s}\|_{C^0(B_1^+)} \leq \epsilon_0 ,
\end{align*}
for some small $\epsilon_0>0$. Here $\nabla' $ denotes the gradient with respect to the tangential directions only and $w_{1,s}$ is defined in \eqref{eq:w1s}.
\item[(A3)] The free boundary $\Gamma_w$ in $B'_1$ only consists of regular free boundary points and is a $C^{1,\alpha}$ graph for some $\alpha\in (0,1)$, i.e. $$\Gamma_w \cap B'_1=\{(x'',g(x''),0): g\in C^{1,\alpha}(B''_1)\}.$$ Moreover, we assume that $g(0)=|\nabla''g(0)|=0$. 
\item[(A4)] The inhomogeneity $f$ is $C^{0,1}(B_1^+)$ regular and it satisfies $\|\tilde{f}\|_{C^{0,1}(B_1^+)}\leq \mu_0$ for a small, positive constant $\mu_0$. 
\end{itemize}

Let us comment on these assumptions. By Proposition~\ref{prop:inhomo_0}, assumption (A1) does not pose any restrictions, as we are interested in the regularity of the free boundary in the presence of smooth inhomogeneities $\tilde{f}$.
If the conditions of the equivalence of the local problem (\ref{eq:fracLa_2}) and the nonlocal problem (\ref{eq:obst_frac}) are satisfied (c.f. Section \ref{sec:back}), the assumptions (A2)-(A3) are consequences of the regularity results for $w$ and the regular free boundary from \cite{CSS}: Since our result is local, we can always assume these by using the scaling and multiplication symmetries from Lemma \ref{lem:sym} combined with the identification of the blow-up limits of solutions $w$ of (\ref{eq:fracLa}) at the regular free boundary (c.f. Section \ref{sec:blowup}). Finally, a further application of Lemma \ref{lem:sym} with a suitable rescaling allows us to always assume the smallness condition for $\tilde{f}$ from (A4).  \\

\begin{rmk}[Optimal regularity]
\label{rmk:opt}
We stress that we do not assume that $w$ has the optimal regularity $\p_i w\in C^{0,s}(B_1^+)$, $i\in \{1,\dots, n\}$ and $x_{n+1}^{1-2s}\p_{n+1}w\in C^{0,1-s}(B_1^+)$. We will see later (c.f. Proposition~\ref{prop:asymp1}) that this optimal regularity is a consequence of our assumptions (A2)-(A4). 
\end{rmk}

\begin{rmk}
We remark that by the boundary Harnack inequality (Theorem 7.7 in \cite{CSS}) we have that $\p_j  g(x'') = - \frac{\p_j w}{\p_n w}\big|_{(x'',g(x''),0)}$, $j\in\{1,\dots, n-1\}$ (where the right hand side is understood as a Hölder continuous extension up to the boundary). Therefore, in this situation we can always assume that $[\nabla'' g]_{C^{0,\alpha}(B_{1/2}^+)}$ is sufficiently small by choosing the constant $\epsilon_0$ from (A3) sufficiently small (by noting that the size of the Hölder norm is controlled by $\| \nabla' w- \nabla' w_{1,s}\|_{C^0(B_1^+)}$, c.f. for instance the proof of Theorem \thmtwo). 
\end{rmk}

\begin{rmk}
Sometimes we extend the solution $w$ and the inhomogeneity $\tilde{f}$ evenly about $x_{n+1}$. Here we use that by the complementary boundary conditions it holds that $\lim\limits_{x_{n+1}\rightarrow 0_+} x_{n+1}^{1-2s}\p_{n+1}w=0$ on $B'_1\setminus \Lambda_w$. Thus, after the extension, $w$ solves
\begin{align*}
\nabla \cdot |x_{n+1}|^{1-2s}\nabla w&= |x_{n+1}|^{3-2s} \tilde{f}\text{ in } B_1\setminus \Lambda_w,\\
w&=0 \text{ on } \Lambda_w.
\end{align*}
With a slight abuse of notation, we still use the symbol $L_s$ to refer to the evenly reflected fractional Laplacian, i.e. $L_s=\nabla \cdot |x_{n+1}|^{1-2s}\nabla$. 
\end{rmk}

\subsection{Notation}
\label{sec:not}
In the sequel we use the following notations: 
\begin{itemize}
\item $\R^{n+1}_+:=\{(x'',x_n,x_{n+1})\in \R^{n+1}: x_{n+1}\geq 0\}$,\\ $\R^n\times\{0\}:=\{(x'',x_n,x_{n+1})\in \R^{n+1}: x_{n+1}=0\}$.
\item Euclidean balls: 
\begin{align*}
B_r(x_0)&:=\{x\in \R^{n+1}: |x-x_0|\leq r\},\\
B^+_r(x_0)&:=B_r(x_0)\cap \R^{n+1}_+,\\
B'_r(x_0)&:=B_r(x_0)\cap (\R^n\times \{0\}).
\end{align*}
If $x_0$ is the origin, we also write $B_r$, $B^+_r$ and $B'_r$ for simplicity.
\item We use $C'_{\eta}(e_n)$ to denote the cone in $\R^{n}\times \{0\}$ with axis $e_n$ and opening angle $\eta$.
\item For $s\in (0,1)$, $L_s:=\nabla \cdot x_{n+1}^{1-2s}\nabla$ is the degenerate elliptic operator associated with the fractional Laplacian $(-\Delta)^s$. To abbreviate the associated weight function we introduce $\bar \omega(x):=x_{n+1}^{1-2s}$.
\item Weighted $L^2$ space: For a measurable set $\Omega\subset \R^{n+1}$, $L^2_{\bar \omega}(\Omega)=L^2(\Omega, x_{n+1}^{1-2s}dx)$ is the Banach space of measurable functions $u:\Omega\rightarrow \R$ such that 
\begin{align*}
\|u\|_{L^2_{\bar \omega}(\Omega)}:=\left( \int_{\Omega}|u(x)|^2 \bar \omega(x)dx\right)^{\frac{1}{2}}<\infty.
\end{align*}
Weighted Sobolev space: 
$H^1_{\bar \omega}(\Omega)=H^1(\Omega, x_{n+1}^{1-2s}dx)$ is the Banach space of functions $u\in L^2_{\bar\omega}(\Omega)$ whose distributional derivatives exist and $|\nabla u|\in L^2_{\bar \omega}(\Omega)$. We define the norm
$$\|u\|_{H^1_{\bar \omega}(\Omega)}:=\|u\|_{L^2_{\bar\omega}(\Omega)}+\|\nabla u\|_{L^2_{\bar\omega}(\Omega)}.$$
\item Given $u\in L^2_{\bar \omega}(\Omega)$, we denote the $L^2_{\bar\omega}$ average by 
\begin{align*}
\|u\|_{\LLw(\Omega)}:=\left(\frac{1}{\bar \omega(\Omega)}\int_{\Omega}|u(x)|^2\bar \omega(x)dx\right)^{\frac{1}{2}}, \quad \bar\omega(\Omega):=\int_{\Omega}\bar\omega(x)dx.
\end{align*}
\item Let $w$ be a solution to the thin obstacle problem (associated to $L_s$) in $B_1^+$. Then 
\begin{align*}
\Lambda_w&:=\{x\in B'_1: w(x)=0\} \quad (\text{contact set}),\\
\Omega_w&:=\{x\in B'_1: w(x)>0\}\quad (\text{positivity set}),\\
\Gamma_w&:=\p_{B'_1}\Lambda_w \quad (\text{free boundary}).
\end{align*}
\item Model solution: 
$$w_{1,s}(x):=\frac{1}{s^2-1}\left(\sqrt{x_n^2+x_{n+1}^2}+x_n\right)^s\left(-x_n+s\sqrt{x_n^2+x_{n+1}^2}\right)$$
is a model solution to the free boundary problem with flat free boundary $\Gamma_{w_{1,s}}=\{x_n=x_{n+1}=0\}$. 
We let 
\begin{align*}
w_{0,s}(x)=w_{0,s}(x_n,x_{n+1}):=\left(\sqrt{x_n^2+x_{n+1}^2}+x_n\right)^s.
\end{align*}
Note that for some, only $s$ dependent constant $c_s$
\begin{align*}
\p_n w_{1,s}(x)&= c_s w_{0,s}(x),\\ 
x_{n+1}^{1-2s}\p_{n+1}w_{1,s}(x)&=c_s \frac{s}{s-1}w_{0,1-s}(-x_n,x_{n+1}).
\end{align*}
\item We use $\epsilon_0>0$ to quantify the closeness of $w$ and the model solution $w_{1,s}$ in the $C^1(B_1^+)$ norm (c.f. assumption (A3)).
\end{itemize}
We usually use $x$ to denote the original coordinates and $y$ to denote the coordinates after the partial hodograph-Legendre transformation. In the following we collect the notation which we use after the change of coordinates:
\begin{itemize}
\item Quarter space:
$$Q_+:=\{(y'',y_n,y_{n+1})\in \R^{n+1}: y_n\geq 0, y_{n+1}\geq 0\}.$$
Edge of the quarter space:
$$P:=\{(y'',y_n,y_{n+1})\in \R^{n+1}:y_n=y_{n+1}=0\}.$$
\item Baouendi-Grushin metric $d_G(x,y)$ (c.f. Definition~\ref{defi:BGgeo}).
\item Baouendi-Grushin balls: 
\begin{align*}
\mathcal{B}_R(y_0)&:=\{y\in \R^{n+1}: d_G(y,y_0)\leq R\},\\
\mathcal{B}_R^+(y_0)&:=\mathcal{B}_R(y_0)\cap Q_+.
\end{align*}
If $y_0$ is the origin, we write $\mathcal{B}_R$ and $\mathcal{B}_R^+$ for simplicity.
\item Fractional Baouendi-Grushin operator: For $s\in (0,1)$
\begin{align*}
\D_{G,s}&:= (y_n y_{n+1})^{1-2s}(y_n^2 + y_{n+1}^2)\D''  + \p_n (y_n y_{n+1})^{1-2s}\p_n  \\
& \quad + \p_{n+1} (y_n y_{n+1})^{1-2s}\p_{n+1},
\end{align*}
where $\Delta''=\sum_{i=1}^{n-1}\p_{ii}$. \\
The associated weight function is also abbreviated as $\omega(y):=(y_ny_{n+1})^{1-2s}$.
\item Similarly, as above, we define the weighted Banach spaces $L^2_\omega(\Omega)=L^2(\Omega, \omega(y)dy)$ and $H^1_\omega(\Omega)=H^1(\Omega, \omega(y)dy)$. Given $u\in L^2_\omega(\Omega)$, we use $\|u\|_{\tilde{L}^2_\omega(\Omega)}$ to denote the $L^2_\omega$ average of $u$.
\item Function spaces: We use the global function spaces $X_{\alpha,\epsilon}, Y_{\alpha,\epsilon}$ and their local analogues $X_{\alpha,\epsilon}(\mathcal{B}_R^+), Y_{\alpha,\epsilon}(\mathcal{B}_R^+)$ (c.f. Definitions \ref{defi:XY}, \ref{defi:XYloc}).
\item  Given $u\in X_{\alpha,\epsilon}(\mathcal{B}_1^+)$ or $u\in X_{\alpha,\epsilon}$, we denote the $r-$neighborhood of $u$ in the corresponding Banach space by
$$\mathcal{U}_{r}(u):=\{v\in X_{\alpha,\epsilon}(\mathcal{B}_1^+):\|v-u\|_{X_{\alpha,\epsilon}(\mathcal{B}_1^+)}<r\}, \ 0<r<\infty .$$ 
\item Model solution $w_{1,s}$ in the Grushin coordinates: $$v_0(y) = - \frac{s}{2(1+s)}y_n^{2s+2} + y_n^{2s} y_{n+1}^2.$$
\item $F$ is the nonlinear function in \eqref{eq:nonlinear}. $L_{v}$ denotes the linearized operator of $F$ at $v$.
\end{itemize}

We also rely on the following convention:
\begin{itemize}
\item We denote the derivative with respect to the $x''$ (or $y''$) components of $x$ (or $y$) by $\nabla''$.
\item We use the Landau symbol $f(x)=O_s(g(x))$ as $x\rightarrow 0$ to denote that $\lim\limits_{x\rightarrow 0}\frac{f(x)}{g(x)} = C_s$, where the constant $C_s$ is allowed to depend on $s$.
\item Without specific notice a constant $C$ is assumed to be universal, i.e. it is assumed to only depend on the dimension $n$.
\end{itemize}

\section{Asymptotics}
\label{sec:asymp}

In this section we derive a leading order asymptotic expansion for solutions of the fractional thin obstacle problem around the regular free boundary (c.f. Proposition \ref{prop:asymp1}). Moreover, we prove regularity results for solutions to the fractional thin obstacle problem (c.f. Proposition \ref{prop:asymp2}). To achieve these objectives, we construct upper and lower barrier functions (c.f. Lemmas \ref{lem:barrier}, \ref{lem:upper_barrier}), which allow us to prove a non-degeneracy result on solutions by means of the comparison principle (c.f. Proposition \ref{prop:comparison}). Then a suitable boundary Harnack inequality (c.f. Proposition \ref{prop:boundHarn}) yields the desired asymptotic expansion around the free boundary.\\

This section is divided into two parts: In the first part (Section \ref{sec:tech_tools}), we provide the necessary technical tools (e.g. the construction of barrier functions, comparison results, a boundary Harnack inequality), which are then applied to the setting of the thin obstacle problem in the second part of the section (Section \ref{sec:application}). We use similar ideas as in \cite{KRSI}, where these technical tools are developed for the variable coefficient thin obstacle problem.

\subsection{Barrier functions, comparison results and the boundary Harnack inequality}
\label{sec:tech_tools}

We recall and provide some necessary tools of dealing with the fractional thin obstacle problem. As the results of this section are also of interest in a more general framework, we use the following conventions in this part of the section. 

\begin{assum}
\label{assum:tools}
In the sequel, we consider the slit domain $B_1 \setminus \Lambda$, where 
\begin{align*}
\Lambda:=\{(x',0): x_n\leq g(x'')\}, 
\end{align*}
for some $C^{1,\alpha}$ function $g$. Moreover, we define its boundary as
\begin{align*}
\Gamma:=\{(x',0):x_n=g(x'')\}.
\end{align*}
For convenience and normalization purposes, we assume that $g(0)=|\nabla'' g(0)|=0$.
\end{assum}
We also recall that $L_{s}:= \nabla\cdot |x_{n+1}|^{1-2s}\nabla$.\\

These assumptions are clearly motivated by the application of the following results to the fractional thin obstacle problem. 
In providing the tools which will later be applied to solutions of the fractional thin obstacle problem, we begin with the construction of a lower barrier function.

\begin{lem}[Lower barrier function]
\label{lem:barrier}
Let $s\in (0,1)$, $\alpha\in (0,1)$, $\tau\in \left(0,\min\{\frac{\alpha}{s}, \frac{1-s}{s}\} \right)$ and let $B_1\setminus \Lambda$ be as in Assumption \ref{assum:tools}. Then, if $[\nabla'' g]_{\dot{C}^{0,\alpha}}$ is sufficiently small depending on $n,s,\tau$, there exists a function $h\in C^{0,s}(B_1)$, $h(x)>0$ in $B_1\setminus \Lambda$ and $h(x)=0$ on $\Lambda$, such that: 
\begin{itemize}
\item[(i)]  $h$ is a subsolution to $L_s$ which satisfies 
\begin{align*}
L_s h(x)&\geq C_{n,s}\tau x_{n+1}^{1-2s} \dist(x,\Gamma)^{-2+s+s\tau} \mbox{ in }B_1 \setminus \Lambda.
\end{align*}
\item[(ii)] $h$ satisfies the non-degeneracy condition: 
$$h(x)\geq  c_n \dist(x,\Gamma)^{s} \left(\frac{\dist(x,\Lambda)}{\dist(x,\Gamma)}\right)^{2s}\text{ for }x\in B_{1}.$$
\item[(iii)] $h$ has the following leading order asymptotic expansion at $x_0\in \Gamma\cap B_{1/2}$: 
\begin{align*}
h(x)&=\left(\sqrt{((x-x_0)\cdot \nu_{x_0})^2+ x_{n+1}^2}+(x-x_0)\cdot \nu_{x_0}\right)^s \\
&+ [\nabla'' g]_{\dot{C}^{0,\alpha}}O_s\left((\sqrt{((x-x_0)\cdot \nu_{x_0})^2+ x_{n+1}^2}+(x-x_0)\cdot \nu_{x_0})^s|x-x_0|^{\alpha}\right).
\end{align*}
Here $\Gamma\ni x_0\mapsto \nu_{x_0}=\frac{(-\nabla'' g(x_0),1,0)}{\sqrt{1+|\nabla'' g(x_0)|^2}}$ is the in-plane, outer unit normal of $\Lambda$ at $x_0$. The symbol $\nabla''$ denotes the gradient with respect to the $x''$ components of $x=(x'',x_n,x_{n+1})$.
\end{itemize}
\end{lem}

\begin{proof}
We construct the desired barrier function by patching together suitably rotated profile functions. These profile functions are given by the derivative of the model solution to the fractional thin obstacle problem. By a slight convexification, it is possible to control the error terms that arise from the patching procedure.\\
 
Let 
$$w_{0,s}(x)= \left( \sqrt{x_n^2+x_{n+1}^{2}}+x_n\right)^s.$$
For $\tau\in (0,\frac{1-s}{s}]$, a direct computation shows that
\begin{equation}
\label{eq:barrier}
\begin{split}
L_s w_{0,s}^{1+\tau}
&=\tau (1+\tau)|x_{n+1}|^{1-2s}|\nabla w_{0,s}|^2 w_{0,s}^{\tau-1}\\
& = \tau(1+\tau)|x_{n+1}|^{1-2s} 2 s^2 (x_n^2+x_{n+1}^2)^{-\frac{1}{2}}w_{0,s}^{1-\frac{1}{s}+\tau}\\
& \geq s^2 \tau(1+\tau)|x_{n+1}|^{1-2s} (x_n^2+x_{n+1}^2)^{\frac{1}{2}(s+s\tau-2)} .
\end{split}
\end{equation}
Here we used that $1-\frac{1}{s}+\tau \leq 0$ by definition of $\tau$ and $w_{0,s}(x)\leq 2^s (x_n^2+x_{n+1}^2)^{s/2}$.\\

Let $\{Q_j\}_{j}$ be a Whitney decomposition of $B_1\setminus \Gamma$ and let $\{\eta_j\}_j$ be a partition of unity associated to $\{Q_j\}$ such that $\eta_k$ satisfies $\p_{n+1}\eta_k=0$ on $\{x_{n+1}=0\}$. Let $\hat x_j$ be the center of the Whitney cube $Q_j$ and let $x_j\in \Gamma$ realize the distance of $\hat x_j$ to $\Gamma$. Let $r_j= \diam(Q_j)$, which (by definition of a Whitney decomposition) is equivalent to $\dist(Q_j,\Gamma)$. Let $\nu_j$ be the (in-plane) outer unit normal to $\Lambda$ at $x_j$ and set
$$ w_k(x) := w_{0,s}((x-x_k)\cdot \nu_k, x_{n+1}).$$
Furthermore, define
$$h_\tau(x):=\sum_j \eta_k(x)( w_k(x))^{1+\tau}. $$
Then
\begin{align*}
L_s h_\tau&=\sum_k (L_s \eta_k ) w_k^{1+\tau} + 2(1+\tau)\sum_k |x_{n+1}|^{1-2s}(\nabla \eta_k\cdot \nabla w_k) w_k^{\tau}\\
&\quad + \sum_k \eta_k (L_s w_k^{1+\tau}).
\end{align*}
We estimate the above three sums separately. Firstly, by using $\sum_k\eta_k=1$, we observe that $\sum_k L_s \eta_k=0$. Thus, for any $x\in Q_\ell$ with $\ell$ fixed,
\begin{align*}
\sum_k (L_s\eta_k(x))w_k^{1+\tau}(x)=\sum_k (L_s \eta_k)(w_k^{1+\tau}(x)-w_\ell^{1+\tau}(x)).
\end{align*}
By the assumption that $\p_{n+1} \eta_k = 0$ on $\{x_{n+1}=0\}$ and by the regularity of $\eta_k$, we further conclude that $|\p_{n+1} \eta_k(x)| \leq C |x_{n+1}|$. Combining this with the fact that 
$$|\nu_k-\nu_\ell|\leq C [\nabla'' g]_{\dot{C}^{0,\alpha}}r_\ell^{\alpha}, \quad Q_k\subset \mathcal{N}(Q_\ell),$$
yields
\begin{align*}
&\sum_k (L_s \eta_k) w_k^{1+\tau}\leq C_{n,s}[\nabla''  g]_{\dot{C}^{0,\alpha}} |x_{n+1}|^{1-2s} r_\ell^{-2+s(1+\tau)+\alpha} .
\end{align*}
Similarly, the second sum can be estimated by
\begin{align*}
&2(1+\tau)\sum_k |x_{n+1}|^{1-2s}(\nabla \eta_k\cdot \nabla w_k) w_k^{\tau}\leq C_{n,s}[\nabla''  g]_{\dot{C}^{0,\alpha}}
|x_{n+1}|^{1-2s} r_\ell^{-2+s(1+\tau)+\alpha} .
\end{align*}
Using \eqref{eq:barrier}, the last sum can be bounded from below by
\begin{align*}
&\sum_k \eta_k (x)(L_s w_k^{1+\tau}(x))\geq Cs^2\tau(1+\tau)|x_{n+1}|^{1-2s} r_\ell^{-2+s(1+\tau)} ,
\end{align*}
for $ x\in Q_\ell$.
Combining all these observations, leads to
\begin{align}\label{eq:htau}
&L_s h_\tau(x) \geq C_{s,n} \tau |x_{n+1}|^{1-2s} r_\ell^{-2+s(1+\tau)} ,
\end{align}
for a fixed $\tau\in (0,\frac{1-s}{s}]$ and $s\in (0,1)$, if $[\nabla'' g]_{\dot{C}^{0,\alpha}}$ is sufficiently small depending on $\tau, s$ and $n$. 
Thus, setting
\begin{align*}
h(x)=h_0(x)+h_\tau(x)=\sum_k \eta_k (w_k(x)+ w_k(x)^{1+\tau}),
\end{align*}
for fixed $\tau\in (0,\min\{\frac{\alpha}{s},\frac{1-s}{s}\})$, yields a function which satisfies $h(x)\geq 0$. Moreover, by similar considerations as above (with $\tau=0$) we have
\begin{align}\label{eq:h0}
L_s h_0(x)\leq C_{n,s}[\nabla'' g]_{\dot{C}^{0,\alpha}}|x_{n+1}|^{1-2s} r_\ell^{-2+s+\alpha}.
\end{align}
Combining this with \eqref{eq:htau} gives
\begin{align*}
L_s h(x)\geq C_{n,s}\tau x_{n+1}^{1-2s} r_\ell^{-2+s(1+\tau)} .
\end{align*} 
Here we have used that $\tau < \alpha/s$ (as there is no gain of the form $\tau s$ in the lower bounds for patching errors originating from $h_0$) and we have chosen $[\nabla'' g]_{\dot{C}^{0,\alpha}}$ sufficiently small depending on $n,\alpha,s$.
Finally, $h(x)$ satisfies the non-degeneracy condition 
\begin{align*}
h(x)\geq c_n \dist(x,\Gamma)^{s}  \left(\frac{\dist(x,\Lambda)}{\dist(x,\Gamma)}\right)^{2s}. 
\end{align*}
This concludes the proof.
\end{proof}

Using a similar proof, we can also construct an upper barrier function:

\begin{lem}[Upper barrier function]
\label{lem:upper_barrier}
Let $s\in (0,1)$, $\alpha\in (0,1)$, $\tau\in \left(0,\min\{\frac{\alpha}{s}, \frac{1-s}{s}\} \right)$ and let $B_1\setminus \Lambda$ be as in Assumption \ref{assum:tools}. Then, if $[\nabla'' g]_{\dot{C}^{0,\alpha}}$ is sufficiently small depending on $n,s,\tau$, there exists a function $\hat h\in C^{0,s}(B_1)$, $\hat h(x)>0$ in $B_1\setminus \Lambda$ and $\hat h(x)=0$ on $\Lambda$, such that: 
\begin{itemize}
\item[(i)]  $\hat h$ is a supersolution to $L_s$ with
\begin{align*}
L_s \hat h(x)&\leq -C_{n,s}\tau x_{n+1}^{1-2s} \dist(x,\Gamma)^{-2+s+s\tau} \mbox{ in } B_1 \setminus \Lambda.
\end{align*}
\item[(ii)] $\hat h$ satisfies 
$$0\leq \hat h(x)\leq \dist(x,\Gamma)^{s} \left(\frac{\dist(x,\Lambda)}{\dist(x,\Gamma)}\right)^{2s}\text{ for }x\in B_{1}.$$
\end{itemize}
\end{lem}

\begin{proof}
Let $\hat h(x)=h_0(x)-h_\tau(x)$, where $h_0$ and $h_\tau$ are the same functions as in the proof of Lemma~\ref{lem:barrier}. The claims of the lemma follow analogously as in the proof of Lemma~\ref{lem:barrier}.
\end{proof}

With the lower barrier function at hand, we can proceed to the following comparison principle.

\begin{prop}[Comparison principle]
\label{prop:comparison}
Let $s\in(0,1)$ and let $B_1\setminus \Lambda$ be as in Assumption \ref{assum:tools}. Suppose that $u\in C(B_1)\cap H^{1}(B_1,|x_{n+1}|^{1-2s}dx)$ solves
\begin{align*}
L_s u=|x_{n+1}|^{1-2s}f \text{ in } B_1\setminus \Lambda, \quad u=0 \text{ on } \Lambda,
\end{align*}
where for some $s_0> 0$ and $\delta_0>0$ the function $f$ satisfies,
\begin{align*}
\left\| \dist(\cdot,\Gamma)^{2-s-s_0} f\right\|_{L^\infty(B_1\setminus \Lambda)}\leq \delta_0.
\end{align*}
Moreover, suppose that $u$ satisfies the following non-degeneracy conditions
\begin{align*}
u(x)&\geq 1 \text{ on } B_1\cap \left\{|x_{n+1}|\geq \ell=\sqrt{\frac{1-s}{2(n+1)}}\right\},\\
u(x)&\geq -2^{-8} \text{ on } B'_1\times (-\ell,\ell).
\end{align*}
Then, if $\delta_0=\delta_0(n,s,s_0,\alpha)$ is sufficiently small, there exists  a constant $c_n>0$ such that
\begin{align*}
u(x)\geq c_n\dist(x,\Gamma)^s\left(\frac{\dist(x,\Lambda)}{\dist(x,\Gamma)}\right)^{2s},\quad x\in B'_{1/2}\times (-\ell,\ell).
\end{align*}
\end{prop}

\begin{proof}
The proof follows from the construction of a suitable comparison function (which relies on our barrier function from Lemma \ref{lem:barrier}) and the comparison principle. \\
For $x_0\in B'_{1/2}\times (-\ell,\ell)$,
let $$P(x)=|x'-x'_0|^2-\frac{n+1}{2-2s}x_{n+1}^2.$$ Note that $L_s P(x)=0$. Let $h(x)$ be the barrier function constructed in Lemma~\ref{lem:barrier} with $\tau=\tau(s,s_0, \alpha)$ satisfying the condition $\tau\in \left(0,\min\{\frac{\alpha}{s},\frac{1-s}{s}\} \right)$ from Lemma \ref{lem:barrier} and chosen such that $s\tau \leq s_0$ (e.g. it would be possible to set $\tau:= \frac{1}{2}\min\{\frac{\alpha}{s},\frac{1-s}{s}, \frac{s_0}{s}\}$). We define our comparison function to be
\begin{align*}
\bar u(x)=u(x)+P(x)-2^{-8}h(x).
\end{align*}
Using the non-degeneracy conditions from the assumptions of the proposition, it follows that
\begin{align*}
\bar u&\geq \frac{1}{2} \text{ on } \{|x_{n+1}|\geq \ell\},\\
\bar u &\geq 0 \text{ on } \p B'_1\times (-\ell,\ell),\\
\bar u&\geq 0 \text{ on } \Lambda.
\end{align*}
Moreover, for $\delta_0=\delta_0(n,s,s_0, \alpha)>0$ sufficiently small,
\begin{align*}
L_s \bar u&= |x_{n+1}|^{1-2s}f - 2^{-8}L_s h\\
&\leq \delta_0 |x_{n+1}|^{1-2s}\dist(x,\Gamma)^{-2+s+s_0}   - 2^{-8}C_{n,s}\tau |x_{n+1}|^{1-2s}\dist(x,\Gamma)^{-2+s+s\tau} \\
&\leq 0 \text{ in } B_1\setminus \Lambda.
\end{align*}
Thus, by the comparison principle, $\bar u(x_0)\geq 0$. This implies that
\begin{align*}
u(x_0)\geq 2^{-8}h(x_0)\geq 2^{-8}c_n\dist(x,\Gamma)^s\left(\frac{\dist(x,\Lambda)}{\dist(x,\Gamma)}\right)^{2s}.
\end{align*}
Since $x_0$ is an arbitrary point in $B'_{1/2}\times (-\ell,\ell)$, we infer the desired lower bound for $u$.
\end{proof}

Combining the comparison principle from Proposition \ref{prop:comparison} with the resulting non-degeneracy property gives the following boundary Harnack inequality.  

\begin{prop}[Boundary Harnack]
\label{prop:boundHarn}
Let $\Lambda, \Gamma$ be as in Assumption \ref{assum:tools}.
Suppose that $u_1,u_2\in C(B_1)\cap H^{1}(B_1,|x_{n+1}|^{1-2s}dx)$ are positive in $B_1\setminus \Lambda$, even in the $x_{n+1}$-variable and that they solve
\begin{align*}
L_s u_1 &= |x_{n+1}|^{1-2s} f_1 \mbox{ in }  B_{1}\setminus \Lambda,\quad  u_1=0 \text{ on } \Lambda,\\
L_s u_2&=|x_{n+1}|^{1-2s} f_2 \mbox{ in }  B_{1}\setminus \Lambda, \quad  u_2=0 \text{ on } \Lambda,
\end{align*}
where the inhomogeneities $f_i$, $i=1,2$, satisfy the following bound: For some $s_0>0$ 
\begin{align*}
\left\| \dist(\cdot,\Gamma)^{2-s-s_0} f_i \right\|_{L^\infty(B_1\setminus \Lambda)}\leq \delta_0.
\end{align*}
Then, if $[\nabla'' g]_{\dot{C}^{0,\alpha}}$ and $\delta_0$ are sufficiently small depending on $n,s, s_0$ and $\alpha$, there exists a constant $C_0=C_0(n,s)>0$, such that 
\begin{align}
\label{eq:bdryH}
C_0\frac{u_2(\frac{1}{2}e_{n+1})}{u_1(\frac{1}{2} e_{n+1})}\leq \frac{u_2(x)}{u_1(x)} \leq C_0^{-1}\frac{u_2(\frac{1}{2}e_{n+1})}{u_1(\frac{1}{2} e_{n+1})} \text{ in } B_{1/2}\setminus \Lambda.
\end{align}
Moreover, $u_2/u_1$ extends to a $C^{0,\beta}$ function in $B_{1/4}$ for some $\beta\in(0,1)$.
More precisely, there exist constants $\beta=\beta(n,s,s_0)\in (0,1)$ and $C=C(n,s)>0$, 
such that for all $x_0\in \Lambda\cap B'_{1/4}$ 
\begin{align}\label{eq:bdryH2}
\left|\frac{u_2}{u_1}(x) - \frac{u_2}{u_1}(x_0) \right| \leq C \frac{u_2(\frac{1}{2}e_{n+1})}{u_1(\frac{1}{2}e_{n+1})} |x-x_0|^{\beta},\quad x \in B_{1/4}(x_0).
\end{align}
\end{prop}

\begin{proof}[Proof of Proposition~\ref{prop:boundHarn}]
\emph{Step 1: Proof of (\ref{eq:bdryH}).}
The inequality \eqref{eq:bdryH} is a consequence of the comparison principle: Without loss of generality we assume that $u_{1}(e_{n+1})= u_{2}(e_{n+1})=1$. By the Harnack inequality, for any $B_{r}(\bar x_0)\Subset B_1\setminus \Lambda$, there exists a constant $C=C(n,s)>0$ such that
\begin{align*}
\sup_{B_{r/2}(\bar x_0)} u_i \leq C \inf\limits_{B_{r/2}(\bar x_0)} u_i + C r^2 \sup\limits_{B_r(\bar x_0)}|f_i|, \quad i=1,2.
\end{align*}
Hence, if $\delta_0=\delta_0(n,s,s_0)$ is sufficiently small, there exist constants $\tilde{c}, \tilde{C}>0$ depending on $n,s$ such that $\tilde{c}\leq u_i(x)\leq \tilde{C}$ in $\{x\in B_{3/4}: |x_{n+1}|\geq \ell=\sqrt{\frac{1-s}{2(n+1)}}\}$  (note that $|f_i(x)|\leq \delta_0\ell^{s+s_0-2}$ if $|x_{n+1}|\geq \ell$). Thus, by a comparison argument, using the upper/lower barrier function (c.f. Lemma \ref{lem:upper_barrier} and Proposition~\ref{prop:comparison}), we have that for all $x\in B_{1/2}$
\begin{align*}
u_1(x) &\leq \tilde{C} \dist(x,\Gamma)^s\left(\frac{\dist(x,\Lambda)}{\dist(x,\Gamma)}\right)^{2s},\\
u_2(x) & \geq \tilde{c} \dist(x,\Gamma)^s\left(\frac{\dist(x,\Lambda)}{\dist(x,\Gamma)}\right)^{2s},
\end{align*}
if $\delta_0=\delta_0(n,s,s_0,\alpha)$ is sufficiently small. Here the constants $\tilde{C}$ and $\tilde{c}$ might be different from the equally denoted ones from above. They however also only depend on $n,s$.
Hence, there exists a constant $C_0=C_0(n,s)\in (0,1)$ with
\begin{align*}
u_2(x)-C_0 u_1(x) \geq 0.
\end{align*}
As the roles of $u_1,u_2$ can be reversed, this results in (\ref{eq:bdryH}). \\

\emph{Step 2: Proof of the Hölder continuity.}
The proof of the Hölder continuity of the quotient follows from a scaling argument. Since the proof is very similar to the one given in Lemma 3.24 in \cite{KRSI}, we only give a short outline here: Without loss of generality we assume that $x_0=0$. As in \cite{KRSI} we prove that there exist sequences of constants $\{a_k\}_{k\in\N},\{b_k\}_{k\in\N}$ such that
\begin{itemize}
\item[(i)] it holds $C_0\leq a_k \leq 1 \leq b_k \leq C_0^{-1}$ and $b_k-a_k \leq C \mu_1^k$ with $\mu_1 \in (0,1)$,
\item[(ii)] $b_k-a_k \geq C \mu_2^k$ with $\mu_2 \in (0,\mu_1]$ and $C>1$ being an absolute constant,
\item[(iii)] $a_k u_1 \leq u_2 \leq b_k u_1$ in $B_{2^{-k}}$.
\end{itemize}
The sequences are constructed inductively by a scaling argument and an application of Step 1: We set
\begin{align*}
\tilde{w}_1(x):= \frac{u_2(2^{-k}x)-a_k u_1(2^{-k}x)}{b_k-a_k}, \ \tilde{w}_2(x):= \frac{b_k u_1(2^{-k}x)-u_2(2^{-k}x)}{b_k-a_k}.
\end{align*}
As these functions are a convex combination of $u_2(2^{-k}x)$, we may without loss of generality assume that 
\begin{align*}
\tilde{w}_1\left(\frac{e_{n+1}}{2}\right) \geq \frac{1}{2}u_1\left( \frac{2^{-k} e_{n+1}}{2}\right).
\end{align*}
We rescale this and define
\begin{align*}
w_1(x) := \frac{\tilde{w}_1(x)}{u_1\left( \frac{2^{-k}e_{n+1}}{2} \right)}, \ \bar{u}(x):= \frac{u_1(2^{-k}x)}{u_1\left( \frac{2^{-k}e_{n+1}}{2} \right)}.
\end{align*}
This in particular implies that $2 w_1(e_{n+1}/2)\geq 1$. In order to prove the existence of the sequence $a_k, b_k$ with the desired properties (i)-(iii), we seek to apply Step 1 to the functions $w_1, \bar{u}$. To this end, we have to check the size assumption on the respectively associated inhomogeneities and the non-degeneracy condition. We only provide the proof for the scaling argument: We have
\begin{align*}
L_s w_1 = \frac{2^{-2k}}{(b_k-a_k) u_1\left( \frac{2^{-k} e_{n+1}}{2} \right)} x_{n+1}^{1-2s} f_1|_{2^{-k}x}=: x_{n+1}^{1-2s}\tilde{f}_1(x).
\end{align*}
Setting $\Gamma_{2^{-k}}:= \{x\in B_1: 2^{-k}x\in \Gamma\}$ and recalling that $\dist(x,\Gamma_{2^{-k}})= 2^{k}\dist(2^{-k}x,\Gamma)$, thus yields
\begin{align*}
&\|\dist(\cdot,\Gamma_{2^{-k}})^{2-s-s_0}\tilde{f}_1\|_{L^{\infty}(B_1)} \\
&= \frac{2^{-2k}}{(b_k-a_k)u_1\left( \frac{2^{-k}e_{n+1}}{2} \right)}2^{(2-s-s_0)k}\|\dist(\cdot,\Gamma)^{2-s-s_0}f_1 \|_{L^{\infty}(B_{2^{-k}})} \\
&\leq \frac{2^{(-s-s_0)k}}{C \mu_2^k 2^{-ks}} \delta_0 \leq C^{-1} 2^{-s_0 k} \mu_2^{-k} \delta_0.
\end{align*} 
Thus, if $\mu_2 \geq C^{-1}4^{-1} 2^{-s_0}$, Step 1 is applicable and it results in
\begin{align*}
\frac{2 w_1}{\bar{u}} \geq C_0.
\end{align*}
Spelling this out, leads to
\begin{align*}
(a_k + \frac{C_0}{2}(b_k-a_k))u_1 \leq u_2 \leq b_k u_1 \mbox{ in } B_{2^{-k-1}}.
\end{align*}
Setting $a_{k+1}:=a_k + \frac{C_0}{2}(b_k-a_k)$ and $b_{k+1}:=b_k$ therefore implies that $b_{k+1}-a_{k+1} = \left( 1- \frac{C_0}{2}\right)(b_k-a_k)$, which inductively yields $b_{k+1}-a_{k+1}=\bar{C}\left( 1- \frac{C_0}{2}\right)^k$, where $\bar C=C_0^{-1}-C_0$. This leads to a Hölder exponent $\beta$ which is less than or equal to $\min\{s_0,1, |\log_2\left( 1- \frac{C_0}{2}\right)|\}$ and a constant $C$ in \eqref{eq:bdryH2} with $C\lesssim C_0^{-1}-C_0$.
\end{proof}

\subsection{Application to the fractional thin obstacle problem}
\label{sec:application}
With the results of Section \ref{sec:tech_tools} at hand, we now turn to the fractional thin obstacle problem. In this context and in the whole of the remaining text, we assume that the conditions (A1)-(A4) from Section \ref{sec:setup} hold.
In particular, we note that $\Lambda_w$ and $\Gamma_w$ satisfy the requirements of Assumption \ref{assum:tools}. Thus, applying the results from the previous section allows us to infer a leading order asymptotic expansion (c.f. Proposition \ref{prop:asymp1}) and regularity results (c.f. Proposition \ref{prop:asymp2}).\\

We begin with a particularly relevant consequence of the boundary Harnack inequality and derive the following leading order asymptotic expansion of solutions to \eqref{eq:fracLa_2} with inhomogeneities which satisfy the condition (A4). We remark that by Proposition~\ref{prop:inhomo_0}, this asymptotic expansion transfers to the solution of \eqref{eq:fracLa}.

\begin{prop}
\label{prop:asymp1}
Let $w:B_1^+ \rightarrow \R$ be a solution to \eqref{eq:fracLa_2} and assume that the conditions (A1)-(A4) hold. Then, if $\|f\|_{C^{0,1}(B_1^+)}$ and $[\nabla'' g]_{\dot{C}^{0,\alpha}}$ are sufficiently small depending on $n,s,\alpha$, there exist a constant $\beta \in (0,1-s)$, a function $c:\Gamma_w \rightarrow \R$ with $\Gamma_w \ni x_0\mapsto c(x_0)$ being $C^{0,\beta}$ regular, such that at each $x_0\in \Gamma_w$
\begin{itemize}
\item[(i)]
\begin{align*}
w(x)&= c(x_0)w_{1,s}(Q_{x_0}(x),x_{n+1})+O_s\left(|x-x_0|^{1+s+ \beta}\right), 
\end{align*}
where $x\in B_{1/4}(x_0)$, $Q_{x_0}(x)=(x-x_0)\cdot \nu_{x_0}$ is an affine transformation at $x_0$, and $\nu_{x_0}$ is the in-plane outer unit normal of $\Lambda_w$ at $x_0$.
\item[(ii)] For $i\in\{1,\dots, n\}$, $x\in B_{1/4}(x_0)$ and $Q_{x_0}(x)$ as in (i),
\begin{align*}
\partial_i w(x) &= c(x_0)(e_i\cdot \nu_{x_0})w_{0,s}(Q_{x_0}(x),x_{n+1})\\
& \quad +O_s\left(w_{0,s}(Q_{x_0}(x),x_{n+1})|x-x_0|^\beta\right).
\end{align*}
\item[(iii)] For $x\in B_{1/4}(x_0)$ and $Q_{x_0}(x)$ as in (i),
\begin{align*}
x_{n+1}^{1-2s}\partial_{n+1}w(x)&= c(x_0)\frac{s}{s-1}w_{0,1-s}(-Q_{x_0}(x),x_{n+1})\\ 
& \quad +O_s\left(w_{0,1-s}(-Q_{x_0}(x),x_{n+1})|x-x_0|^{\beta}\right).
\end{align*}
\end{itemize}
Here the notation $O_s$ in (ii) means that there exists positive $C=C(n,s)$ universal, such that 
\begin{align*}
|\p_iw(x)-c(x_0)(e_i\cdot \nu_{x_0})w_{0,s}(Q_{x_0}(x),x_{n+1})|\leq C(n,s)|x-x_0|^\beta w_{0,s}(Q_{x_0}(x),x_{n+1})
\end{align*}
holds for any $x_0\in \Gamma_w$ and $x\in B_{1/4}(x_0)$. The same applies to the use of $O_s$ in (i) and (iii).
\end{prop}

\begin{rmk}
Proposition~\ref{prop:inhomo_0} implies that the asymptotic expansion from Proposition~\ref{prop:asymp1} also yields an analogous asymptotic expansion  around the regular free boundary of a solution $w$ to \eqref{eq:fracLa} with inhomogeneity $\tilde{f}\in C^{3,1}(B_1^+)$ (i.e. without the modifications of Proposition \ref{prop:inhomo_0}). We also note that the asymptotic expansion for tangential derivatives $\p_i w$, $i\in \{1,\dots, n\}$, only requires that $\tilde{f}\in C^{0,1}(B_1)$. The higher regularity assumption on $\tilde{f}$ is only needed for the asymptotic expansion of $\p_{n+1}w$ (and hence of $w$).
\end{rmk}

\begin{proof}
We first show property (ii). This follows by applying the boundary Harnack inequality (Proposition~\ref{prop:boundHarn}) to $\p_e w$ and $h_0$, where $e\in \mathcal{C}'_\eta(e_n)$ and $h_0(x)$ is the barrier function from Lemma~\ref{lem:barrier} with $\tau=0$. Here $\mathcal{C}'_\eta(e_n)$ denotes a cone with opening angle $\eta$ in $\R^{n}\times\{0\}$.  More precisely, by the proof of Lemma~\ref{lem:barrier} (c.f. \eqref{eq:h0}) we have that on the one hand
\begin{align*}
L_s h_0=|x_{n+1}|^{1-2s}k(x) \text{ in } B_1\setminus\Lambda_w,
\end{align*}
where the function $k$ satisfies
\begin{align*}
\left\| \dist(\cdot,\Gamma_w)^{2-s-\alpha} k\right\|_{L^\infty(B_1\setminus \Lambda_w)}\leq C_{n,s} [\nabla'' g]_{\dot{C}^{0,\alpha}}.
\end{align*}
Moreover, $h_0$ satisfies the non-degeneracy conditions and the asymptotics stated in Lemma~\ref{lem:barrier} (ii), (iii). Thus, $h_0$ satisfies the assumptions of Proposition~\ref{prop:boundHarn}. On the other hand, by the assumption (A2) and Proposition~\ref{prop:comparison}, if $\epsilon_0$ is sufficiently small, then $\p_ew$ with $e\in \mathcal{C}'_\eta(e_n)$ is positive in $B_1 \setminus \Lambda_w$. Moreover, it solves $L_s (\p_ew)=|x_{n+1}|^{1-2s}\p_ef$, where by (A4) $\|\p_ef\|_{L^\infty}\leq \mu_0$. Thus, by Proposition~\ref{prop:boundHarn}, if $\mu_0=\mu_0(n,s,\alpha)$ and $[\nabla'' g]_{\dot{C}^{0,\alpha}}$ are sufficiently small,
the quotient $\p_e w/h_0$ is $C^{0,\beta}$ regular up to $\Lambda_w$ for some $\beta=\beta(n,s)$. More precisely, there exists a $C^{0,\beta}$ function $b_e:\Lambda_w\rightarrow \R$ and a constant $C=C(n,s)$ such that for each $x_0\in \Lambda_w$,
\begin{align*}
\left|\frac{\p_ew(x)}{h_0(x)}-b_e(x_0)\right|\leq C |x-x_0|^{\beta}.
\end{align*}
Multiplying by $h_0$ on both sides of the estimate, we obtain that
\begin{align*}
\p_ew(x)=b_e(x_0)h_0(x)+O_s(h_0(x)|x-x_0|^{\beta}).
\end{align*}
Using the asymptotics of $h_0$ (c.f. (iii) in Lemma~\ref{lem:barrier}) and restricting $b$ to $\Gamma_w$, we therefore deduce that around $x_0\in \Gamma_{w}$
\begin{align*}
\p_ew(x)=b_e(x_0)w_{0,s}(Q_{x_0}(x),x_{n+1})+O_s(w_{0,s}(Q_{x_0}(x),x_{n+1})|x-x_0|^\beta).
\end{align*}
For $\tau\notin \mathcal{C}'_\eta(e_n)$, we express $\tau=c_1e_n+c_2e$ for some $e \in \mathcal{C}'_\eta(e_n)$ and write $\p_\tau w=c_1\p_nw+c_2\p_e w$.\\
In order to obtain the leading order asymptotic expansion for the (weighted) normal derivative $x_{n+1}^{1-2s}\p_{n+1}w$, we note that it satisfies the conjugate equation with respect to $w$ (c.f. \cite{CaS}): More precisely, let $\bar{w}(x):=|x_{n+1}|^{1-2s}\p_{n+1}w $. We reflect $w$ and $f$ oddly about $x_{n+1}$ (thus $\bar w$ is even about $x_{n+1}$). From the equation we have $|x_{n+1}|^{2s-1}\p_{n+1}\bar w=0$ on $\Lambda_w$. Thus, $\bar w$ solves 
\begin{align*}
L_s \bar{w} = \p_{n+1} \bar{f} \mbox{ in } B_{1}\setminus \overline{\Omega_w},\quad
\bar{w}=0 \mbox{ on } \overline{\Omega_w},
\end{align*}
where 
\begin{align*}
\bar{f}(x):=\left\{  
\begin{array}{ll}
x_{n+1}^2f(x, x_{n+1}) &\mbox{ for } x_{n+1}\geq 0,\\
-x_{n+1}^2f(x, -x_{n+1}) & \mbox{ for } x_{n+1}<0.
\end{array}
\right.
\end{align*} 
We note that the inhomogeneity $\p_{n+1}\bar f$ is of the form $|x_{n+1}|h(x)$ with $h(x)\in L^\infty$. Moreover, by assumption (A4) the smallness condition for $f$ implies a smallness condition for $h$. As a consequence, we may apply the comparison result of Proposition \ref{prop:comparison} to $\bar{w}$ with $s$ being replaced by $1-s$.  This concludes the proof on the asymptotic expansion for $\bar{w}$.\\
To obtain the asymptotic expansion of $w$ at $x_0\in \Gamma_w$ which is claimed in (i), we use an argument relying on path integration as in Corollary 4.8 in \cite{KRSI}. We obtain that $w(x)=c(x_0)w_{1,s}(Q_{x_0}(x),x_{n+1})+O_s(|x-x_0|^{1+s+\beta})$, where $c(x_0)=\frac{s-1}{s}b_{n+1}(x_0)$. Thus, $c(x_0)>0$ for any $x_0\in \Gamma_w$ and $c\in C^{0,\beta}(\Gamma_w)$. \\
In the end, we express $b_e(x_0)$ in terms of $c(x_0)$ and $\nu_{x_0}$, and infer that $b_e(x_0)=c(x_0)(e\cdot \nu_{x_0})$ for any $e\in S^n\cap \{e_{n+1}=0\}$. This completes the proof of property (i).
\end{proof}

For convenience of notation we introduce the following convention:
\begin{conv}
Let $\alpha\in(0,1)$ be the exponent from assumption (A3) and let $\beta\in(0,1)$ be the Hölder exponent from Proposition \ref{prop:asymp1}. Then in the sequel, we assume that $\beta=\alpha$. 
\end{conv}

We stress that this does not cause severe restrictions on our set-up, as we did not specify the explicit form of the exponent $\alpha$ (therefore we can always reduce its size appropriately) and as the regularity of the free boundary is also inferred from a boundary Harnack inequality (c.f. Theorem 7.7 in \cite{CSS}).

\begin{conv}\label{conv:c0}
Without loss of generality we assume that $c(0)=1$, i.e. the blow-up of $\p_nw $ at $0$ is $w_{0,s}$. Moreover, without loss of generality we will assume that $\frac{1}{2}\leq c(x_0)\leq \frac{3}{2} \text{ for any } x_0\in \Gamma_w.$ This can be achieved by a scaling $w(\lambda_0x)/\lambda_0^{1+s}$ for a sufficiently small $\lambda_0=\lambda_0(n,s)$. 
\end{conv}

Invoking interior regularity estimates for the fractional Laplacian in non-tangential regions (c.f. Propositions \ref{prop:Dirichlet} and \ref{prop:Neumann}), we also obtain the leading order asymptotic expansion for higher derivatives of $w$. To this end, we compare the derivatives of $w$ with their blow-up at a given free boundary point $x_0$. More precisely for $x_0\in \Gamma_w$ and $\lambda \in (0,1/2)$, we consider 
\begin{align}\label{eq:scale_w}
w_{x_0,\lambda}(x):=\frac{w(x_0+\lambda x)}{\lambda^{1+s}}. 
\end{align}
We denote the associated blow-up of $w_{x_0,\lambda}$ by $w_{x_0}(x):= \lim\limits_{\lambda \rightarrow 0} w_{x_0,\lambda}(x)$ and note that by the asymptotics from Propostion \ref{prop:asymp1}
\begin{align}\label{eq:blowup_w}
w_{x_0}(x)=c(x_0)w_{1,s}(x\cdot \nu_{x_0},x_{n+1}).
\end{align}
Using interior estimates in non-tangential regions around each $x_0\in \Gamma_w$, we obtain the following (higher order) asymptotic expansion:

\begin{prop}
\label{prop:asymp2}
Let $w:B_1^+ \rightarrow \R$ be a solution to \eqref{eq:fracLa_2} and assume that the conditions (A1)-(A4) hold. Let $\alpha$ be the constant from Proposition \ref{prop:asymp1}. Then in a non-tangential cone $\mathcal{N}_0=\{x:|x|\leq \frac{1}{2}\sqrt{x_n^2+x_{n+1}^2}\}$ we have:
\begin{itemize}
\item[(i)] In $A_-:=\mathcal{N}_0\cap \{x_n<x_{n+1}, 1/4 \leq |x|\leq 2\}$, there exists a constant $C_s>0$ such that for any $\gamma\in(0,1)$ and $\lambda\in (0,1/4)$
\begin{align*}
&\quad \|x_{n+1}^{-2s}(w_{x_0,\lambda}-w_{x_0})\|_{C^{0,\gamma}(A_-)}+\sum_{i=1}^{n}\|x_{n+1}^{-2s}\p_i( w_{x_0,\lambda}- w_{x_0})\|_{C^{0,\gamma}(A_-)}\\
&+\|x_{n+1}^{1-2s}\p_{n+1}(w_{x_0,\lambda}-w_{x_0})\|_{C^{0,\gamma}(A_-)}
+\sum_{i,j=1}^{n+1}\left\| \p_ix_{n+1}^{1-2s}\p_j( w_{x_0,\lambda}- w_{x_0})\right\|_{C^{0,\gamma}(A_-)}\\
&\leq C_s \lambda^{\alpha}.
\end{align*}
\item[(ii)] In $A_+:=\mathcal{N}_0\cap \{x_n>-x_{n+1}, 1/4\leq |x|\leq 2\}$ there exists a constant $C_s>0$ such that for any $\gamma\in(0,1)$ and $\lambda\in (0,1/4)$
\begin{align*}
\|w_{x_0,\lambda}-w_{x_0}\|_{C^{2,\gamma}(A_+)}&\leq C_s \lambda^{\alpha}.
\end{align*}
\end{itemize}
\end{prop}

\begin{proof}
We note that since the regular free boundary in $B_1$ is a $C^{1,\alpha}$ graph, $\Gamma_w=\{x: x_n=g(x'')\}$ with $g(0)=|\nabla'' g(0)|=0$, we have
\begin{align}\label{eq:flat}
Q_{0}(x)=x_n, \quad |Q_{x_0}(x)-Q_{0}(x)|\leq C[\nabla'' g]_{\dot{C}^{0,\alpha}} |x|^{1+\alpha}.
\end{align}
Thus, for $[\nabla'' g]_{\dot{C}^{0,\alpha}}$ sufficiently small, $\Gamma_{w_{x_0,\lambda}}\cap \{x\in \mathcal{N}_0: 1/4\leq |x|\leq 2\}$ is empty for every $\lambda\in (0,1/2)$. Hence,
\begin{align*}
L_s(w_{x_0,\lambda}-w_{x_0})&=x_{n+1}^{1-2s}f_{x_0,\lambda} \text{ in } A_+\cup A_-,\\
w_{x_0,\lambda}-w_{x_0}&=0 \text{ on }B'_1\cap A_-, \\
\lim_{x_{n+1}\rightarrow 0_+} x_{n+1}^{1-2s}\p_{n+1}(w_{x_0,\lambda}-w_{x_0})&=0 \text{ on }B'_1\cap A_+,
\end{align*}
where 
$f_{x_0,\lambda}(x)=\lambda^{1-s} x_{n+1}^2f(x_0+\lambda x)$
with $[f_{x_0,\lambda}]_{\dot{C}^{0,1}(B_1^+)}=\lambda^{2-s}[f]_{\dot{C}^{0,1}(B_\lambda(x_0))}$.
Moreover, by property (i) of Proposition~\ref{prop:asymp1} $$|w_{x_0,\lambda}-w_{x_0}|\leq C_s \lambda^\alpha.$$  Then (i)-(ii) follow immediately from the up to the boundary a priori estimates for the operator $L_s=\nabla\cdot x_{n+1}^{1-2s}\nabla $ with Dirichlet and Neumann boundary conditions, respectively (c.f. Propositions \ref{prop:Dirichlet} and \ref{prop:Neumann}). 
\end{proof}

\section{Hodograph-Legendre Transformation}
\label{sec:HL}
Relying on the asymptotic expansion from the previous section, in this section we carry out a Legendre-Hodograph transformation to fix and flatten the regular free boundary (c.f. Proposition \ref{prop:invert}). While this fixes the free boundary, it comes at the expense of transforming the fractional Laplacian into a fully nonlinear, degenerate elliptic, fractional Baouendi-Grushin type operator (c.f. Proposition \ref{prop:eqnonlin} and Example \ref{ex:model}). As in the previous section, we assume throughout the section that the conditions (A1)-(A4) are valid.\\

We consider a partial Hodograph transformation which is adapted to the fractional thin obstacle problem:
\begin{equation}
\label{eq:Legendre}
\begin{split}
T: B_1^+&\rightarrow Q_+:=\{y\in \R^{n+1}: y_n\geq 0, \ y_{n+1}\geq 0\},\\
x&\mapsto T(x)=:y, \\
y''&=x'', \ 
(y_n)^{2s}=\p_nw(x),\ 
(y_{n+1})^{2(1-s)}=-\frac{1-s}{s}x_{n+1}^{1-2s}\p_{n+1}w(x).
\end{split}
\end{equation}
Here $w:B_1^+ \rightarrow \R$ is a solution to the thin obstacle problem (\ref{eq:fracLa_2}) satisfying the assumptions (A1)-(A4).
We note that by the asymptotic expansions in (ii) and (iii) of Proposition \ref{prop:asymp1}, the Hodograph transform has the following mapping properties
\begin{align*}
&T(\inte(B_1^+)) \subset \inte(Q_+), \ T(\Lambda_w) \subset \{y_n=0, y_{n+1}\geq 0\}, \\
&T(\inte(\Omega_w)) \subset \{y_{n+1}=0, y_{n}> 0\}, \ T(\Gamma_w) \subset \{y_n=0, y_{n+1}=0\}.
\end{align*} 
Associated with the transformation $T$, we define the Legendre function, $v$, associated with $w$
\begin{equation}\label{eq:Leg0}
 v(y):=w(x)-x_ny_n^{2s}+\frac{1}{2(1-s)}x_{n+1}^{2s} y_{n+1}^{2(1-s)},
\end{equation}
where $y=T(x)$. With this definition, the function $v$ satisfies the following dual conditions
\begin{equation}
\label{eq:inverse}
\begin{split}
\partial_{i} v(y)&=\partial_i w(x),\  i\in\{1,\dots, n-1\},\\
y_n^{1-2s} \partial_n v(y)&= -(2s)x_n,\\
y_{n+1}^{2s-1}\partial_{n+1} v(y)&=x_{n+1}^{2s}.
\end{split}
\end{equation}
In particular, the free boundary is parametrized by 
\begin{align*}
x_n=-\frac{1}{2s}y_n^{1-2s}\p_{n}v(y)\big|_{y=(y'',0,0)}.
\end{align*}
Thus the study of the free boundary reduces to the analysis of the regularity properties of $v$ and its (weighted) derivatives (c.f. Sections \ref{sec:asymp1}-\ref{sec:IFT}).

\subsection{Invertibility}
In this section we show that the Hodograph-Legendre transform which was defined in (\ref{eq:Legendre}) and (\ref{eq:inverse}) maps $B_{\delta_0}^+$ invertibly onto its image (c.f. Proposition \ref{prop:invert}) if the radius $\delta_0 = \delta_0(s)$ is chosen small enough. 
To this end, we observe that, if $T$ is the Hodograph-Legendre transform with respect to a solution $w$ of the thin obstacle problem \eqref{eq:fracLa_2}, the regularity of $w$ (c.f. Remark~\ref{rmk:opt}) and the asymptotic expansion of $w$ (c.f. Proposition \ref{prop:asymp1})) immediately imply that $T$ is continuous up to $B'_1$. Moreover, by using the asympototics of $\p_n w$ and $x_{n+1}^{1-2s}\p_{n+1}w$ from Proposition~\ref{prop:asymp1}, we also obtain the asymptotics of $y=T(x)$ in a neighborhood of $x_0\in \Gamma_w$:
\begin{equation}
\label{eq:asym_L}
\begin{split}
y'' &=x'',\\
y_n&=c(x_0)^{1/(2s)}(e_n\cdot \nu_{x_0})^{1/(2s)} w_{0,1/2}(Q_{x_0}(x),x_{n+1})\\
& \quad +O_s\left(w_{0,1/2}(Q_{x_0}(x),x_{n+1})|x-x_0|^{\alpha}\right),\\
y_{n+1}&= c(x_0)^{1/2(1-s)} w_{0,1/2}(-Q_{x_0}(x),x_{n+1})\\
&\quad + O_s\left(w_{0,1/2}(-Q_{x_0}(x),x_{n+1})|x-x_0|^{\alpha}\right).
\end{split}
\end{equation}
Here $Q_{x_0}(x)=(x-x_0)\cdot\nu_{x_0}$. We note that in deducing (\ref{eq:asym_L}) we have implicitly used that $c(x_0)\geq \frac{1}{2}, (e_n\cdot \nu_{x_0})\geq \delta $ for some $\delta>0$ (in order to expand the corresponding roots). These lower bounds are consequences of the assumptions (A2)-(A4). 

Using the asymptotic expansions from (\ref{eq:asym_L}), we obtain improved regularity properties for the Hodograph transform:

\begin{prop}
\label{prop:J}
Assume that $w:B_1^+ \rightarrow \R$ is a solution of (\ref{eq:fracLa_2}). Let $w_{x_0,\lambda}$ and $w_{x_0}$ be the associated rescalings and the blow-up limit from \eqref{eq:scale_w} and \eqref{eq:blowup_w}. Denote by $T^{w_{x_0,\lambda}}$ and $T^{w_{x_0}}$ their respective Hodograph transformations and use $A_+, A_-$ to denote the non-tangential sets from Proposition \ref{prop:asymp2}. Then $T^{w_{x_0,\lambda}}\in C^1(A_+\cup A_-)$ for any $\lambda\in (0,1/2)$. Moreover:
\begin{itemize}
\item[(i)] For any $\gamma\in (0,1)$ and $\lambda\in (0,1/4)$,
\begin{equation}\label{eq:J2}
\quad \|DT^{w_{x_0,\lambda}}-DT^{w_{x_0}}\|_{C^{0,\gamma}(A_+\cup A_-)}\leq C \lambda^{\alpha}.
\end{equation}
\item[(ii)] There exist constants $c_s, C_s>0$ and $\lambda_0=\lambda_0(s)>0$, such that for any $\lambda\in(0,\lambda_0)$ and $x_0\in \Gamma_w$ it holds
\begin{equation}\label{eq:J22}
c_s\leq \|\det(DT^{w_{x_0,\lambda}})\|_{C^{0}(A_+\cup A_-)}\leq C_s.
\end{equation}
\end{itemize}
\end{prop}

\begin{proof}
Using the definition in \eqref{eq:Legendre}, the Jacobian matrix $(DT^w)_{ij}$ can be computed to read
\begin{equation}\label{eq:asymp_J}
\begin{split}
\partial_{i} T^w_j&=\delta_{ij}, \quad i,j\in \{1,\dots, n-1\},\\
\partial_{n} T^w_i & = 0, \quad \partial_{n+1} T^w_{i}=0,\\ 
\partial_i T^w_n&=\frac{1}{2s}y_n^{1-2s}\p_{ni}w, \quad
\partial_i T^w_{n+1}=-\frac{1}{2s}y_{n+1}^{2s-1}\p_i(x_{n+1}^{1-2s}\p_{n+1}w) ,\\
(D T^w)_{i,j=n,n+1}&=
\begin{pmatrix}
\frac{1}{2s}y_n^{1-2s}\p_{nn}w & \frac{1}{2s}y_n^{1-2s}\p_{n,n+1}w\\
-\frac{1}{2s}y_{n+1}^{2s-1}\p_n(x_{n+1}^{1-2s}\p_{n+1}w) & -\frac{1}{2s}y_{n+1}^{2s-1}\p_{n+1}(x_{n+1}^{1-2s}\p_{n+1}w)
\end{pmatrix}.
\end{split}
\end{equation}
Here $y= T^{w}(x)=(T^w_1(x),\dots,T^w_{n+1}(x))$. 
We note that by the asymptotic expansions from \eqref{eq:asym_L}, there exists a constant $c\in (0,1)$ such that
\begin{equation}\label{eq:x_y}
\begin{split}
&1-c\leq \left|\frac{y_n(x)}{x_{n+1}}\right|,\  |y_{n+1}(x)|\leq 1+c,\  \text{ on } A_-,\\
&1-c\leq \left|\frac{y_{n+1}(x)}{x_{n+1}}\right|,\  |y_{n}(x)|\leq 1+c,\  \text{ on } A_+.
\end{split}
\end{equation}
Thus, using the asymptotics in (\ref{eq:asym_L}) in combination with the regularity of $\nu_{x_0}$, we infer that
\begin{align*}
\|T^{w_{x_0}}-T^{w_{x_0,\lambda}}\|_{L^{\infty}} \leq C \lambda^{\alpha}.
\end{align*}

We proceed with the H\"older estimates. Here we seek to reduce the estimates to already known bounds on the function $w$ (c.f. the estimates from Proposition \ref{prop:asymp2}). Hence, we have to control the terms in (\ref{eq:asymp_J}) which involve expressions in $y$ by comparable expressions in $x$.\\
We begin by rewriting the expressions with $y$ in terms of controlled expressions in $x$:
\begin{align*}
\left(\frac{y_n(x)}{x_{n+1}}\right)^{1-2s}&=(x_{n+1}^{-2s}\p_nw)^{\frac{1-2s}{2s}},\\
y_{n+1}(x)^{2s-1}&=\left(-\frac{1-s}{s}x_{n+1}^{1-2s}\p_{n+1}w\right)^{\frac{2s-1}{2(1-s)}}, \text{ if } x\in A_-,
\end{align*}
and 
\begin{align*}
y_n(x)^{1-2s}&=(\p_nw)^{\frac{1-2s}{2s}},\\
\left(\frac{y_{n+1}(x)}{x_{n+1}}\right)^{2s-1}&=\left(-\frac{1-s}{s}x_{n+1}^{-1}\p_{n+1}w\right)^{\frac{2s-1}{2(1-s)}},\quad \text{if } x\in A_+.
\end{align*}
By \eqref{eq:x_y} and Proposition~\ref{prop:asymp2}, if $[\nabla'' g]_{\dot{C}^{0,\alpha}}$ is sufficiently small,  we have
\begin{align*}
\|(x_{n+1}^{-2s}\p_nw_{x_0,\lambda})^{\frac{1-2s}{2s}}-(x_{n+1}^{-2s}\p_nw_{x_0})^{\frac{1-2s}{2s}}\|_{C^{0,\gamma}(A_-)}\leq C_s \lambda^{\alpha}.
\end{align*}
Here we used that $x_{n+1}^{-2s}\p_nw_{x_0}$ is uniformly bounded away from zero in $A_-$ for each $\lambda\in (0,1/2)$. As a consequence, combining the above estimates with the estimates in Proposition \ref{prop:asymp2} (which controls $x_{n+1}^{1-2s}\p_{in}(w_{x_0,\lambda}-w_{x_0})$), we infer a bound for $\|\p_{i}T^{w_{x_0,\lambda}}_n- \p_{i} T^{w_{x_0}}_n\|_{C^{0,\gamma}(A_-)}$ with $i\in\{1,\dots,n-1\}$. \\
Similarly,
\begin{align*}
&\quad \left\|\left(-\frac{1-s}{s}x_{n+1}^{1-2s}\p_{n+1}w_{x_0,\lambda}\right)^{\frac{1-2s}{2(1-s)}}-\left(-\frac{1-s}{s}x_{n+1}^{1-2s}\p_{n+1}w_{x_0}\right)^{\frac{1-2s}{2(1-s)}}\right\|_{C^{0,\gamma}(A_-)}\\
&+\left\|\left(-\frac{1-s}{s}x_{n+1}^{-1}\p_{n+1}w_{x_0,\lambda}\right)^{\frac{2s-1}{2(1-s)}}-\left(-\frac{1-s}{s}x_{n+1}^{-1}\p_{n+1}w_{x_0}\right)^{\frac{2s-1}{2(1-s)}}\right\|_{C^{0,\gamma}(A_+)}\\
&+\left\|(\p_nw_{x_0,\lambda})^{\frac{1-2s}{2s}}-(\p_nw_{x_0})^{\frac{1-2s}{2s}}\right\|_{C^{0,\gamma}(A_+)}\leq C_s \lambda^{\alpha}.
\end{align*}
Invoking Proposition \ref{prop:asymp2} once more, then yields the Hölder bounds.\\
Finally, we note that by \eqref{eq:asym_L}
\begin{align}\label{eq:T0}
T^{w_0}(x)=\left(x'',w_{0,1/2}(x_n,x_{n+1}), w_{0,1/2}(-x_n,x_{n+1})\right),
\end{align}
which in the $x_n,x_{n+1}$ variables is the square root mapping. Moreover, 
\begin{align*}
&T^{w_{x_0}}(x)\\
=&\left(x'', c(x_0)^{\frac{1}{2s}}(e_n\cdot \nu_{x_0})^{\frac{1}{2s}}w_{0,1/2}(x \cdot \nu_{x_0}, x_{n+1}), c(x_0)^{\frac{1}{2(1-s)}}w_{0,1/2}(-x\cdot \nu_{x_0},x_{n+1})\right).
\end{align*}
Computing the explicit expression for $\det(DT^{w_{x_0}})$ and using Convention~\ref{conv:c0}, we conclude that there exist constants $c_s, C_s>0$ which only depend on $s$, such that 
\begin{align}
\label{eq:J}
c_s \leq |\det(DT^{w_{x_0}})|\leq C_s, \quad \forall x_0\in \Gamma_w\cap B_1,
\end{align}
if $[\nabla'' g]_{\dot{C}^{0,\alpha}}$ is sufficiently small.
Combining this with the first estimate from (\ref{eq:J2}) yields the desired result.
\end{proof}

As a consequence of the previous proposition the Hodograph-Legendre transformation satisfies the same properties as the one in \cite{KRSIII}. Arguing along the lines of Proposition 3.8 of \cite{KRSIII}, we therefore obtain the invertibility of the Hodograph-Legendre transform: 

\begin{prop}
\label{prop:invert}
Let $w:B_1^+ \rightarrow \R$ be a solution to \eqref{eq:fracLa_2} and assume that the conditions (A1)-(A4) hold.
If $[\nabla'' g]_{\dot{C}^{0,\alpha}}$ is sufficiently small, then there exists a radius $\delta_0=\delta_0(s)>0$ such that $T:=T^{w}$ is a homeomorphism from $B_{\delta_0}^+$ to $T(B_{\delta_0}^+)\subset\{y\in \R^{n+1}: y_n \geq 0, y_{n+1}\geq 0\}$. Moreover, away from $\Gamma_w$, $T$ is a $C^1$-diffeomorphism. 
\end{prop}

\begin{proof}
We only give an outline of the proof. As in \cite{KPS14} and \cite{KRSIII} we reduce the invertibility discussion to an analysis on annuli at which the mapping $T$ is $C^1$.\\

\emph{Step 1.} Let $\psi:\R^n\rightarrow \R^n$, $\psi(z)=(z'',\frac{1}{2}(z_n^2-z_{n+1}^2), z_nz_{n+1})$ and let $T_1=\psi\circ T$. 
We observe that by \eqref{eq:T0} the map $T_1^{w_{0}}$ is a non-degenerate linear map. By the definition of $T_1$,  by Proposition~\ref{prop:J} (i) and by a triangle inequality, we obtain that
\begin{align*}
\|T_1^{w_{x_0,\lambda}}-T_1^{w_{0}}\|_{C^1(A_{1/2,1})}\leq C_s\left([\nabla'' g]_{\dot{C}^{0,\alpha}}+\lambda^{\alpha}\right),
\end{align*}
where $A_{1/2,1}=\{x: |x''|\leq 1, \frac{1}{2}\leq \sqrt{x_n^2+x_{n+1}^2}\leq 1, x_{n+1}\geq 0\}$. 
This implies that if $[\nabla g]_{\dot{C}^{0,\alpha}}$ and $\delta_0$ are sufficiently small depending on $s$, then $T_1^{w_{x_0,\lambda}}$ is injective in $A_{1/2,1}$ for $\lambda\in (0,\delta_0)$. By scaling we conclude that $T_1$ is injective in $A_{r/2,r}(x_0)$ for each $r\in (0,\delta_0)$. Here $A_{r/2,r}(x_0)=x_0+r A_{1/2,1}$.\\

\emph{Step 2.} Next we show that $T_1$ is injective from $B_{\delta_0}^+$ to $T_1(B_{\delta_0}^+)$.
As in \cite{KPS14} and \cite{KRSIII} it suffices to prove the injectivity for fixed $x''$. To this end, we assume that injectivity were wrong, i.e. there existed two points $x_1, x_2\in B_1^+$ with $x_1''=x_2''=x''_0$ but $x_1 \neq x_2$ such that $T_1(x_1)= T_1(x_2)$. As $\Gamma_w$ is a graph, $T_1$ is injective on it. Furthermore, the non-degeneracy conditions from Proposition \ref{prop:asymp1} further imply that we may assume that neither $x_1$ nor $x_2$ are free boundary points.\\ 
By the asymptotics of $y_n(x)$, $y_{n+1}(x)$ we observe that
\begin{align*}
T_1(x)=x + E_{x_0}(x),\quad x\in B_{1/2}(x_0),
\end{align*}
where $x_0=(x''_0,g(x''_0),0)$ and $|E_{x_0}(x)|\leq C[\nabla'' g]_{\dot{C}^{0,\alpha}}|x_0|^{\alpha}.$ Using this in combination with the assumption that $T_1(x_1)=T_1(x_2)$ thus leads to
$$
x_2-x_0 = x_1 - x_0 + E_{x_0}(x_1)-E_{x_0}(x_2).
$$
By the estimate for $E_{x_0}$ this however implies
$$
\frac{1}{2} \leq \frac{|x_1-x_0|}{|x_2-x_0|} \leq 2.
$$
Hence, the points $x_1,x_2$ lie in $A_{r,2r}(x_0)$ for $r=\max\{|x_1-x_0|,|x_2-x_0|\}$. This is a contradiction to the injectivity result from Step 1.\\

\emph{Step 3.} As a consequence of the invariance of domain theorem, the previous result implies that $T_1$ is a homeomorphism from $\inte(B_{\delta_0}^+)$ to $T_1(\inte(B_{\delta_0}^+))$. By a discussion of the respective boundary points, it further extends to a homeomorphism from $B_{\delta_0}^+$ to $T_1(B_{\delta_0}^+)$. Hence, also $T$ is a homeomorphism from $B_{\delta_0}^+$ to $T(B_{\delta_0}^+)$. The fact that $T$ is a diffeomorphism follows from Proposition~\ref{prop:J} (ii) and the second equation in \eqref{eq:J2}.
\end{proof}

\subsection{Nonlinear PDE} 
In this section we compute the equation which is satisfied by the Legendre function associated with a solution $w$ of the thin obstacle problem \eqref{eq:fracLa_2} such that the assumptions (A1)-(A4) are satisfied.
As our main result of this section we show that the Legendre function $v$ solves a Monge-Amp\`ere type PDE. A first example (c.f. Example \ref{ex:model}) indicates that this equation should be interpreted as a perturbation of a fractional Baouendi-Grushin equation.

\begin{prop}
\label{prop:eqnonlin}
Let $w:B_1^+ \rightarrow \R$ be a solution to \eqref{eq:fracLa_2} and assume that the conditions (A1)-(A4) hold. Let $y=T^w(x)$ and let $v$ be a Legendre function associated with $w$ (c.f.\eqref{eq:Leg0}). Then in $T^w(B_{\delta_0}^+)$ (where $\delta_0=\delta_0(s)$ is the same constant as in Proposition~\ref{prop:invert}) the function $v$ satisfies the fully nonlinear equation
\begin{equation} 
\label{eq:nonlinear}
\begin{split}
&F(D^2v, Dv,y)= \\
&\sum_{i=1}^{n-1}x_{n+1}(y)^{2-4s}\det\begin{pmatrix}
\p_{ii}v & \p_{in}v & \p_{i,n+1}v\\
\frac{1}{2s}\p_i(y_n^{1-2s}\p_{n}v) & -\frac{1}{2s}\p_n(y_n^{1-2s}\p_nv) & -\frac{1}{2s}\p_{n+1}(y_n^{1-2s}\p_nv)\\
-\frac{1}{2s}\p_i(y_{n+1}^{2s-1}\p_{n+1}v)& \frac{1}{2s}\p_n(y_{n+1}^{2s-1}\p_{n+1}v)& \frac{1}{2s}\p_{n+1}(y_{n+1}^{2s-1}\p_{n+1}v)
\end{pmatrix}\\
&+\p_n(y_n^{1-2s}y_{n+1}^{1-2s}\p_{n}v)+ x_{n+1}(y)^{2-4s}\p_{n+1}(y_n^{2s-1}y_{n+1}^{2s-1}\p_{n+1}v)\\
&-x_{n+1}(y)^{3-2s}J(v)f(y'',x_n(y),x_{n+1}(y))=0,
\end{split}
\end{equation}
with mixed Dirichlet-Neumann boundary condition 
\begin{align*}
v(y)=0 \text{ on } \{y_n=0\}, \quad \lim\limits_{y_{n+1}\rightarrow 0_+} y_{n}^{1-2s}y_{n+1}^{1-2s}\p_{n+1}v(y)=0 \text{ on }\{y_{n+1}=0\}.
\end{align*}
Here
\begin{align*}
x_n(y)&=-\frac{1}{2s}y_n^{1-2s}\p_nv(y),\\
x_{n+1}(y)&=(y_{n+1}^{2s-1}\p_{n+1}v)^{\frac{1}{2s}},\\
J(v)&=\det\begin{pmatrix}
-\frac{1}{2s}\p_n(y_n^{1-2s}\p_nv)& -\frac{1}{2s}\p_{n+1}(y_n^{1-2s}\p_{n}v)\\
\frac{1}{2s} x_{n+1}^{1-2s}\p_{n}(y_{n+1}^{2s-1}\p_{n+1}v) & \frac{1}{2s} x_{n+1}^{1-2s} \p_{n+1}(y_{n+1}^{2s-1}\p_{n+1}v)
\end{pmatrix}.
\end{align*}
\end{prop}

\begin{rmk}[Boundary data]
\label{rmk:bound_data}
We point out that the boundary condition for $v$ is not uniquely determined. In particular, this holds for the Neumann condition. By \eqref{eq:inverse} we know that $\p_{n+1}v(y)=y_{n+1}^{1-2s}x_{n+1}^{2s}(y)$, which by \eqref{eq:inverse22} is of the magnitude $O(y_n^{2s}y_{n+1})$ as $y_{n+1}\rightarrow 0_+$. Thus, the Legendre function for instance also satisfies the condition $\lim\limits_{y_{n+1}\rightarrow 0_+} \p_{n+1} (y_n^{-2s}v)=0$.
In the formulation of Proposition \ref{prop:eqnonlin} we have chosen boundary conditions which ``fit'' to the linearized operator and come handy in proving approximation results in Sections \ref{sec:approx_poly}. However, in the definition of our function spaces $X_{\alpha,\epsilon}$ we will use the flexibility which we have at this point (c.f. Definition \ref{defi:XY}, Proposition \ref{prop:decomp} and its proof in Section \ref{sec:prove}).
\end{rmk}

\begin{rmk}[Relation to $s=1/2$]
We remark that for $s=1/2$ the above expression simplifies to the equation from \cite{KRSIII} with $a^{ij}=\delta^{ij}$.
\end{rmk}

\begin{rmk}[Divergence structure of leading terms]
We remark that it is possible to rewrite all terms of the form
\begin{align*}
x_{n+1}^{c}\p_{j}(y_{n+1}^{2s-1}\p_{n+1}v), \ j\in\{1,\dots,n+1\},
\end{align*}
into divergence form:
\begin{align*}
x_{n+1}^{c}\p_{j}(y_{n+1}^{2s-1}\p_{n+1}v) = \frac{4s^2+c}{4s^2} \p_j (x_{n+1}^c y_{n+1}^{2s-1}\p_{n+1}v).
\end{align*}
This demonstrates that we may view the leading order part of the equation either as a divergence or as a non-divergence form operator.
\end{rmk}

\begin{proof}
We compute the corresponding expressions for $w$ and its derivatives in terms of $v$. We first observe that the following 
relations hold between $D^2 w$ and $D^2 v$:
\begin{align*}
\p_{ii}w&=\p_{ii}v+(\p_{in}v, \p_{i,n+1}v)\cdot\begin{pmatrix}
\frac{\p y_n}{\p x_i}\\
\frac{\p y_{n+1}}{\p x_i}
\end{pmatrix},  \\
(x_{n+1})^{1-2s}\p_{nn}w&=(2s)(x_{n+1})^{1-2s}y_n^{2s-1}\frac{\p y_n}{\p x_n},\\
\p_{n+1}(x_{n+1}^{1-2s}\p_{n+1}w)&=(-2s)y_{n+1}^{1-2s}\frac{\p y_{n+1}}{\p x_{n+1}}.
\end{align*} 
Let 
\begin{align*}
J(v):=\det\begin{pmatrix}
\frac{\p x_{n}}{\p y_n} & \frac{\p x_{n}}{\p y_{n+1}}\\
\frac{\p x_{n+1}}{\p y_n} & \frac{\p x_{n+1}}{\p y_{n+1}}
\end{pmatrix}=\det\begin{pmatrix}
-\frac{1}{2s}\p_n(y_n^{1-2s}\p_nv)& -\frac{1}{2s}\p_{n+1}(y_n^{1-2s}\p_{n}v)\\
\frac{1}{2s} x_{n+1}^{1-2s}\p_{n}(y_{n+1}^{2s-1}\p_{n+1}v) & \frac{1}{2s} x_{n+1}^{1-2s} \p_{n+1}(y_{n+1}^{2s-1}\p_{n+1}v)
\end{pmatrix}.
\end{align*}
Using
\begin{align*}
\begin{pmatrix}
\frac{\p y_{n}}{\p x_n} & \frac{\p y_{n}}{\p x_{n+1}}\\
\frac{\p y_{n+1}}{\p x_n} & \frac{\p y_{n+1}}{\p x_{n+1}}
\end{pmatrix}=\frac{1}{J(v)}\begin{pmatrix}
\frac{\p x_{n+1}}{\p y_{n+1}} & -\frac{\p x_{n}}{\p y_{n+1}}\\
-\frac{\p x_{n+1}}{\p y_n} & \frac{\p x_{n}}{\p y_{n}}
\end{pmatrix},
\end{align*}
we write 
\begin{align*}
(x_{n+1})^{1-2s}\p_{nn}w&=(2s)(x_{n+1})^{1-2s}y_n^{2s-1}\frac{1}{J(v)}\frac{\p x_{n+1}}{\p y_{n+1}},\\
\p_{n+1}(x_{n+1}^{1-2s}\p_{n+1}w)&=(-2s)y_{n+1}^{1-2s}\frac{1}{J(v)}\frac{\p x_{n}}{\p y_{n}}.
\end{align*}
By virtue of \eqref{eq:inverse},
\begin{align*}
\frac{\p x_n}{\p y_n}&=-\frac{1}{2s}\p_n(y_n^{1-2s}\p_nv),\\
\frac{\p x_{n+1}}{\p y_{n+1}}&=\frac{1}{2s}x_{n+1}^{1-2s}\p_{n+1}(y_{n+1}^{2s-1}\p_{n+1}v).
\end{align*}
Thus, 
\begin{align*}
(x_{n+1})^{1-2s}\p_{nn}w&=\frac{1}{J(v)}y_n^{2s-1}x_{n+1}^{2-4s}\p_{n+1}(y_{n+1}^{2s-1}\p_{n+1}v),\\
\p_{n+1}(x_{n+1}^{1-2s}\p_{n+1}w)&=\frac{1}{J(v)}y_{n+1}^{1-2s}\p_n(y_n^{1-2s}\p_nv).
\end{align*}
In order to express $\p_{ii}w$ in terms of $v$, we use 
\begin{align*}
&\begin{pmatrix}
\frac{\p y_n}{\p x_i}\\
\frac{\p y_{n+1}}{\p x_i}
\end{pmatrix}=\frac{1}{J(v)}\begin{pmatrix}
\frac{\p x_{n+1}}{\p y_{n+1}} & -\frac{\p x_{n}}{\p y_{n+1}}\\
-\frac{\p x_{n+1}}{\p y_n} & \frac{\p x_{n}}{\p y_{n}}
\end{pmatrix}
\begin{pmatrix}
\frac{\p x_n}{\p y_i}\\
\frac{\p x_{n+1}}{\p y_i}
\end{pmatrix}\\
&=\frac{1}{J(v)}\begin{pmatrix}
\frac{1}{2s}x_{n+1}^{1-2s}\p_{n+1}(y_{n+1}^{2s-1}\p_{n+1}v) & \frac{1}{2s}\p_{n+1}(y_n^{1-2s}\p_{n}v)\\
-\frac{1}{2s}x_{n+1}^{1-2s}\p_{n}(y_{n+1}^{2s-1}\p_{n+1}v) & -\frac{1}{2s}\p_n(y_n^{1-2s}\p_nv)
\end{pmatrix}
\begin{pmatrix}
\frac{-1}{2s}y_n^{1-2s}\p_{in}v\\
\frac{1}{2s}x_{n+1}^{1-2s}y_{n+1}^{2s-1}\p_{i,n+1}v
\end{pmatrix}.
\end{align*}
Recalling the equation of $w$
\begin{equation*}
(x_{n+1})^{1-2s} \Delta' w + \partial_{n+1} \left( x_{n+1}^{1-2s} \partial_{n+1} w\right)=(x_{n+1})^{3-2s}f \text{ in } B_1^+,
\end{equation*}
the equation of $w$ is transformed into a nonlinear equation for $v$:
\begin{align*}
&x_{n+1}^{1-2s}\sum_{i=1}^{n-1}\p_{ii}v +  \frac{1}{J(v)}y_n^{2s-1}x_{n+1}^{2-4s}\p_{n+1}(y_{n+1}^{2s-1}\p_{n+1}v)+ \frac{1}{J(v)}y_{n+1}^{1-2s}\p_n(y_n^{1-2s}\p_nv)\\
&+\frac{x_{n+1}^{1-2s}}{J(v)}\sum_{i=1}^{n-1}
\begin{pmatrix}
\p_{in}v & \p_{i,n+1}v 
\end{pmatrix} \cdot\\
&
\begin{pmatrix}
\frac{1}{2s}x_{n+1}^{1-2s}\p_{n+1}(y_{n+1}^{2s-1}\p_{n+1}v) & \frac{1}{2s}\p_{n+1}(y_n^{1-2s}\p_{n}v)\\
-\frac{1}{2s}x_{n+1}^{1-2s}\p_{n}(y_{n+1}^{2s-1}\p_{n+1}v) & -\frac{1}{2s}\p_n(y_n^{1-2s}\p_nv)
\end{pmatrix}
\begin{pmatrix}
\frac{-1}{2s}y_n^{1-2s}\p_{in}v\\
\frac{1}{2s}x_{n+1}^{1-2s}y_{n+1}^{2s-1}\p_{i,n+1}v
\end{pmatrix}\\
& =(x_{n+1})^{3-2s}f(T^{-1}(y)).
\end{align*}
Multiplying by $J(v)$ on both sides and rearranging, yields \eqref{eq:nonlinear}.\\
In order to deduce the form of the boundary data, we note that 
by the mapping properties of $T^w$ and by Proposition~\ref{prop:invert} we have that $\{y_n=0\}\cap T^w(B_{1/2}^+)= T^w(\Lambda_w\cap B_{1/2}^+)$. 
Thus the statement on the Dirichlet data then follows from the definition of $v(y)$ in terms of $w(x)$. In order to infer the result on the Neumann data, we observe that by (\ref{eq:inverse}) and Proposition~\ref{prop:invert} $\{y_{n+1}=0\}\cap T^w(B_{1/2}^+)=  T^{w}(\bar{\Omega}_w\cap B_{1/2}^+)$. The result then follows by recalling that $x_{n+1}^{2s} = y_{n+1} ^{2s-1}\p_{n+1}v(y)$ and by consulting the asmptotics from (\ref{eq:asym_L}) (c.f. also the asymptotics of $x_n^{2s}$ in terms of $y$ from Lemma \ref{lem:reverse}):
$
(y_n y_{n+1})^{1-2s}\p_{n+1}v \sim (y_n y_{n+1})^{2-2s},
$
which vanishes as $y_{n+1} \rightarrow 0_+$.
\end{proof}

\subsection{Asymptotics of the Legendre Function}
\label{sec:asymp1}
In this section we derive the asymptotics of the Legendre function by exploiting the corresponding bounds for the solution to the fractional thin obstacle problem (c.f. Propositions \ref{prop:asymp1}, \ref{prop:J}). In order to achieve this, we first derive a relation between the rescalings of the Legendre and the respective original functions (Lemma \ref{lem:rescaling}). With this at hand, we deduce an asymptotic formula for $x$ in terms of $y$. This then allows to transfer the results from Propositions \ref{prop:asymp1} and \ref{prop:asymp2} to results on the Legendre function (c.f. Propositions \ref{prop:asymp_v1}, \ref{prop:asymp_v2}).\\

\begin{lem}
\label{lem:rescaling}
Let $v$ be the Legendre function associated to a solution $w$ of the fractional thin obstacle problem (\ref{eq:fracLa_2}) satisfying the assumptions (A1)-(A4). Let $y_0\in T^{w}(\Gamma_w\cap B'_{\delta_0})$, where $B^+_{\delta_0}$ is as in Proposition~\ref{prop:invert}. Then the function
\begin{align*}
v_{y_0,\lambda}(y)&:=\frac{v(y_0+\delta_\lambda (y))-(\frac{1}{2s}y_n^{1-2s}\p_nv)(y_0)(\delta_\lambda (y))_n^{2s}}{\lambda^{2+2s}},
\end{align*}
with $\delta_{\lambda}(y)=(\lambda^2 y'', \lambda y_n, \lambda y_{n+1})$ is the Legendre function of 
\begin{align*}
w_{x_0,\lambda^2}(x):=\frac{w(x_0+\lambda^2 x)}{\lambda^{2(1+s)}}, \quad x_0= (T^{w_{x_0,\lambda^2}})^{-1}(y_0)
\end{align*}
with the Hodograph transformation $y= T^{w_{x_0,\lambda^2}}(x)$.
In particular, $v_{y_0}(y)$, $y=T^{w_{x_0}}(x)$, is the Legendre function of $w_{x_0}$, where $v_{y_0}$ and $w_{x_0}$ denote the respective blow-ups of $v_{y_0,\lambda}$, $w_{x_0,\lambda^2}$ as $\lambda\rightarrow 0_+$.
\end{lem}

\begin{proof}
Consider the Hodograph transformation associated with $w_{x_0,\lambda^2}$ and let $y(x):=T^{w_{x_0,\lambda^2}}(x)$ (c.f. \eqref{eq:Legendre}). Then,
\begin{equation*}
\label{eq:coord_a}
\begin{split}
 (\lambda y_n)^{2s}&=\p_nw(z)|_{z=x_0+\lambda^2 x}, \\
  (\lambda y_{n+1})^{2(1-s)}&=-\frac{1-s}{s}(\lambda^2 x_{n+1})^{1-2s}\p_{n+1}w(z)|_{z=x_0+\lambda^2 x}.
  \end{split}
\end{equation*}
Let $\tilde{v}_{y_0,\lambda}$ be the Legendre function of $w_{x_0,\lambda^2}$ associated with the Hodograph transformation $T^{w_{x_0,\lambda^2}}$. Then by the definition of the Legendre function in (\ref{eq:Leg0}),
 \begin{align*}
 \tilde{v}_{y_0,\lambda}(y)&= w_{x_0,\lambda^2}(x)-x_ny_n^{2s}+\frac{1}{2(1-s)}x_{n+1}^{2s}y_{n+1}^{2(1-s)}\\
 &=\frac{w(x_0+\lambda^2 x)}{\lambda^{2+2s}}-\frac{\lambda^2 x_n(\lambda  y_n)^{2s}}{\lambda^{2+2s}}+\frac{1}{2(1-s)}\frac{(\lambda^2 x_{n+1})^{2s}(\lambda y_{n+1})^{2(1-s)}}{\lambda^{2+2s}}\\
 &=\frac{w(x_0+\lambda^2x)}{\lambda^{2+2s}}-\frac{(x_0+\lambda^2x)_n(\lambda y_n)^{2s}}{\lambda^{2+2s}}+\frac{(x_0)_n(\lambda y_n)^{2s}}{\lambda^{2+2s}}\\
 &\quad  +\frac{1}{2(1-s)}\frac{((x_0+\lambda^2x)_{n+1})^{2s}(\lambda y_{n+1})^{2(1-s)}}{\lambda^{2+2s}}.
 \end{align*}
Since $v$ is the Legendre function of $w$ associated with the Hodograph transformation $T^w$, equations (\ref{eq:coord_a}) and \eqref{eq:inverse} yield
 \begin{equation}\label{eq:rescale1}
 \begin{split}
x_0+\lambda^2 x&=(z'',-\frac{1}{2s}z_n^{1-2s}\p_nv(z), (z_{n+1}^{2s-1}\p_{n+1}v(z))^{\frac{1}{2s}})\big |_{z=y_0+\delta_\lambda (y)}\\
 &=(T^w)^{-1}(y_0+\delta_\lambda (y)).
 \end{split}
 \end{equation}
Using this in the expression of $\tilde{v}_{y_0,\lambda}$ and invoking again the definition of $v$ in \eqref{eq:Leg0} we have
 \begin{align*}
 \tilde{v}_{y_0,\lambda}(y)=\frac{v(y_0+\delta_\lambda (y))}{\lambda^{2+2s}}+\frac{(x_0)_n(\lambda y_n)^{2s}}{\lambda^{2+2s}}.
 \end{align*}
Finally, recalling that 
$$(x_0)_n=g(x''_0)=-\frac{1}{2s}(y_n^{1-2s}\p_nv)(y_0),$$ 
we obtain $\tilde{v}_{y_0,\lambda}=v_{y_0,\lambda}$.
\end{proof}

Next, we reverse the asymptotic expansion of $T^{w}$ which yields the following explicit formulae. We recall that $g$ is the parametrization of the regular free boundary, i.e. $\Gamma_w\cap B'_1=\{(x'',x_n,0)\in B'_1: x_n=g(x'')\}$ with $g\in C^{1,\alpha}(B''_1)$.

\begin{lem}
\label{lem:reverse}
Suppose that $w:B_1^+ \rightarrow \R$ is a solution of \eqref{eq:fracLa_2} such that the assumptions (A1)-(A4) hold. Let $T^{w}$ be the associated Hodograph transform and let $y_0\in T^w(\Gamma_w\cap B_{\delta_0})$, $x_0=(T^{w})^{-1}(y_0)$. Then,
\begin{itemize}
\item[(i)] the following asymptotic expansions hold:
\begin{equation}
\label{eq:inverse22}
\begin{split}
x_n(y)&=g(y'')  + a_0(y'')y_n^2-a_1(y'')y_{n+1}^2+O_s\left((y_n^2+y_{n+1}^2)^{1+\alpha}\right),\\
x_{n+1}^{2s}(y)&=2a_1(y'')y_n^{2s}y_{n+1}^{2s}+O_s\left(y_n^{2s}y_{n+1}^{2s}(y_n^2+y_{n+1}^2)^{\alpha}\right).
\end{split}
\end{equation}
Here, 
\begin{align*}
a_0(y_0)=\frac{1}{2c(x_0)^{1/s}(e_n\cdot \nu_{x_0})^{(1+s)/s}}, \quad a_1(y_0)=\frac{1}{2c(x_0)^{1/(1-s)}(e_n\cdot \nu_{x_0})},
\end{align*}
are positive $C^{0,\alpha}$ functions which are uniformly bounded away from zero.
\item[(ii)] We have that
\begin{equation*}
\begin{split}
v(y)& =-g(y'')y_n^{2s}-\frac{s}{1+s}a_0(y'')y_n^{2s+2}+a_1(y'')y_n^{2s}y_{n+1}^2+ y_n^{2s}y_{n+1}^2 O_s\left((y_n^2+y_{n+1}^2)^{\alpha}\right),\\
v_{y_0}(y) &= - \nabla'' g(y_0'')\cdot(y''-y_0'') y_{n}^{2s} -\frac{s}{1+s}a_0(y''_0)y_n^{2s+2}+a_1(y''_0)y_n^{2s}y_{n+1}^2.
\end{split}
\end{equation*}
\end{itemize}
\end{lem}

\begin{proof}
We begin by deriving the asymptotics in (i).
Reversing the asymptotics (\ref{eq:asym_L}) of $y=T(x)$ around $x_0\in \Gamma_w$, we compute that in a neighborhood of $y_0=T(x_0)=(x''_0,0,0)$,
\begin{align*}
(x-x_0)\cdot \nu_{x_0}&=\frac{1}{2}\left(\frac{y_n^2}{c(x_0)^{1/s}(e_n\cdot \nu_{x_0})^{1/s}}-\frac{y_{n+1}^2}{c(x_0)^{1/(1-s)}}\right)+O_s\left((y_n^2+y_{n+1}^2)^{1+\alpha}\right),\\
x_{n+1}&=\frac{y_n y_{n+1}}{c(x_0)^{1/(2s)}c(x_0)^{1/2(1-s)}(e_n\cdot \nu_{x_0})^{1/(2s)}}+O_s\left(y_ny_{n+1}(y_n^2+y_{n+1}^2)^{\alpha}\right).
\end{align*}
Using that $\nu_{x_0}=\frac{(-\nabla'' g(x_0''),1, 0)^t}{\sqrt{1+|\nabla g(x''_0)|^2}},$ 
we furthermore obtain  
\begin{align*}
x_n&= g(y_0'')+\nabla'' g(y_0'')\cdot (y''-y_0'') \\
&+\frac{1}{2}\left(\frac{y_n^2}{c(x_0)^{1/s}(e_n\cdot \nu_{x_0})^{(1+s)/s}}-\frac{y_{n+1}^2}{c(x_0)^{1/(1-s)}(e_n\cdot \nu_{x_0})}\right)+O_s\left((y_n^2+y_{n+1}^2)^{1+\alpha}\right).
\end{align*}
By Convention~\ref{conv:c0} and by the assumption (A3) the functions $c(x_0)$ and $(e_n\cdot \nu_{x_0})$ are uniformly bounded away from zero and are $C^{0,\alpha}$ regular functions. Setting
\begin{equation*}
a_0(y_0)=\frac{1}{2c(x_0)^{1/s}(e_n\cdot \nu_{x_0})^{(1+s)/s}}, \quad a_1(y_0)=\frac{1}{2c(x_0)^{1/(1-s)}(e_n\cdot \nu_{x_0})},
\end{equation*}
we hence arrive at the desired asymptotics in \eqref{eq:inverse22}.
The formula for $v$ in (ii) follows by integrating the relations
\begin{align*}
\p_{n+1}v(y) = y_{n+1}^{1-2s}x_{n+1}^{2s}, \ \p_{n}v(y) = -2s y_{n}^{2s-1}x_n,
\end{align*}
in combination with the asymptotics from (\ref{eq:inverse}) and the mixed Dirichlet-Neumann boundary condition. Using the expression for $v_{y_0,\lambda}$ in Lemma~\ref{lem:rescaling} and passing to the limit we obtain the formula for $v_{y_0}$.
\end{proof}

With these auxiliary results at hand, we now address the asymptotic behavior of the Legendre transform (which can be regarded as a partial analogue of Proposition \ref{prop:asymp1}):

\begin{prop}\label{prop:inverse_T}
Let $\mathcal{C}_1^+:=\{y\in Q_+:|y''|<1, \frac{1}{4}<y_n^2+y_{n+1}^2<1\}$. There exists $\lambda_0=\lambda_0(s)>0$ such that for any $\lambda\in (0,\lambda_0)$ and any $x_0\in \Gamma_w\cap B_{\delta_0}^+$, $(T^{w_{x_0,\lambda^2}})^{-1} \in C^{1,\gamma}(\mathcal{C}_1^+)$. Moreover, for any $\gamma\in (0,1)$,
\begin{align*}
 \|(T^{w_{x_0,\lambda^2}})^{-1} - (T^{w_{x_0}})^{-1}\|_{L^\infty(B_1\cap Q_+)}&\leq C_s \lambda^{2 \alpha},\\
\|D(T^{w_{x_0,\lambda^2}})^{-1} - D(T^{w_{x_0}})^{-1}\|_{C^{0,\gamma}(\mathcal{C}_1^+)}&\leq C_s\lambda^{2\alpha}.
\end{align*}
\end{prop}

\begin{proof}
The first inequality directly follows from Lemma~\ref{lem:reverse} (i) by exploiting the coefficient regularity.\\
The second inequality follows from the first inequality and Proposition~\ref{prop:J} (i), (ii). More precisely, one uses the relation $DT^{-1}(y)=(DT(T^{-1}(y)))^{-1}$. The constant $\lambda_0$ is determined in the same way as Proposition~\ref{prop:J} (ii).
\end{proof}

\begin{prop}
\label{prop:asymp_v1}
Let $v_{y_0,\lambda}$ and $v_{y_0}$ be as in Lemma \ref{lem:rescaling}. Let $\lambda_0$ and $\mathcal{C}_1^+$ be as in Proposition~\ref{prop:inverse_T}. Then, 
\begin{align*}
&\sum_{i=1}^{n-1}\|y_n^{-2s}\p_i(v_{y_0,\lambda}-v_{y_0})\|_{C^{0,\gamma}(\mathcal{C}^+_1)}
+\|y_n^{1-2s}\p_n(v_{y_0,\lambda}-v_{y_0})\|_{C^{0,\gamma}(\mathcal{C}^+_1)}\\
&+
\|y_n^{-2s}y_{n+1}^{-1}\p_{n+1}(v_{y_0,\lambda}-v_{y_0})\|_{C^{0,\gamma}(\mathcal{C}^+_1)}\leq C_s\lambda^{2\alpha},
\end{align*}
for any $\lambda \in (0,\lambda_0)$ and $\gamma\in (0,1)$.
\end{prop}

\begin{proof}
 We note that by \eqref{eq:inverse} 
\begin{equation}
\label{eq:est23}
\begin{split}
y_n^{-2s}\p_iv(y)&=\frac{\p_iw(x)}{\p_nw(x)}\big|_{x=(T^w)^{-1}(y)},\ i\in \{1,\dots, n-1\},\\
y_n^{1-2s}\p_n v(y)&=-(2s)x_n\big|_{x=(T^w)^{-1}(y)},\\
y_n^{-2s}y_{n+1}^{-1}\p_{n+1}v(y)&=y_n^{-2s}y_{n+1}^{-2s}x_{n+1}^{2s}\big|_{x=(T^w)^{-1}(y)}.
\end{split}
\end{equation}

\emph{Step 1: Estimates for $i\in \{1,\dots,n-1\}$}. Using \eqref{eq:est23} and the explicit expression for $v_{y_0}$ in Lemma~\ref{lem:reverse} (ii), we observe that
\begin{align*}
y_n^{-2s}\p_iv_{y_0,\lambda}(y)-y_n^{-2s}\p_iv_{y_0}(y)=\frac{\p_iw_{x_0,\lambda^2}(x)}{\p_nw_{x_0,\lambda^2}(x)}\big|_{x=(T^{w_{x_0,\lambda^2}})^{-1}(y)}-(-\p_ig(y_0)).
\end{align*}
By the boundary Harnack inequality, for any $y\in \mathcal{C}^+_1$
\begin{align*}
\quad \left|\frac{\p_iw_{x_0,\lambda^2}(x)}{\p_nw_{x_0,\lambda^2}(x)}\big|_{x=(T^{w_{x_0,\lambda^2}})^{-1}(y)}-(-\p_ig(y_0))\right|\leq C\lambda^{2\alpha}|(T^{w_{x_0,\lambda^2}})^{-1}(y)-y_0|^{\alpha}.
\end{align*}
This gives the $L^\infty $ bound. \\
To prove the H\"older bound, noting that $-\p_ig(y_0)$ is constant, we have
\begin{align*}
&\quad \left|y_n^{-2s}\p_i(v_{y_0,\lambda}-v_{y_0})(y_1)-y_n^{-2s}\p_i(v_{y_0,\lambda}-v_{y_0})(y_2)\right|\\
&= \left|\frac{\p_iw_{x_0,\lambda^2}(x)}{\p_nw_{x_0,\lambda^2}(x)}\big|_{x=(T^{w_{x_0,\lambda^2}})^{-1}(y_1)}-\frac{\p_iw_{x_0,\lambda^2}(x)}{\p_nw_{x_0,\lambda^2}(x)}\big|_{x=(T^{w_{x_0,\lambda^2}})^{-1}(y_2)}\right|.
\end{align*}
By Proposition~\ref{prop:asymp1} (i) and by \eqref{eq:v0}, there exists $\lambda_0=\lambda_0(s)$ such that $\mathcal{C}_1^+\subset T^{w_{x_0,\lambda^2}}(A_+\cup A_-)$ for any $0<\lambda\leq \lambda_0$. We will consider the H\"older estimate in $T^{w_{x_0,\lambda^2}}(A_+)$ and $T^{w_{x_0,\lambda^2}}(A_-)$ separately. 
First for $y_1,y_2\in \mathcal{C}_1^+\cap T^{w_{x_0,\lambda^2}}(A_-)$, we rewrite
\begin{align*}
\frac{\p_iw_{x_0,\lambda^2}(x)}{\p_nw_{x_0,\lambda^2}(x)}=\frac{x_{n+1}^{-2s}\p_iw_{x_0,\lambda^2}(x)}{x_{n+1}^{-2s}\p_nw_{x_0,\lambda^2}(x)}.
\end{align*} 
Using Proposition~\ref{prop:asymp2} and the fact that the denominator is uniformly bounded away from zero for $x\in A_-$ we have that for any $\gamma\in(0,1)$
\begin{align*}
\left|y_n^{-2s}\p_i(v_{y_0,\lambda}-v_{y_0})(y_1)-y_n^{-2s}\p_i(v_{y_0,\lambda}-v_{y_0})(y_2)\right|\\
\leq C_s\lambda^{2\alpha}|(T^{w_{x_0,\lambda^2}})^{-1}(y_1)-(T^{w_{x_0,\lambda^2}})^{-1}(y_2)|^{\gamma}.
\end{align*}
By Proposition~\ref{prop:inverse_T} we thus conclude that
\begin{align}\label{eq:holder}
\left|y_n^{-2s}\p_i(v_{y_0,\lambda}-v_{y_0})(y_1)-y_n^{-2s}\p_i(v_{y_0,\lambda}-v_{y_0})(y_2)\right|
\leq C_s\lambda^{2\alpha}|y_1-y_2|^{\gamma}.
\end{align}
For $y_1,y_2\in \mathcal{C}_1^+\cap T^{w_{x_0,\lambda^2}}(A_+)$, we directly estimate the quotient $\frac{\p_iw_{x_0,\lambda^2}(x)}{\p_nw_{x_0,\lambda^2}(x)}$. Arguing similarly as above, we conclude that \eqref{eq:holder} also holds in this case. Combining these bounds we deduce
\begin{align*}
\|y_n^{-2s}\p_i(v_{y_0,\lambda}-v_{y_0})\|_{C^{0,\gamma}(\mathcal{C}_1^+)}\leq C_s\lambda^{2\alpha}.
\end{align*}

\emph{Step 2: Estimates for the remaining two expressions.} The remaining two estimates are shown in a similar way. More precisely, we observe that
\begin{align*}
\left|y_n^{1-2s}\p_nv_{y_0,\lambda}(y)-y_n^{1-2s}\p_{n}v_{y_0}(y)\right|=(2s)\left|((T^{w_{x_0,\lambda^2}})^{-1}(y))_n-((T^{w_{x_0}})^{-1}(y))_n\right|.
\end{align*}
Using Proposition~\ref{prop:inverse_T}, we obtain the desired bound for $y_n^{1-2s}\p_nv$. \\
Next we consider the term $y_n^{-2s}y_{n+1}^{-1}\p_{n+1}v$. Using \eqref{eq:est23}, the asymptotics of $x_{n+1}^{2s}(y)$ from Lemma~\ref{lem:reverse} (i), and noting that $y_n^{-2s}y_{n+1}^{-1}\p_{n+1}v_{y_0}(y)=2a_1(y_0)$, we immediately infer that
\begin{align*}
\left\|y_n^{-2s}y_{n+1}^{-1}\p_{n+1}v_{y_0,\lambda}-y_n^{-2s}y_{n+1}^{-1}\p_{n+1}v_{y_0}\right\|_{L^\infty(\mathcal{C}_1^+)}\leq C_s\lambda^{2\alpha}.
\end{align*}
To show the H\"older continuity, we write
\begin{align*}
y_n^{-2s}y_{n+1}^{-1}\p_{n+1}v_{y_0,\lambda}(y)&=y_{n+1}^{-2s}\left(\frac{x_{n+1}}{y_n}\right)^{2s}\big|_{x=(T^{w_{x_0,\lambda^2}})^{-1}(y)} \text{ if } y\in T^{w_{x_0,\lambda^2}}(A_-),\\
y_n^{-2s}y_{n+1}^{-1}\p_{n+1}v_{y_0,\lambda}(y)&= y_n^{-2s}\left(\frac{x_{n+1}}{y_{n+1}}\right)^{2s} \big|_{x=(T^{w_{x_0,\lambda^2}})^{-1}(y)} \text{ if } y\in T^{w_{x_0,\lambda^2}}(A_+).
\end{align*}
By the proof of Proposition~\ref{prop:J}, we have $(x_{n+1}/y_n)^{2s}\in C^{0,\gamma}(A_-)$ and $(x_{n+1}/y_{n+1})^{2s}\in C^{0,\gamma}(A_+)$. Thus, arguing similarly as for (i) (where we use that in $T^{w_{x_0,\lambda^2}}(A_-)$, $y_{n+1}\sim 1$ and that in $T^{w_{x_0,\lambda^2}}(A_+)$, $y_{n}\sim 1$), we infer that for any $\gamma\in(0,1)$ and for all $y_1,y_2\in \mathcal{C}_1^+$
\begin{align*}
&\left|y_n^{-2s}y_{n+1}^{-1}\p_{n+1}v_{y_0,\lambda}(y_1)-y_n^{-2s}y_{n+1}^{-1}\p_{n+1}v_{y_0,\lambda}(y_2)\right|\\
&\leq C_s\lambda^{2\alpha}\left|(T^{w_{x_0,\lambda^2}})^{-1}(y_1)-(T^{w_{x_0,\lambda^2}})^{-1}(y_2)\right|^{\gamma}.
\end{align*}
Therefore, the estimate follows by invoking Proposition~\ref{prop:inverse_T}.
This concludes the proof.
\end{proof}

Similarly as in Section \ref{sec:asymp}, it is possible to extend these asymptotics to second order estimates:

\begin{prop}
\label{prop:asymp_v2}
Let $v_{y_0,\lambda}$, $v_{y_0}$, $\lambda_0$ and $\mathcal{C}_1^+$ be as in Proposition~\ref{prop:asymp_v1}. Then for any $\lambda\in (0,\lambda_0)$ and any $\gamma\in (0,1)$, $\p_iy_n^{1-2s}\p_jv_{y_0,\lambda}\in C^{0,\gamma}(\mathcal{C}_1^+)$ with $i,j\in\{1,\dots, n+1\}$. Moreover,
\begin{align*}
\sum_{i,j=1}^{n+1}\left\|\p_iy_n^{1-2s}\p_j\left(v_{y_0,\lambda}-v_{y_0}\right)\right\|_{C^{0,\gamma}(\mathcal{C}_1^+)}\leq C_s\lambda^{2\alpha}.
\end{align*}
\end{prop}

\begin{proof}
By \eqref{eq:inverse} a direct computation gives
\begin{align*}
\p_n(y_n^{1-2s}\p_n v)&=-(2s)(DT^{-1})_{n,n},\\
y_n^{1-2s}\p_{n+1}\p_{n}v&=-(2s)(DT^{-1})_{n,n+1},\\
y_n^{1-2s}\p_{n+1,n+1}v&=(2s)(y_{n+1}^{-1}y_n^{-2s}\p_{n+1}v)^{-\frac{2s-1}{2s}}(DT^{-1})_{n+1,n+1}\\
&\quad +(1-2s)y_n^{1-2s}y_{n+1}^{-1}\p_{n+1}v,\\
\p_{i}(y_n^{1-2s}\p_nv)&=(2s) (DT^{-1})_{n,i},\\
\p_i(y_n^{1-2s}\p_{n+1}v)
&=(2s)(y_n^{-2s}y_{n+1}^{-1}\p_{n+1}v)^{\frac{2s-1}{2s}}(DT^{-1})_{n+1,i}.
\end{align*}
Here $v$ stands representatively for $v_{y_0,\lambda}$ and $v_{y_0}$, and $T$ for $T^{w_{x_0,\lambda^2}}$, $T^{w_{x_0}}$, respectively ($x_0=(T^w)^{-1}(y_0)$). By invoking Proposition~\ref{prop:asymp_v1} and Proposition~\ref{prop:inverse_T}, this yields the estimates.\\
Finally, we compute that for $i,j\in \{1,\dots,n-1\}$,
\begin{align*}
y_n^{1-2s}\p_{ij}v&=y_n^{1-2s}\p_{ij}w(x)\big|_{x=T^{-1}(y)}\\
&+y_n^{1-2s}\p_{nj}w(x)(DT^{-1})_{n,j}+ y_n^{1-2s}\p_{n+1,j}w(x)(DT^{-1})_{n+1,j}\big|_{x=T^{-1}(y)}.
\end{align*}
We only show the estimate for $y_n^{1-2s}\p_{ij}w$ in detail, as the remainder of the proof is similar. Let
$$G(y):=y_n^{1-2s}\p_{ij}w(x)\big|_{x=(T^w)^{-1}(y)}.$$ 
If $y\in \mathcal{C}_1^+\cap T^w(A_+)$, we have $y_n\sim 1$. By Proposition~\ref{prop:asymp2} (ii) and Proposition~\ref{prop:inverse_T}, for $\lambda \in (0,\lambda_0)$, 
$$\|G_{\lambda}(y)-G_0(y)\|_{C^{0,\gamma}(T^{w_{x_0,\lambda^2}}(A_+)\cap \mathcal{C}_1^+)}\leq C_s\lambda^{2\alpha}.$$ 
Here $G_\lambda(y):=y_n^{1-2s}\p_{ij}w_{x_0,\lambda^2}(x),\ x=(T^{w_{x_0,\lambda^2}})^{-1}$ and and $G_0$ corresponds to $w_{x_0}$.\\
If $y\in \mathcal{C}_1^+\cap T^{w}(A_-)$, we write
\begin{align*}
G(y)&=\left(\frac{y_n}{x_{n+1}}\right)^{1-2s}(x_{n+1}^{1-2s}\p_{ij}w(x))\big|_{x=T^{-1}(y)}\\
&=(y_{n+1}^{1-2s}y_n^{-2s}\p_{n+1}v)^{\frac{2s-1}{2s}}(x_{n+1}^{1-2s}\p_{ij}w(x))\big|_{x=T^{-1}(y)}.
\end{align*}
By Proposition~\ref{prop:asymp_v1}, Proposition~\ref{prop:asymp2} and Proposition~\ref{prop:inverse_T} we obtain 
\begin{align*}
\|G_\lambda-G_0\|_{C^{0,\gamma}(T^{w_{x_0,\lambda^2}}(A_-)\cap \mathcal{C}_1^+)}\leq C\lambda^{2\alpha}.
\end{align*}
In the end, arguing similarly as in the proof of Proposition~\ref{prop:asymp_v1} (i), we obtain the desired estimate for $G_\lambda-G_0$ in $\mathcal{C}_1^+$.
This concludes the proof.
\end{proof}

Finally to conclude this section, we discuss the model solution $v_0$ which is defined as the blow-up of $v$ at $y=0$ and which (up to constants) is the Legendre function of the model solution $w_{1,s}$ from \eqref{eq:w1s}:

\begin{example}
\label{ex:model}
Recalling the assumptions that $g(0)=0=|\nabla''g(0)|$ yields that the Legendre function, $v_0$, of the blow-up $w_0$ at zero (c.f. (\ref{eq:scale_w})) is of a particularly simple form (compare this also to the more general expressions from Lemma \ref{lem:reverse} in Section \ref{sec:asymp1}). Up to rescaling it reads:
\begin{align}\label{eq:v0}
v_0(y)=-\frac{s}{(s+1)}y_n^{2s+2}+y_n^{2s}y_{n+1}^2.
\end{align}
Computing the expressions for $x$ and $J(v_0)$ in terms of $y$ (c.f. Proposition \ref{prop:eqnonlin}) in this particular case then yields
\begin{align*}
x_n(y)=\frac{1}{2}(y_n^2-y_{n+1}^2),\quad x_{n+1}(y)=y_ny_{n+1}, \quad 
J(v_0)= y_n^2+y_{n+1}^2. 
\end{align*}
Hence, in the case of vanishing inhomogeneity, $f=0$, up to a constant, the linearization of the nonlinear functional $F$ from (\ref{eq:nonlinear}) at $v_0$ is
\begin{align}
\D_{G,s} \tilde{v}& =(y_n y_{n+1})^{1-2s}(y_n^2+y_{n+1}^2)\p_{ii}\tilde{v} +\p_n((y_ny_{n+1})^{1-2s}\p_n\tilde{v})\label{eq:frac_1tional_Gru}\\
& \quad +\p_{n+1}((y_ny_{n+1})^{1-2s}\p_{n+1}\tilde{v}).\notag
\end{align}
This is a \emph{Baouendi-Grushin type fractional Laplacian}, which serves as a first motivation of the introduction of the Baouendi-Grushin geometry in the following section. 
\end{example} 

\section{Geometry and Function Spaces}
\label{sec:function_spaces}

Motivated by the linearization result from Example \ref{ex:model}, we introduce the geometry and function spaces in which we will be working in the sequel. More precisely, we consider the intrinsic geometry which is induced by the Baouendi-Grushin operator (c.f. Definition \ref{defi:BGgeo}). For the choice of our function spaces (c.f. Definition \ref{defi:XY}) we build on this. Guided by the explicit form of the model solution from Example \ref{ex:model} and the a priori estimates for the fractional Baouendi-Grushin operator (c.f. Sections \ref{sec:reg2}), we construct weighted spaces with the right asymptotic behavior at the straightened free boundary $P:=\{y_n=y_{n+1}=0\}$.\\

We begin by introducing the relevant geometric quantities in Definitions \ref{defi:BGgeo}.

\begin{defi}[Baouendi-Grushin geometry]
\label{defi:BGgeo}
Let 
\begin{align*}
Y_{1}&:= y_n \p_{y_1}, \ Y_{2}:= y_{n+1} \p_{y_1}, \ \dots, \  Y_{2n-3}:=y_n\p_{y_{n-1}},\ Y_{2n-2}:= y_{n+1}\p_{y_{n-1}}, \\
 Y_{2n-1}&:= \p_{y_n} , \ Y_{2n}:=\p_{y_{n+1}},
\end{align*}
denote the \emph{Baouendi-Grushin vector fields}. We consider the associated \emph{Baouendi-Grushin metric}
\begin{align*}
g_{y}(v,w) = (y_n^2 + y_{n+1}^2)^{-1} \left(\sum\limits_{i=1}^{n-1} w_i v_i \right) + v_n w_n + v_{n+1}w_{n+1},
\end{align*}
for all $y\in \R^{n+1}$ and $v,w\in \spa\{Y_i(y)\}_{i\in\{1,\dots,2n\}} $.
This induces a (Carnot-Caratheodory) distance on $\R^{n+1}$:
\begin{align*}
d_G(x,y)&:= \inf \left\{\int\limits_{a}^{b} \sqrt{g_{\gamma(t)}(\dot{\gamma}(t),\dot{\gamma}(t))}dt: \gamma(a)=x, \gamma(b)=y , \ a,b\in \R, \right.\\ 
& \quad \quad \quad \left. \dot{\gamma}(t) \in \spa\{Y_i(\gamma(t))\}_{i\in\{1,\dots,2n\}} \right\}.
\end{align*}
We denote the associated (closed) \emph{Baouendi-Grushin balls of radius $R$ and with center $y_0\in \R^{n+1}$} by $\mathcal{B}_R(y_0)$, i.e.
\begin{align*}
\mathcal{B}_R(y_0):=\{y\in \R^{n+1}: d_G(y,y_0)\leq R\}.
\end{align*}
Let 
$Q_+:=\{y\in\R^{n+1}:y_n\geq 0, y_{n+1}\geq 0\}$ denote the upper quarter space. 
We denote the intersection of balls with $Q_+$ by $\mathcal{B}^{+}_R(y_0)$.
\end{defi}

\begin{rmk}\label{rmk:quasi}
As in Remark 4.2 of \cite{KRSIII}, it is not hard to show that $d_G$ is equivalent to the following quasi-metric:
\begin{align*}
\tilde{d}_G(\hat{y},y):= |\hat{y}_n-y_n|+|\hat{y}_{n+1}-y_{n+1}| + \frac{|\hat{y}''-y''|}{|\hat{y}_n|+|\hat{y}_{n+1}|+|y_n|+|y_{n+1}|+|\hat{y}''-y''|^{1/2}}.
\end{align*}
Here $\hat{y}=(\hat{y}'',\hat{y}_n,\hat{y}_{n+1})$, $y=(y'',y_n,y_{n+1})$.
\end{rmk}

Using the previous notation, it is possible to define Hölder spaces with respect to the Baouendi-Grushin metric:

\begin{defi}[Hölder spaces]
\label{defi:hoelder}
Let $\Omega \subset \R^{n+1}$ and let $\alpha\in (0,1]$. Then we define
\begin{align*}
&[\cdot]_{\dot{C}^{0,\alpha}_{\ast}(\overline{\Omega})}: L^{\infty}(\bar{\Omega}) \rightarrow [0,\infty],\\
&[u]_{\dot{C}^{0,\alpha}_{\ast}(\overline{\Omega})}:= \sup\limits_{\hat{y},y\in \overline{\Omega}}\frac{|u(\hat{y})-u(y)|}{d_G(\hat{y},y)^{\alpha}}.
\end{align*}
Moreover, we set
\begin{align*}
&\|\cdot \|_{C^{0,\alpha}_{\ast}(\overline{\Omega})}: L^{\infty}(\bar{\Omega}) \rightarrow [0,\infty],\\
&\|u\|_{C^{0,\alpha}_{\ast}(\overline{\Omega})}:= \sup\limits_{\hat{y}\in \overline{\Omega}}|u(\hat{y})| + [u]_{\dot{C}^{0,\alpha}_{\ast}(\overline{\Omega})},
\end{align*}
and define
\begin{align*}
C^{0,\alpha}_{\ast}(\overline{\Omega}):=\{u \in L^{\infty}( \bar{\Omega} ):  \|u\|_{C^{0,\epsilon}_{\ast}(\overline{\Omega})}<\infty\}.
\end{align*}
\end{defi}

\begin{rmk}
As in the Euclidean case, the spaces $C^{0,\alpha}_\ast(\overline{\Omega})$ form Banach spaces.
\end{rmk}

In order to approximate functions with respect to the Baouendi-Grushin geometry, we rely on the notion of \emph{homogeneous} polynomials. These are polynomials whose tangential and normal degrees are counted differently. This is motivated by the different scaling behavior of the tangential and normal components of the operator from Example \ref{ex:model}.

\begin{defi}[Homogeneous polynomials]
\label{defi:homo_poly}
Let $k\in \N$ and let $\beta=(\beta_1,\dots, \beta_{n+1})$ with $\beta_i \in \N\cup\{0\}$ be a multi-index. We define
\begin{align*}
\mathcal{P}_k^{hom}:=\left\{p(y)=\sum_{|\beta|=k}b_\beta y^\beta: b_\beta\in \R, \ b_\beta=0 \text{ if } \sum\limits_{j=1}^{n-1}2\beta_i + \beta_n + \beta_{n+1}>k\right\}
\end{align*}
as the space of \emph{homogeneous polynomials with degree $k$}. Here the notion homogeneous refers to the scaling behavior
\begin{align*}
p_k(\delta_\lambda (y))=\lambda^k p_k(y), \quad p_k\in \mathcal{P}_k^{hom},
\end{align*}
where $\delta_\lambda(y)=(\lambda^2y'',\lambda y_n, \lambda y_{n+1})$ is the dilation associated with the Baouendi-Grushin vector fields. We define
\begin{align*}
\mathcal{P}_k:=\left\{p(y)= \sum\limits_{|\beta|\leq k} b_{\beta} y^{\beta}: \ b_{\beta}\in \R, \  b_{\beta}=0 \mbox{ if  } \sum\limits_{j=1}^{n-1}2\beta_i + \beta_n + \beta_{n+1}>k\right\},
\end{align*}
as the vector space of \emph{the homogeneous polynomials with degree less or equal to $k$}.
\end{defi}

Finally, as the last ingredient before defining our function spaces, we introduce the notion of \emph{differentiability at $P:=\{y_n=y_{n+1}=0\}$}:

\begin{defi}
\label{defi:CkP}
We say that a function $f:Q_+ \rightarrow \R$ is \emph{$C^{k,\alpha}_{\ast}$ at $P$} if for each point $y_0\in P$ there exists a homogeneous polynomial $P_{y_0, k} \in \mathcal{P}_k$ such that
\begin{align*}
|f(y)-P_{y_0,k}(y)| \leq C d_G(y,y_0)^{k+2\alpha} \mbox{ in } \mathcal{B}_1^+(y_0).
\end{align*}
We call the polynomial $P_{y_0,k}$ an \emph{approximating polynomial} for $f$ at $y_0$.
\end{defi}

\begin{rmk}
We emphasize that the previous notion does not imply the differentiability of $f$ at $P$ up to order $k$. If however $f$ is $k$-times classically differentiable at a point $y_0\in P$, then the approximating polynomial $P_{y_0,k}$ corresponds to the Taylor polynomial (of homogeneous degree less than or equal to $k$) of $f$ at $P$.
\end{rmk}

With the previous preparation, we can now define our main function spaces which are needed for the application of the implicit function theorem in Section~\ref{sec:IFT}. We are seeking Banach spaces $X$ and $Y$ such that 
\begin{itemize}
\item[(i)] the nonlinear functional $F$ in \eqref{eq:nonlinear} is smooth (or analytic) from $X$ to $Y$, if the inhomogeneity $f$ is smooth (or analytic),
\item[(ii)] the model solution $v_0$ is contained in $X$, and $\Delta_{G,s}=D_{v_0}F: X\rightarrow Y$ is invertible,
\item[(iii)] the Legendre function $v$ from Section~\ref{sec:HL} is in $X$. Morevoer,  the differential $D_vF$ is a perturbation of $\Delta_{G,s}: X\rightarrow Y$.
\end{itemize}

\begin{defi}[Function spaces]
\label{defi:XY}
Let $\alpha,\epsilon \in (0,1)$. We set
\begin{align*}
X_{\alpha,\epsilon} & := \{v\in L^{\infty}(Q_+) : \ v= 0 \mbox{ on } \{y_n=0\}, \ y_n^{-2s}v \in C^{2,\alpha}_{\ast} \mbox{ at } P,\\ 
& \quad \lim\limits_{y_{n+1} \rightarrow 0_+} \p_{n+1}(y_n^{-2s}  v(y) )= 0, \  \lim\limits_{y \rightarrow P} \p_n(y_n^{1-2s} \p_n v(y)) = 0,
 \ \supp(\D_{G,s} v)\subset \mathcal{B}_1^{+}, \\
& \quad  \|v\|_{X_{\alpha,\epsilon}}<\infty  \},\\
Y_{\alpha,\epsilon} & := \{f\in L^\infty(Q_+): y_{n+1}^{2s-1}f\in C^{1,\alpha}_{\ast} \mbox{ at } P, \ \lim\limits_{y \rightarrow P}y_{n+1}^{2s-1}f(y)=0=\lim\limits_{y \rightarrow P}\p_{n+1}(y_{n+1}^{2s-1} f(y)),\\
&\quad   \supp(f)\subset \mathcal{B}_1^{+},  \ \|f\|_{Y_{\alpha,\epsilon}}<\infty  \},
\end{align*}
where the associated norms are given by
\begin{align*}
\|v\|_{X_{\alpha,\epsilon}}&=\sup_{\bar y\in P}\left( \left\|d_G(\cdot, \bar y)^{-(2+2\alpha)}y_n^{-2s}(v- y_{n}^{2s}P^s_{\bar y,2})\right\|_{L^\infty(Q_+ )} \right.\\
& + \left[ d_G(\cdot, \bar y)^{-2\alpha+\epsilon}y_{n+1}^{-1}y_{n}^{-2s}\p_{n+1}(v- y_{n}^{2s}P^s_{\bar y,2})\right]_{\dot{C}^{0,\epsilon}_{\ast}(Q_+ )} \\
& + \left[d_G(\cdot, \bar y)^{-(2+2\alpha-\epsilon)}y_{n}^{1-2s}\p_{n}(v- y_{n}^{2s}P^s_{\bar y,2})\right]_{\dot{C}^{0,\epsilon}_{\ast}(Q_+ )} \\
& + \sum\limits_{i=1}^{n-1}\left[d_G(\cdot, \bar y)^{-(2\alpha-\epsilon)}y_n^{-2s}\p_{i}( v-y_{n}^{2s} P^s_{\bar y,2})\right]_{\dot{C}^{0,\epsilon}_{\ast}(Q_+ )} \\
&\left.+\sum\limits_{i,j=1}^{n+1}\left[ d_G(\cdot, \bar y)^{-(1+2\alpha-\epsilon)}  Y_iy_n^{1-2s}Y_j(v- y_{n}^{2s} P^s_{\bar y,2})\right]_{\dot{C}^{0,\epsilon}_\ast(Q_+)} \right),\\
\|f\|_{Y_{\alpha,\epsilon}}&=\sup_{\bar y\in P}\left[d_G(\cdot, \bar y)^{-(1+2\alpha-\epsilon)} y_{n+1}^{2s-1} (f- y_{n+1}^{1-2s} Q_{\bar y,1}^s)\right]_{\dot{C}^{0,\epsilon}_\ast(Q_+)} .
\end{align*}
The functions $P_{\bar{y},2}^s$ and $Q_{\bar{y},1}^s$ denote the respective (in the homogeneous sense) second and first order approximating polynomials (in the sense of Definition \ref{defi:CkP}) of $y_{n}^{-2s} v(y)$ and $y_{n+1}^{2s-1}f(y)$ at $\bar{y}:=(\bar{y}'',0,0)\in P$.
\end{defi}

Let us briefly comment on the main ideas leading to these definitions. The spaces are constructed so as to measure the deviation of functions from suitable approximations at the boundary of $Q_+$. In this sense they mimic the asymptotic expansions of the Legendre function $v$ (for the definition of the space $X_{\alpha,\epsilon}$) and of the function $\D_{G,s}v$ (for the space $Y_{\alpha,\epsilon}$), c.f. Section \ref{sec:asymp1}. The asymptotic behavior at the boundary of $Q_+$ is thus encoded by considering the difference of $v$ to $y_n^{2s}P^s_{\bar{y},2}$ and $y_{n+1}^{1-2s}Q_{\bar{y},1}^s$. These specific approximations are motivated by the structure of the ``eigenpolynomials'' to the fractional Baouendi-Grushin Laplacian and the associated elliptic regularity estimates (c.f. Sections \ref{sec:reg1}-\ref{sec:Eigenf}). The existence of such an approximation is ensured by the requirement that $y_n^{-2s} v \in C^{2,\alpha}_{\ast}$ at $P$. The choice of the norms rests on the availability of ``Schauder type'' elliptic estimates for the Grushin Laplacian with mixed Dirichlet-Neumann conditions with respect to these (c.f. Sections \ref{sec:reg1} and \ref{sec:reg2}). \\
The pointwise conditions imposed on functions in the space $X_{\alpha,\epsilon}$ are a mixture of boundary and normalization properties: We require Dirichlet conditions on $\{y_n=0\}$ and a ``strengthened'' form of Neumann conditions on $\{y_{n+1}=0\}$ (c.f. Remark \ref{rmk:bound_data}). The Neumann condition in particular rules out the presence of a linear term $y_{n+1}$ in the asymptotic expansion of $y_n^{-2s}v$ at $P$ and is hence adapted to the expansion of $y_n^{-2s}v$.
Finally, the remaining pointwise condition is a normalization which excludes the presence of the linear term $y_n$ in the approximating polynomial for $y_n^{-2s}v$ at $P$. \\
The requirement that $v\in L^{\infty}$ in combination with the compact support condition on $\D_{G,s}v$ entails that the space $X_{\alpha,\epsilon}$ is a Banach space.

\begin{rmk}
\label{rmk:norms}
Restrictions on the specific form of the weights (which are used in the norms) originate from the second order estimates (c.f. Propositions \ref{prop:Dirichlet}, \ref{prop:Neumann}), the compatibility with the linear and nonlinear operators, and the aim of proving analyticity of the nonlinear function (\ref{eq:nonlinear}) in the function spaces $X_{\alpha,\epsilon}, Y_{\alpha,\epsilon}$. We remark that this still leaves a non-negligible amount of freedom for instance in the exact choice of the weights for the lowest order contributions.
\end{rmk}

\begin{rmk}
\label{rmk:supp}
Due to the compact support of $f$ and due to the definition of $P_{\bar{y},1}^s$, we have that
\begin{align*}
&\left\|d_G(\cdot, \bar y)^{-(1+2\alpha)} y_{n+1}^{2s-1} (f-P_{\bar y,1}^s) \right\|_{L^\infty(Q_+)} \\
&\leq  C
\left[d_G(\cdot, \bar y)^{-(1+2\alpha-\epsilon)}y_{n+1}^{2s-1} (f-P_{\bar y,1}^s)\right]_{\dot{C}^{0,\epsilon}_\ast (Q_+)}. 
\end{align*}
\end{rmk}

As in \cite{KRSIII} these function spaces have local analoga:

\begin{defi}
\label{defi:XYloc}
Let $\alpha,\epsilon \in (0,1)$. We set
\begin{align*}
X_{\alpha,\epsilon}(\mathcal{B}^{+}_R) & 
 := \{v\in L^{\infty}(\mathcal{B}_R^+) : \ v= 0 \mbox{ on } \{y_n=0\}\cap \mathcal{B}_R^+, \ y_n^{-2s}v \in C^{2,\alpha}_{\ast} \mbox{ at } P\cap \mathcal{B}_R^+,\\ 
& \quad \lim\limits_{y_{n+1} \rightarrow 0_+} \p_{n+1}(y_n^{-2s} v(y)) =0\mbox{ on } \mathcal{B}_R^+, \ \lim\limits_{y \rightarrow P\cap \mathcal{B}_R^+} \p_n(y_n^{1-2s}\p_n v(y)) = 0, \\
& \quad   \|v\|_{X_{\alpha,\epsilon}(\mathcal{B}_R^+)}<\infty   \},\\
Y_{\alpha,\epsilon}(\mathcal{B}^{+}_R )& := \{f: Q_+ \rightarrow \R:  y_{n+1}^{2s-1}f\in C^{1,\alpha}_{\ast} \mbox{ at } P\cap \mathcal{B}_R^+,\\
& \quad \ \lim\limits_{y \rightarrow P \cap \mathcal{B}_R^+}y_{n+1}^{2s-1}f(y)=0=\lim\limits_{y \rightarrow P \cap \mathcal{B}_R^+}\p_{n+1}(y_{n+1}^{2s-1} f(y)), \ \|f\|_{Y_{\alpha,\epsilon}(\mathcal{B}_R^+)}<\infty  \},
\end{align*}
where the associated norms are defined as above but now contain the full $C^{0,\epsilon}_{\ast}$ norms, e.g. $$\|v\|_{X_{\alpha,\epsilon}(\mathcal{B}_R^+)}=\|v\|_{\dot{X}_{\alpha,\epsilon}(\mathcal{B}_R^+)}+\sup_{\bar y\in P\cap \mathcal{B}_R^+} |P^s_{\bar y,2}|. $$ 
Here $\|\cdot\|_{\dot{X}_{\alpha,\epsilon}(\mathcal{B}_R^+)}$ is the homogeneous part which is defined the same way as for the global spaces (c.f. Definition~\ref{defi:XY}) with $Q_+$ replaced by $\mathcal{B}_R^+$,  and $|P^s_{\bar y,2}|:=\sum_{|\beta|\leq 2}|b_\beta(\bar y)|$ for $P^s_{\bar y,2}(y)=\sum_{|\beta|\leq 2} b_\beta (\bar y) y^\beta$. 
\end{defi}

As in \cite{KRSIII}, the function spaces from Definition \ref{defi:XY} have a characterization in terms of decompositions in appropriate Hölder spaces: 

\begin{prop}
\label{prop:decomp}
Let $\alpha,\epsilon\in(0,1)$ with $\epsilon\leq \alpha$ and let $X_{\alpha,\epsilon}, Y_{\alpha,\epsilon}$ be the function spaces from Definition \ref{defi:XY}. Then,\\
\emph{(i)} $v\in X_{\alpha,\epsilon}$ if and only if there exist functions $a_0, a_1\in C^{0,\alpha}(\R^{n-1})$, $c_0\in C^{1,\alpha}(\R^{n-1})$ and $ C_0, C_k, C_{k\ell}\in C^{0,\epsilon}_\ast(Q_+)$ with $k,\ell \in \{1,\dots,n+1\}$ such that
\begin{itemize}
\item[(a)] for $i,j\in\{1,\dots,n-1\}$ the following decomposition holds:
\begin{align*}
v(y) & = c_0(y'')y_n^{2s} + a_0(y'')y_n^{2+2s} + a_1(y'')y_n^{2s}y_{n+1}^2 \\
& \quad + y_n^{2s}r^{2+2\alpha-\epsilon}C_{0}(y),\\
y_n^{1-2s}\p_nv(y)&=(2s)c_0(y'')+(2+2s)a_0(y'')y_n^2+(2s)a_1(y'')y_{n+1}^2\\
& \quad +r^{2+2\alpha-\epsilon}C_n(y),\\
\p_{n+1}v(y)&=2a_1(y'')y_n^{2s}y_{n+1}+ y_n^{2s}y_{n+1}r^{2\alpha-\epsilon}C_{n+1}(y),\\
\p_iv(y)&=\p_i c_0(y'')y_n^{2s}+y_n^{2s}r^{2\alpha-\epsilon}C_i(y),\\
\p_n(y_n^{1-2s}\p_nv)(y)&=2(2+2s)a_0(y'')y_n+r^{1+2\alpha-\epsilon}C_{nn}(y),\\
\p_{n+1}(y_n^{1-2s}\p_nv)(y)&=(4s)a_1(y'')y_{n+1}+r^{1+2\alpha-\epsilon}C_{n+1,n}(y),\\
y_n^{1-2s}\p_{n,n+1}v(y)&=(4s)a_1(y'')y_{n+1}+r^{1+2\alpha-\epsilon}C_{n,n+1}(y),\\
y_{n}^{1-2s}\p_{n+1,n+1}v(y)&=2a_1(y'')y_n + r^{1+2\alpha-\epsilon}C_{n+1,n+1}(y),\\
y_{n}^{1-2s}\p_{ij}v(y)&= r^{-1+2\alpha-\epsilon}C_{ij}(y),\\
y_n^{1-2s}\p_{in}v(y)&=(2s)\p_ic_0(y'')+r^{2\alpha-\epsilon}C_{in}(y).
\end{align*}
\item[(b)] The following estimate holds:
\begin{align*}
&[a_0]_{\dot{C}^{0,\alpha}}+[a_1]_{\dot{C}^{0,\alpha}} + [c_0]_{\dot{C}^{1,\alpha}}+
\sum\limits_{i=0}^{n+1}[C_i]_{\dot{C}^{0,\epsilon}_{\ast}}+ \sum\limits_{i,j=1}^{n+1}[C_{ij}]_{\dot{C}^{0,\epsilon}_{\ast}}\leq C \|v\|_{X_{\alpha,\epsilon}}.
\end{align*}
\item[(c)] The functions $C_0, C_{i}, C_{ij}$, $i\in\{1,\dots,n+1\}$ vanish on $P$ and $C_{n+1}$ vanishes on $\{y_n=0\}$.
\item[(d)] $\supp(\D_{G,s})v \in \mathcal{B}_1^+$.
\end{itemize}
\emph{(ii)} $f\in Y_{\alpha,\epsilon}$ if and only if there exist functions $f_0 \in C^{0,\alpha}(P)$, $f_1\in C^{0,\epsilon}_{\ast}(Q_+)$ such that
\begin{align*}
f(y) =y_n y_{n+1}^{1-2s}  f_0(y'')  + y_{n+1}^{1-2s} r^{1+2\alpha-\epsilon} f_1(y),
\end{align*}
with
\begin{align*}
[f_0]_{\dot{C}^{0,\alpha}} + [f_1]_{\dot{C}^{0,\epsilon}_{\ast}} \leq C \|f\|_{Y_{\alpha,\epsilon}},
\end{align*} 
$f_1(y)=0$ for $y\in P$ and $\supp(f_0), \supp(f_1)\subset \mathcal{B}_1''\times \R^{2}_+$.
\end{prop}

\begin{rmk}\label{rmk:equiv_local}
For the local spaces $X_{\alpha,\epsilon}(\mathcal{B}_R^+)$ and $Y_{\alpha,\epsilon}(\mathcal{B}_R^+)$, there are similar characterizations. One has the equivalence of the norms
\begin{align*}
\|a_0\|_{C^{0,\alpha}(\mathcal{B}_R^+\cap P)}+ \|a_1\|_{C^{0,\alpha}(\mathcal{B}_R^+\cap P)}+\|c_0\|_{C^{1,\alpha}(\mathcal{B}_R^+\cap P)}\\
+\sum_{i=1}^{n+1}\|C_i\|_{C^{0,\epsilon}_\ast(\mathcal{B}_R^+)}+\sum_{i,j=1}^{n+1}\|C_{ij}\|_{C^{0,\epsilon}_\ast(\mathcal{B}_R^+)}\sim \|v\|_{X_{\alpha,\epsilon}(\mathcal{B}_R^+)}.
\end{align*}
\end{rmk}

Using these definitions and characterizations, we note that for $v\in X_{\alpha,\epsilon}$, it follows that $\D_{G,s}v\in Y_{\alpha,\epsilon}$.  Here $\D_{G,s}$ is the fractional Baouendi-Grushin Laplacian defined in \eqref{eq:frac_1tional_Gru}. Moreover, the decomposition from Proposition \ref{prop:decomp} can be used to prove that $X_{\alpha,\epsilon},Y_{\alpha,\epsilon}$ form Banach spaces (c.f. Section \ref{sec:Banach}):

\begin{prop}
\label{prop:Banach}
Let $\alpha,\epsilon\in(0,1)$ with $\alpha\geq \epsilon$ and let $X_{\alpha,\epsilon}, Y_{\alpha,\epsilon}$ be as in Definition \ref{defi:XY}. Then, $X_{\alpha,\epsilon}, Y_{\alpha,\epsilon}$ are Banach spaces. 
\end{prop}

Next we state the following a priori estimate, whose proof is provided in the Appendix (c.f. Proposition~\ref{prop:map2}):

\begin{prop}
\label{prop:map}
Suppose that $v\in X_{\alpha, \epsilon}$ for some $\alpha, \epsilon\in (0,1)$. Assume that $\Delta_{G,s}v=f$ for some $f\in Y_{\alpha, \epsilon}$, where $\D_{G,s}$ is the fractional Baouendi-Grushin Laplacian in \eqref{eq:frac_1tional_Gru}. Then, 
\begin{align*}
\|v\|_{X_{\alpha, \epsilon}(\mathcal{B}_{1/2}^+)}\leq C\left(\|f\|_{Y_{\alpha, \epsilon}(\mathcal{B}_1^+)}+\|v\|_{L^\infty(\mathcal{B}_1^+)}\right).
\end{align*}
\end{prop}

Last but not the least, we show that the Legendre function $v$ associated with a solution $w$ to \eqref{eq:fracLa_2} satisfying the assumptions (A1)-(A4) is in $X_{\alpha,\epsilon}(\mathcal{B}_{\delta_0}^+)$ for some small $\delta_0>0$ which only depends on $s$.

\begin{cor}[Regularity of the Legendre function]
\label{cor:reg_v}
Let $v$ be the Legendre function associated with a solution $w$ of the thin obstacle problem \eqref{eq:fracLa_2}. Then there exists  a small radius $\delta_0>0$ depending on $s$, such that for any $\alpha$ as in Proposition \ref{prop:asymp1} and $0<\epsilon\leq \alpha$ we have $v\in X_{\alpha,\epsilon}(\mathcal{B}^+_{\delta_0})$.
\end{cor}

\begin{proof}
This follows immediately from the local version of the characterization of the space $X_{\alpha,\epsilon}$ in Proposition~\ref{prop:decomp}, from Lemma~\ref{lem:reverse} (ii), from Propositions \ref{prop:asymp_v1} and Proposition \ref{prop:asymp_v2} by scaling, i.e. by setting $\lambda = \sqrt{y_n^2 + y_{n+1}^2}$. Here we have used that in $\mathcal{C}_1^+$, where $\sqrt{y_n^2+y_{n+1}^2}\sim 1$, the Baouendi-Grushin metric is equivalent to the Euclidean metric.
\end{proof}

\section{Mapping Properties}
\label{sec:mapping_prop}
In this section we discuss the mapping properties of the nonlinear function $F$ from (\ref{eq:nonlinear}) (c.f. Section \ref{sec:nonlinear}) and of its linearization (c.f. Section \ref{sec:lin}). In particular, we prove that $F$ is analytic as a function from a subset of $X_{\alpha,\epsilon}$ to $Y_{\alpha,\epsilon}$ for a suitable choice of $\alpha,\epsilon$, if the inhomogeneity $f$ is also analytic (c.f. Proposition \ref{prop:analyticity}).
These mapping properties are necessary conditions for the application of the implicit function theorem in Section~\ref{sec:IFT}. They a posteriori justify the choice of our spaces from Section \ref{sec:function_spaces} and make the intuition from Example \ref{ex:model} rigorous. 

Throughout the section, we assume that the conditions (A1)-(A4) from Section \ref{sec:setup} hold.

\subsection{Mapping properties of the nonlinear functional}
\label{sec:nonlinear}

Let $F(D^2v,Dv,y)$ be the nonlinear functional from \eqref{eq:nonlinear}. 
For convenience, we abbreviate it as $F(v):=F(D^2v,Dv,y)$. Let 
$$v_0(y)= -\frac{s}{s+1} y_n^{2+2s} + y_n^{2s} y_{n+1}^2$$ 
be the model solution from \eqref{eq:v0}. We note that 
$$\|v_0\|_{X_{\alpha,\epsilon}(\mathcal{B}_1^+)}=1+\frac{s}{s+1},$$ 
and observe that $F$ is well-defined in a small $X_{\alpha,\epsilon}$-neighborhood of $v_0$. \\
In the sequel, given $u\in X_{\alpha,\epsilon}(\mathcal{B}_1^+)$ or $u\in X_{\alpha,\epsilon}$, we denote the $r-$neighborhood of $u$ in the corresponding Banach space by
$$\mathcal{U}_{r}(u):=\{v\in X_{\alpha,\epsilon}(\mathcal{B}_1^+):\|v-u\|_{X_{\alpha,\epsilon}(\mathcal{B}_1^+)}<r\}, \ 0<r<\infty .$$ 
We will show that there exists $r_0\in (0,\frac{1}{4}\|v_0\|_{X_{\alpha,\epsilon}}]$ such that the nonlinear functional $F$ defined in \eqref{eq:nonlinear} maps a neighborhood, $\mathcal{U}_{r_0}(v_0)$, of $v_0$ in $X_{\alpha,\epsilon}(\mathcal{B}_1^+)$ into $Y_{\alpha,\epsilon}(\mathcal{B}_{1}^+)$. Moreover, we discuss the analyticity (or smoothness) properties of $F$ as a mapping from $\mathcal{U}_{r_0}(v_0)$ into $Y_{\alpha,\epsilon}(\mathcal{B}_1^+)$. Here a major difficulty is that the functional $F$ contains the term $(y_{n+1}^{2s-1}\p_{n+1}v)^{\frac{1-2s}{s}}$, which, if $s\neq \frac{1}{2}$, is not analytic in standard Hölder spaces. This however is overcome by the use of the function spaces from Section \ref{sec:function_spaces}.

\begin{prop}
\label{prop:map_non}
Let $F$ be the nonlinear functional from Proposition \ref{prop:eqnonlin}. Assume that the inhomogeneity $f\in C^{0,1}(B_1^+)$. 
Let $\alpha \in (0,1)$ and $\epsilon\in(0, \alpha]$ and let  $X_{\alpha,\epsilon}(\mathcal{B}_1^+), Y_{\alpha,\epsilon}(\mathcal{B}_1^+)$ be the spaces from Definition \ref{defi:XY}.  
Then for any $r_0\in \left(0,\frac{1}{4}\|v_0\|_{X_{\alpha,\epsilon}}\right]$, 
\begin{align*} 
X_{\alpha,\epsilon}(\mathcal{B}_1^+) \supset \mathcal{U}_{r_0}(v_0) \ni v \mapsto F(v)\in Y_{\alpha,\epsilon}(\mathcal{B}_1^+).
\end{align*}
\end{prop}

\begin{proof}
The argument follows by inserting the characterization of the function spaces from Proposition \ref{prop:decomp} into the expression for $F$. We concentrate on dealing with the $(y_{n+1}^{2s-1}\p_{n+1}v)^{\frac{1-2s}{s}} $ contribution in the equation and the inhomogeneity, as these are the only non-standard terms.\\ 
We begin with the $(y_{n+1}^{2s-1}\p_{n+1}v)^{\frac{1-2s}{s}} $ contribution and observe that 
\begin{align*}
y_{n+1}^{2s-1}\p_{n+1}v_0(y)= 2y_n^{2s}y_{n+1}^{2s}.
\end{align*}
Thus, suppose that $v\in X_{\alpha,\epsilon}(\mathcal{B}_1^+)\cap \mathcal{U}_{r_0}(v_0)$ with $0<r_0\leq \frac{1}{4}\|v_0\|_{X_{\alpha,\epsilon}}$, and suppose (by Proposition~\ref{prop:decomp}) that $v$ has the decomposition
$$
y_{n+1}^{2s-1}\p_{n+1}v(y)=2a_1(y'')y_n^{2s}y_{n+1}^{2s}+C_{n+1}(y)y_n^{2s}y_{n+1}^{2s}r^{2\alpha-\epsilon}.
$$
Then Remark~\ref{rmk:equiv_local} yields that
$\frac{3}{4}\leq a_1(y'')\leq  \frac{5}{4}$ and $|C_{n+1}(y)|\leq \frac{1}{4}$ for any $y\in \mathcal{B}_1^+$. This implies that $|C_{n+1}(y)/a_1(y'')|\leq \frac{1}{3}$. Thus for any $v\in \mathcal{U}_{r_0}(v_0)$,
\begin{align*}
(y_{n+1}^{2s-1}\p_{n+1}v)^{\frac{1-2s}{s}} &= y_n^{2-4s} y_{n+1}^{2-4s} (a_1(y'') +C_{n+1}(y) r^{2\alpha - \epsilon})^{\frac{1-2s}{s}}\\
&= y_n^{2-4s} y_{n+1}^{2-4s}a_1(y'')^{\frac{1-2s}{s}} \left(1 +\frac{C_{n+1}(y)}{a_1(y'')} r^{2\alpha - \epsilon}\right)^{\frac{1-2s}{s}}\\
&=y_n^{2-4s} y_{n+1}^{2-4s} (\tilde{a}_1(y'') +\tilde{C}(y)r^{2\alpha-\epsilon}),
\end{align*}
where $\tilde{a}_1=a_1^{\frac{1-2s}{s}}\in C^{0,\alpha}(P\cap \mathcal{B}_1^+)$ and $\tilde{C}\in C^{0,\epsilon}_\ast(\mathcal{B}_1^+)$. Here we used the analyticity of the function $t\mapsto (1+t)^\frac{1-2s}{s}$ for $|t|<\frac{1}{2}$.
Using this and the asymptotics from Proposition \ref{prop:decomp}, we then obtain the desired mapping property.\\
In order to deal with the ``inhomogeneity'', i.e. with the term 
$$(y_{n+1}^{2s-1}\p_{n+1}v)^{\frac{3-2s}{2s}}J(v)f(y'',x_n(y),x_{n+1}(y)),$$ 
we make use of the asymptotics from Proposition~\ref{prop:decomp} and the choice of $\alpha $ and $\epsilon$, i.e. $\epsilon\leq \alpha$:
\begin{align}
\label{eq:inhomo_a}
(y_{n+1}^{2s-1}\p_{n+1}v)^{\frac{3-2s}{2s}}J(v)f(y'',x_n(y),x_{n+1}(y)) = (y_n y_{n+1})^{3-2s}r^2 C(y) ,
\end{align}
for some $C(y)\in C^{0,\epsilon}_\ast(\mathcal{B}_1^+)$. It can further be written in the form $y_{n+1}^{1-2s}r^{7-2s}\tilde{C}(y)$ with $\tilde{C}(y):= \left( \frac{y_n^{3-2s} y_{n+1}^2}{r^{5-2s}} \right) C(y) \in  C^{0,\epsilon}_{\ast}(\mathcal{B}_1^+)$. Here we have used the requirement that $s\in(0,1)$.
\end{proof}

Based on the previous observations and the form of our function spaces, we prove analyticity of $F$ as a mapping from an $X_{\alpha,\epsilon}(\mathcal{B}_1^+)$ neighborhood, $\mathcal{U}_{r_0}(v_0)$, of the model solution $v_0$ into $Y_{\alpha,\epsilon}(\mathcal{B}_1^+)$.
To simplify the notation, we set 
\begin{align*}
\tilde{F}(u,v,w):=\det\begin{pmatrix}
\p_{ii}u & \p_{in}u & \p_{i,n+1}u\\
\frac{1}{2s}\p_i(y_n^{1-2s}\p_{n}v) & -\frac{1}{2s}\p_n(y_n^{1-2s}\p_nv) & -\frac{1}{2s}\p_{n+1}(y_n^{1-2s}\p_nv)\\
-\frac{1}{2s}\p_i(y_{n+1}^{2s-1}\p_{n+1}w)& \frac{1}{2s}\p_n(y_{n+1}^{2s-1}\p_{n+1}w)& \frac{1}{2s}\p_{n+1}(y_{n+1}^{2s-1}\p_{n+1}w)
\end{pmatrix},
\end{align*}
and
\begin{align}
\label{eq:Ja_bi}
\tilde{J}(u,v):=\det\begin{pmatrix}
-\frac{1}{2s}\p_n(y_n^{1-2s}\p_nu)& -\frac{1}{2s}\p_{n+1}(y_n^{1-2s}\p_{n}u)\\
\frac{1}{2s} \p_{n}(y_{n+1}^{2s-1}\p_{n+1}v) & \frac{1}{2s} \p_{n+1}(y_{n+1}^{2s-1}\p_{n+1}v)
\end{pmatrix} = (y_{n+1}^{2s-1}\p_{n+1}v)^{\frac{2s-1}{2s}}J(v).
\end{align}
Then the nonlinear functional $F$ can be written as 
\begin{align*}
F(v)&=(y_{n+1}^{2s-1}\p_{n+1}v)^{\frac{1-2s}{s}}\tilde{F}(v,v,v)\\
&\quad +\p_n(y_n^{1-2s}y_{n+1}^{1-2s}\p_nv)+(y_{n+1}^{2s-1}\p_{n+1}v)^{\frac{1-2s}{s}}\p_{n+1}(y_n^{2s-1}y_{n+1}^{2s-1}\p_{n+1}v)\\
& \quad - (y_{n+1}^{2s-1}\p_{n+1}v)^{\frac{2-2s}{s}}\tilde{J}(v,v) f\left(y'',-\frac{1}{2s}y_n^{1-2s} \p_n v(y), (y_{n+1} ^{2s-1}\p_{n+1}v)^{\frac{1}{2s}}\right).
\end{align*}

\begin{prop}
\label{prop:analyticity}
Let $F$ be the nonlinear function from Proposition \ref{prop:eqnonlin} with inhomogeneity $f$. Denote by $v_0$ be the model solution from \eqref{eq:v0} and let $\epsilon,\alpha, r_0>0$ be the constants from Proposition~\ref{prop:map_non}.
Then, if $f$ is smooth in its arguments, the mapping
\begin{align*}
\mathcal{U}_{r_0}(v_0) \ni v \mapsto F(v) \in Y_{\alpha,\epsilon}(\mathcal{B}_1^+),
\end{align*}
is smooth. If $f$ is real analytic in its arguments, the above mapping is real analytic. 
\end{prop}

\begin{proof}
We only prove the real analyticity result. The proof for the smooth case is similar. We show that there exist $r_0,r_1>0$ such that for every $v\in \mathcal{U}_{r_0}(v_0)$ and every $h\in \mathcal{U}_{r_1}(0)$, we have an absolutely converging expansion
\begin{align*}
F(v+h) = \sum\limits_{j=0}^{\infty} \frac{1}{j!} D^j_vF (h^j),
\end{align*}
where $D_v^j F$ denotes the $j$-th order differential of $F$ with respect to $v$.\\

\emph{Step 1:} We claim that for each $v\in \mathcal{U}_{r_0}(v_0)\subset X_{\alpha,\epsilon}(\mathcal{B}_1^+)$ with $r_0=\frac{1}{4}\|v_0\|_{X_{\alpha,\epsilon}}$, for all $h,u, u_1,u_2,u_3 \in X_{\alpha,\epsilon}(\mathcal{B}_1^+)$, for each $k\in \mathbb{N}$ and for $s\in(0,1)$ 
\begin{align*}
&\left\|(y_{n+1}^{2s-1}\p_{n+1}v)^{\frac{1-2s}{s}-k}(y_{n+1}^{2s-1}\p_{n+1}h)^{k}\p_{n+1}(y_n^{2s-1}y_{n+1}^{2s-1}\p_{n+1}u)\right\|_{Y_{\alpha,\epsilon}(\mathcal{B}_1^+)}\\
&\leq C_s\left(\frac{1}{3}\|v_0\|_{X_{\alpha,\epsilon}(\mathcal{B}_1^+)}\right)^{-k}\|h\|_{X_{\alpha,\epsilon}(\mathcal{B}_1^+)}^{k}\|u\|_{X_{\alpha,\epsilon}(\mathcal{B}_1^+)},\\
&\left\|(y_{n+1}^{2s-1}\p_{n+1}v)^{\frac{1-2s}{s}-k}(y_{n+1}^{2s-1}\p_{n+1}h)^{k}\tilde{F}(u_1,u_2,u_3)\right\|_{Y_{\alpha,\epsilon}(\mathcal{B}_1^+)}\\
&\leq C_s \left(\frac{1}{3}\|v_0\|_{X_{\alpha,\epsilon}(\mathcal{B}_1^+)}\right)^{-k}\|h\|_{X_{\alpha,\epsilon}(\mathcal{B}_1^+)}^{k}\prod_{i=1}^{3}\|u_i\|_{X_{\alpha,\epsilon}(\mathcal{B}_1^+)}\\
&\left\|(y_{n+1}^{2s-1}\p_{n+1}v)^{\frac{2-2s}{s}-k}(y_{n+1}^{2s-1}\p_{n+1}h)^{k}\tilde{J}(u_1,u_2)\right\|_{Y_{\alpha,\epsilon}(\mathcal{B}_1^+)}\\
&\leq C_s \left(\frac{1}{3}\|v_0\|_{X_{\alpha,\epsilon}(\mathcal{B}_1^+)}\right)^{-k}\|h\|^k_{X_{\alpha,\epsilon}(\mathcal{B}_1^+)}\prod_{i=1}^{2}\|u_i\|_{X_{\alpha,\epsilon}(\mathcal{B}_1^+)},\\
&\left\|(y_{n+1}^{2s-1}\p_{n+1}u)^{\frac{2-2s}{s}}\tilde{J}(u,u)(y_n^{1-2s}\p_nh)^k\right\|_{Y_{\alpha,\epsilon}(\mathcal{B}_1^+)}\leq C_s \|h\|_{X_{\alpha,\epsilon}(\mathcal{B}_1^+)}^k\| u\|_{X_{\alpha,\epsilon}}^2.
\end{align*}
Indeed, the estimates follow from the decomposition in Proposition~\ref{prop:decomp} and the proof of Proposition~\ref{prop:map_non}. We will only show the first inequality and the arguments for the remaining ones are similar. Indeed, by the local version of Proposition~\ref{prop:decomp}, for $v, u, h\in X_{\alpha, \epsilon}(\mathcal{B}_1^+)$ we may assume that they have the decompositions
\begin{align*}
y_{n+1}^{2s-1}\p_{n+1}v &= \left(2a_1(y'')+r^{2\alpha-\epsilon}C_{n+1}(y)\right)y_n^{2s}y_{n+1}^{2s},\\
y_{n+1}^{2s-1}\p_{n+1}h& = \left(2\tilde{a}_1(y'')+r^{2\alpha-\epsilon}\tilde{C}_{n+1}(y)\right)y_n^{2s}y_{n+1}^{2s},\\
\p_{n+1}(y_n^{2s-1}y_{n+1}^{2s-1}\p_{n+1}u)&=4sy_n^{4s-1}y_{n+1}^{2s-1}\hat a_1(y'')+y_n^{4s-2}y_{n+1}^{2s-1}r^{1+2\alpha-\epsilon}\hat C(y).
\end{align*}
Then,
\begin{align*}
I&:=(y_{n+1}^{2s-1}\p_{n+1}v)^{\frac{1-2s}{s}-k}(y_{n+1}^{2s-1}\p_{n+1}h)^{k}\p_{n+1}(y_n^{2s-1}y_{n+1}^{2s-1}\p_{n+1}u)\\
&=\left(2a_1(y'')+r^{2\alpha-\epsilon}C_{n+1}(y)\right)^{\frac{1-2s}{s}-k}\left(y_n^{2s}y_{n+1}^{2s}\right)^{\frac{1-2s}{s}-k}\\
&\left(2\tilde{a}_1(y'')+r^{2\alpha-\epsilon}\tilde{C}_{n+1}(y)\right)^{k}\left(y_n^{2s}y_{n+1}^{2s}\right)^{k}\left(y_n^{4s-1}y_{n+1}^{2s-1}\hat a_1(y'')+y_n^{4s-2}y_{n+1}^{2s-1}r^{1+2\alpha-\epsilon}\hat C(y)\right).
\end{align*}
Simplifying the above expression leads to
\begin{align*}
I=(2a_1(y''))^{\frac{1-2s}{s}}\left(\frac{\tilde{a}_1(y'')}{a_1(y'')}\right)^k\left(\frac{1+r^{2\alpha-\epsilon}\tilde{D}(y)}{1+r^{2\alpha-\epsilon}D(y)}\right)^k\left(y_{n+1}^{1-2s}y_n\hat a_1(y'')+y_{n+1}^{1-2s}r^{1+2\alpha-\epsilon}\hat C(y)\right),
\end{align*}
where $D(y)=\frac{C_{n+1}(y)}{2a_1(y'')}$ and $\tilde{D}(y)=\frac{\tilde{C}_{n+1}(y)}{2\tilde{a}_1(y'')}$. If $u,v,h\in \mathcal{U}_{r_0}(v_0)\subset X_{\alpha,\epsilon}(\mathcal{B}_1^+)$ with $r_0=\frac{1}{4}\|v_0\|_{X_{\alpha,\epsilon}(\mathcal{B}_1^+)}$, then similarly as in the proof of Proposition~\ref{prop:map_non} we infer that
\begin{align*}
\|a_1\|_{C^{0,\alpha}(\mathcal{B}_1^+\cap P)}, \|\tilde{a}_0\|_{C^{0,\alpha}(\mathcal{B}_1^+\cap P)}&\geq \frac{3}{4}\geq \frac{1}{2}\|v_0\|_{X_{\alpha,\epsilon}(\mathcal{B}_1^+)},\\
\|C_{n+1}\|_{\dot{C}^{0,\epsilon}_\ast(\mathcal{B}_1^+)}, \| \tilde{C}_{n+1}\|_{\dot{C}^{0,\epsilon}_\ast(\mathcal{B}_1^+)}&\leq \frac{1}{4}\|v_0\|_{X_{\alpha,\epsilon}(\mathcal{B}_1^+)}.
\end{align*}
 Using these relations as well as the characterization for space $Y_{\alpha,\epsilon}$ in Proposition~\ref{prop:decomp},results in the claimed estimate for $\|I\|_{Y_{\alpha,\epsilon}(\mathcal{B}_1^+)}$.\\

\emph{Step 2:} 
We first discuss the contributions originating from the expansion of 
\begin{align*}
F_1(v):=(y_{n+1}^{2s-1}\p_{n+1}v)^{\frac{1-2s}{s}}\p_{n+1}(y_n^{2s-1} y_{n+1}^{2s-1}\p_{n+1}v).
\end{align*}
We begin by noting that 
\begin{align*}
D_v F_1(h) &= \frac{1-2s}{s} (y_{n+1}^{2s-1}\p_{n+1}v)^{\frac{1-2s}{s}-1} (y_{n+1} ^{2s-1}\p_{n+1}h)\p_{n+1}(y_n^{2s-1} y_{n+1}^{2s-1}\p_{n+1}v)\\
&\quad +(y_{n+1}^{2s-1}\p_{n+1}v)^{\frac{1-2s}{s}}\p_{n+1}(y_n^{2s-1}y_{n+1}^{2s-1}\p_{n+1}h).
\end{align*}
In general for $k\geq 2$
\begin{align*}
D_v^kF_1(h^k)&=c_{s,k}(y_{n+1}^{2s-1}\p_{n+1}v)^{\frac{1-2s}{s}-k}(y_{n+1}^{2s-1}\p_{n+1}h)^k\p_{n+1}(y_n^{2s-1} y_{n+1}^{2s-1}\p_{n+1}v)\\
&\quad +c_{s,k-1}(y_{n+1}^{2s-1}\p_{n+1}v)^{\frac{1-2s}{s}-(k-1)}(y_{n+1}^{2s-1}\p_{n+1}h)^{k-1}\p_{n+1}(y_n^{2s-1}y_{n+1}^{2s-1}\p_{n+1}h),
\end{align*}
where $c_{s,k}=\prod_{j=1}^{k}\frac{1-2s-(j-1)s}{s^{j}}$.
Thus, by virtue of Step 1 we obtain that $D^k_vF_1$ satisfies
\begin{align*}
\|D^k_vF_1(h^k)\|_{Y_{\alpha,\epsilon}(\mathcal{B}_1^+)}\leq C_s c_{s,k}\left(\frac{1}{3}\|v_0\|_{X_{\alpha,\epsilon}(\mathcal{B}_1^+)}\right)^{-k}\|h\|^k_{X_{\alpha,\epsilon}(\mathcal{B}_1^+)}.
\end{align*}
To show the absolute convergence of the Taylor series, we note that
\begin{align*}
\sum_{k=0}^{\infty}\frac{1}{k!}\|D_vF^k(h^k)\|_{Y_{\alpha,\epsilon}(\mathcal{B}_1^+)}&\leq C_s\sum_{k=0}^{\infty}\frac{1}{k!}c_{s,k}\left(\frac{1}{3}\|v_0\|_{X_{\alpha,\epsilon}(\mathcal{B}_1^+)}\right)^{-k}\|h\|_{X_{\alpha,\epsilon}(\mathcal{B}_1^+)}^k\\
\end{align*}
As $\frac{|c_{s,k}|}{k!}\leq C<\infty$, by majorization by a geometric series, the series hence converges absolutely if $\|h\|_{X_{\alpha,\epsilon}(\mathcal{B}_1^+)}< \frac{1}{3}\|v_0\|_{X_{\alpha,\epsilon}(\mathcal{B}_1^+)}$.\\

\emph{Step 3:} Next we discuss the contribution originating from the expansion of 
\begin{align*}
F_2(v):=(y_{n+1}^{2s-1}\p_{n+1}v)^{\frac{1-2s}{s}}\tilde{F}(v,v,v).
\end{align*}
We observe that
\begin{align*}
D_v\tilde{F}(h)&=\tilde{F}(h,v,v)+\tilde{F}(v,h,v)+\tilde{F}(v,v,h),\\
D^2_v\tilde{F}(h,h)&= 2\tilde{F}(h,h,v)+2\tilde{F}(v,h,h)+2\tilde{F}(h,v,h),\\
D^3_v\tilde{F}(h,h,h)&=6\tilde{F}(h,h,h).
\end{align*}
Thus, estimating this similarly as in Step 2 and using Step 1, we obtain the analyticity of $F_2$.\\

\emph{Step 4:} We finally discuss the contribution coming from the inhomogeneity:
\begin{equation}
\label{eq:inhomo_trafo}
\begin{split}
&J(v)(y_{n+1}^{2s-1}\p_{n+1}v)^{\frac{3-2s}{2s}}f\left(y'', -\frac{1}{2s}y_n^{1-2s}\p_n v, (y_{n+1}^{2s-1}\p_{n+1}v)^{\frac{1}{2s}}\right)\\
&=:\tilde{J}(v,v)(y_{n+1}^{2s-1}\p_{n+1}v)^{\frac{2-2s}{s}}\tilde{\mathcal{F}}(v),
\end{split}
\end{equation}
with $\tilde{\mathcal{F}}(v):=f\left(y'', -\frac{1}{2s}y_n^{1-2s}\p_n v, y_{n+1}^{2s-1}\p_{n+1}v\right)$. Hence, $\tilde{\mathcal{F}}(v)$ is a composition of the analytic function $f$, and the maps 
\begin{align*}
\R^{n+1}_+ \ni z&\mapsto g(z):=(z'',z_n,z_{n+1}^{1/2s}), \\
 \mathcal{U}_{r_0}(v_0) \times \R^{n+1}_+ \ni (v,y)& \mapsto h(v,y):= (y'',-\frac{1}{2s}\p_n v, y_{n+1}^{2s-1}\p_{n+1}v). 
\end{align*}
We note that away from $\{z_{n+1}=0\}$ the function $(f\circ g)(z)$, as a function of $z\in \R^{n+1}_+$, is analytic. However, instead of expanding on this observation to provide the proof of the analyticity of (\ref{eq:inhomo_trafo}), we argue similarly as in Steps 2 and 3. We use the product and chain rules and invoke the last two inequalities in Step 1: More precisely, 
\begin{align*}
&D_v\tilde{\mathcal{F}}(h)=(\p_nf)|_{\left(y'', -\frac{1}{2s}y_n^{1-2s}\p_n v,(y_{n+1}^{2s-1}\p_{n+1}v)^{\frac{1}{2s}}\right)}\left(-\frac{1}{2s}y_n^{1-2s}\p_nh\right)\\
&+(\p_{n+1}f)|_{\left(y'', -\frac{1}{2s}y_n^{1-2s}\p_n v,(y_{n+1}^{2s-1}\p_{n+1}v)^{\frac{1}{2s}}\right)}\frac{1}{2s}(y_{n+1}^{2s-1}\p_{n+1}v)^{\frac{1}{2s}-1}\left(y_{n+1}^{2s-1}\p_{n+1}h\right).
\end{align*}
Using the chain rule, it is possible to explicitly compute the higher derivatives $D_v^k\tilde{\mathcal{F}}(h^k)$. Thus, to estimate $D_v^k\left(\tilde{J}(v,v)(y_{n+1}^{2s-1}\p_{n+1}v)^{\frac{2-2s}{s}}\tilde{\mathcal{F}}(v)\right)$, if the differentiation falls on $\tilde{\mathcal{F}}(v)$, we invoke the last inequality in Step 1. Using the analyticity of $\tilde{\mathcal{F}}$ in its arguments, we obtain the convergence of the power series $\sum_{j=0}^{\infty}\frac{1}{j!}\tilde{J}(v,v)(y_{n+1}^{2s-1}\p_{n+1}v)^{\frac{2-2s}{s}}D_v^j\tilde{\mathcal{F}}(h^j)$ in $Y_{\alpha,\epsilon}$. 
If the differentiation falls on $\tilde{J}(v,v)$ or on $(y_{n+1}^{2s-1}\p_{n+1}v)^{\frac{2-2s}{s}}$, the argument is analogous as in Steps 2 and 3: We use the last two inequalities from Step 1 and the linear dependence of $\tilde{J}(v_1,v_2)$ on $v_1,v_2$.
\end{proof}

\subsection{Mapping properties of the linearized equation}
\label{sec:lin}
Let $L_v=D_vF$ denote the first order differential of the nonlinear functional with respect to $v$. In this section we show that for $v\in \mathcal{U}_{r_0}(v_0)$ with $r_0=\frac{1}{4}\|v_0\|_{X_{\alpha,\epsilon}(\mathcal{B}_1^+)}$, the linear operator $L_v:X_{\alpha,\epsilon}(\mathcal{B}_1^+)\rightarrow Y_{\alpha,\epsilon}(\mathcal{B}_1^+)$ is a perturbation of the constant coefficient fractional Grushin operator $\D_{G,s}$ in \eqref{eq:frac_1tional_Gru}.

\begin{prop}
\label{prop:perturb}
Let $v_0$ be the model solution from \eqref{eq:v0}. Let $f$ be the inhomogeneity in (\ref{eq:fracLa_2}) and suppose that it is $C^{1,\epsilon}(B_1^+)$ regular (in $x$ coordinates). Assume that $\alpha, \epsilon,r_0$ are as in Proposition~\ref{prop:map_non}. 
Then for any $v, h\in \mathcal{U}_{r_0}(v_0)$ with $r_0=\frac{1}{4}\|v_0\|_{X_{\alpha,\epsilon}}$ we have
\begin{align*}
\left\|(L_{v} -\D_{G,s})h\right\|_{Y_{\alpha,\epsilon}(\mathcal{B}_{1}^+)}&\leq C_{n,s}\left(\|v-v_0\|_{X_{\alpha,\epsilon}(\mathcal{B}_1^+)}\|h\|_{X_{\alpha,\epsilon}(\mathcal{B}_1^+)} \right.\\
& \left.\quad + \|v\|_{X_{\alpha,\epsilon}(\mathcal{B}_1^+)}^2\|h\|_{X_{\alpha,\epsilon}(\mathcal{B}_1^+)}\|f\|_{C^{1,\epsilon}(B_{1}^+)}\right).
\end{align*}
\end{prop}

\begin{proof}
The proof follows from the chain rule and Proposition~\ref{prop:analyticity}. To further simplify the notation, we define
\begin{align*}
W_1(v)&:=(y_{n+1}^{2s-1}\p_{n+1}v)^{\frac{1-2s}{s}},\quad W_2(v):=(y_{n+1}^{2s-1}\p_{n+1}v)^{\frac{2-2s}{s}},\\
G(v)&:=\tilde{F}(v,v,v)+\p_{n+1}(y_n^{2s-1}y_{n+1}^{2s-1}\p_{n+1}v),\\
\tilde{J}(v)&:= \tilde{J}(v,v),\\
\tilde{\mathcal{F}}(v)&:=f\left(y'',-\frac{1}{2s}y_n^{1-2s}\p_nv, (y_{n+1}^{2s-1}\p_{n+1}v )^{\frac{1}{2s}}\right),
\end{align*}
where $\tilde{J}(v,v)$ and $\tilde{F}(v,v,v)$ are as in (\ref{eq:Ja_bi}) (note that $\tilde{J}(v)$ differs from $J(v)$ in \eqref{eq:nonlinear} by a factor $(y_{n+1}^{2s-1}\p_{n+1}v)^{\frac{2s-1}{2s}}$).
Then $F$ can be written as 
\begin{align*}
F(v)= W_1(v)G(v)+\p_n(y_n^{1-2s}y_{n+1}^{1-2s}\p_{n}v) + W_2(v)J(v)\mathcal{F}(v).
\end{align*}
By the product rule, 
\begin{align*}
L_v h& = G(v)D_v W_1(h)+ W_1(v)D_vG(h) +\p_n(y_n^{1-2s}y_{n+1}^{1-2s}\p_nh)\\
& \quad  + \tilde{J}(v)\tilde{\mathcal{F}}(v) D_v W_2(h) + W_2(v) D_v \tilde{J}(h)\tilde{\mathcal{F}}(v)+W_2(v)\tilde{J}(v)D_v\tilde{\mathcal{F}}(h) .
\end{align*}
Noting that $\D_{G,s}h=G(v_0)D_{v_0}W_1(h)-W_1(v_0)D_{v_0}G(h)$, we obtain
\begin{align*}
&L_v h - \D_{G,s}h\\
&=G(v)D_vW_1(h)+W_2(v)D_vG(h) -G(v_0)D_{v_0}W_1(h)-W_1(v_0)D_{v_0}G(h)\\
& \quad +\tilde{J}(v)\tilde{\mathcal{F}}(v) D_v W_2(h) + W_2(v) D_v \tilde{J}(h)\tilde{\mathcal{F}}(v)+W_2(v)\tilde{J}(v)D_v\tilde{\mathcal{F}}(h) \\
&=\left(D_vW_1-D_{v_0}W_1\right)(h)G(v)+D_{v_0}W_1(h)\left(G(v)-G(v_0)\right)\\
&\quad + \left(D_vG-D_{v_0}G\right)(h)W_1(v)+D_{v_0}G(h)\left(W_1(v)-W_1(v_0)\right)\\
& \quad +\tilde{J}(v)\tilde{\mathcal{F}}(v) D_v W_2(h) + W_2(v) D_v \tilde{J}(h)\tilde{\mathcal{F}}(v)+W_2(v)\tilde{J}(v)D_v\tilde{\mathcal{F}}(h) \\
&=:  I+II+III.
\end{align*}
Using the estimates in Step 1 of Proposition~\ref{prop:analyticity}, we obtain
\begin{align*}
&\|I\|_{Y_{\alpha,\epsilon}(\mathcal{B}_1^+)}+\|II\|_{Y_{\alpha,\epsilon}(\mathcal{B}_1^+)}\\
\leq &C_s\|v-v_0\|_{X_{\alpha,\epsilon}(\mathcal{B}_1^+)}\left(\|v\|_{X_{\alpha,\epsilon}(\mathcal{B}_1^+)}+\|v_0\|_{X_{\alpha,\epsilon}(\mathcal{B}_1^+)}+1\right)\|h\|_{X_{\alpha,\epsilon}(\mathcal{B}_1^+)}.
\end{align*}
Hence it remains to bound III. To this end, using Step 4 of Proposition~\ref{prop:analyticity} we note that
\begin{align*}
\|III\|_{Y_{\alpha,\epsilon}(\mathcal{B}_1^+)} \leq C_s \|v\|_{X_{\alpha,\epsilon}(\mathcal{B}_1^+)}^2 \|h\|_{X_{\alpha,\epsilon}(\mathcal{B}_1^+)} \| f \|_{C^{1,\epsilon}(B_{1}^+)}.
\end{align*}
Here we used a bound similar as in the proof of Proposition \ref{prop:analyticity}, Step 1.
\end{proof}

\section{Application of the Implicit Function Theorem}
\label{sec:IFT}

In this section we invoke the implicit function theorem to deduce the smoothness and analyticity of the regular free boundary (for smooth and analytic inhomogeneities respectively). To this end, we introduce an auxiliary one-parameter family of diffeomorphisms, which infinitesimally acts as a translation on $P$. Composing our function with this one-parameter family of diffeomorphisms creates a parameter-dependent problem, to which we apply the implicit function theorem (c.f. \cite{An90}, \cite{KL12}, \cite{KRSIII}). This then yields the desired tangential regularity of our solution and hence proves Theorem \ref{thm:ana}.\\
As we rely on the results of the previous sections, we always assume that the conditions (A1)-(A4) are satisfied in the sequel.\\

We begin by defining our family of diffeomorphisms. As this is identical to the set-up in \cite{KRSIII}, we do not present the details of the proof.

\begin{lem}[One-parameter family of diffeomorphisms, \cite{KRSIII}]
\label{lem:diffeo}
Let $y\in \mathcal{B}_1^+$, $a\in \mathcal{B}''_1$.  Consider $\phi_a:[0,1]\rightarrow \R^{n-1}$, which is defined as the solution to the ODE
\begin{align*}
\phi_a'(t) &= a\left((3/4)^2-|\phi(t)|^2\right)^3_+ \eta(y_n, y_{n+1}),\\
\phi_a(0) & = y''.
\end{align*}
Here $\eta:\R^2 \rightarrow \R$ is a smooth, radial cut-off function, which is one for $|(y_n,y_{n+1})|\leq 1/4$ and which vanishes for $|(y_n,y_{n+1})|\geq 1/2$.
Let $\Phi_a(y):=(\phi_{a,y}(1),y_n,y_{n+1})$. Then the mapping $\Phi_a: \mathcal{B}_1^+ \rightarrow \R^{n+1}$ is well-defined and satisfies the following properties:
\begin{itemize}
\item[(i)] For any fixed $y\in \mathcal{B}_{1/2}^+$, the map $\mathcal{B}_{1}''\ni a\mapsto \Phi_a(y) \in \R^{n+1}$ is analytic. 
\item[(ii)] For each $a\in \mathcal{B}''_1$, the map $\mathcal{B}_1^+\ni y\mapsto \Phi_a(y) \in \R^{n+1}$ is $C^{3}$, and moreover $\|\Phi_a(y)-y\|_{C^3(\mathcal{B}_1^+)}\leq Ca$.
\item[(iii)] $\Phi_a (\{y_n=0\})\subset \{y_n=0\}$, $\Phi_a (\{y_{n+1}=0\})\subset \{y_{n+1}=0\}$.
\item[(iv)] $\p_{n+1}\Phi_a|_{\{y_{n+1}=0\}} =e_{n+1}$.
\item[(v)] $\p_{n}\Phi_a^j(y) = 0 = \p_{n+1}\Phi_a^j(y)$ for all $y$ with $|(y_n,y_{n+1})|\leq 1/4$ and $j\in\{1,\dots,n-1\}$.
\item[(vi)] $\Phi_a(y)=y$ for $y\in \{y\in \mathcal{B}_1^+: |y''|\geq \frac{3}{4}\text{ or } |(y_n,y_{n+1})|\geq \frac{1}{2}\}$.
\end{itemize}
\end{lem}

We now apply the implicit function theorem to deduce the tangential smoothness or analyticity of our Legendre function $ v$. Recall that given a solution $w$ to the fractional thin obstacle problem (\ref{eq:fracLa_2}) satisfying the assumptions (A1)-(A4), the Hodograph-Legendre transformation was invertible in $B_{\delta_0}^+$ with some small radius $\delta_0=\delta_0(s)$. Hence, it is possible to consider the Legendre function $v$ (c.f. \eqref{eq:Leg0}) in the corresponding image domain. The asymptotics and regularity properties of $v$ were studied in Section~\ref{sec:asymp1}. In particular, by Corollary~\ref{cor:reg_v}, $v\in X_{\alpha,\epsilon}(\mathcal{B}_{\delta_0}^+)$, where $\alpha$ is the Hölder exponent of the regular free boundary $\Gamma_w$ and $\epsilon$ is any number in $(0,\alpha)$ (here $\delta_0$ might be different from above the constant from above, but it also only depends on $s$).\\

To simplify the notation, in the sequel we will assume that 
$$\delta_0=1,\quad \text{ i.e. }v\in X_{\alpha,\epsilon}(\mathcal{B}_1^+).$$
Furthermore, we suppose that $v$ is close to the model solution $v_0$ in $X_{\alpha,\epsilon}(\mathcal{B}_1^+)$, i.e. 
\begin{equation}
\label{eq:nor2}
v\in \mathcal{U}_{\frac{r_0}{2}}(v_0)\subset X_{\alpha,\epsilon}(\mathcal{B}_1^+), \quad \text{where }r_0=\frac{1}{4}\|v_0\|_{X_{\alpha,\epsilon}}.
\end{equation} 
We remark that by Proposition~\ref{prop:asymp_v1} and Proposition~\ref{prop:asymp_v2}, these assumptions are always satisfied by choosing $\epsilon_0$, $\mu_0$ and $[\nabla g]_{\dot{C}^{0,\alpha}}$ in (A2)-(A4) to be sufficiently small, and by a scaling with a factor which only depends on $s$. \\

After these normalizations, given a Legendre function $v$ as above, as in \cite{KRSIII} we now consider a one-parameter family of Legendre functions:
\begin{align}\label{eq:v_a}
\tilde{v}_a(y):= v(\Phi_a(y)).
\end{align}
We note that by Lemma~\ref{lem:diffeo} (vi), the perturbation $\Phi_a$ is only active in $\mathcal{B}_{1/2}^+\Subset \mathcal{B}_1^+$. More precisely, $\tilde{v}_a(y)=v(y)$ in $\mathcal{B}_1^+\setminus\{y: |y''|\leq \frac{3}{4}, \ |(y_n,y_{n+1})|\leq \frac{1}{2}\}$. 
We claim that $\tilde{v}_a$ satisfies the following further properties:

\begin{prop}
\label{prop:comp}
Let $\alpha,\epsilon$ satisfy the same assumptions as in Proposition~\ref{prop:map_non} and let $v, \tilde{v}_a$ be as above. Then, 
\begin{itemize}
\item[(i)] there exists a constant $\eta_0=\eta_0(n,s)\in (0,1/4)$ such that for all $a \in B''_{\eta_0}$ we have, $\tilde{v}_a\in \mathcal{U}_{r_0}(v_0)\subset X_{\alpha,\epsilon}(\mathcal{B}_1^+)$ for $r_0=\frac{1}{4}\|v_0\|_{X_{\alpha,\epsilon}(\mathcal{B}_1^+)}$. In particular, $\tilde{v}_a$  satisfies the mixed Dirichlet-Neumann conditions
\begin{align*}
\tilde{v}_a = 0 \mbox{ on } \{y_n = 0\}, \  \lim\limits_{y_{n+1}\rightarrow 0_+} y_{n}^{-2s}\p_{n+1}\tilde{v}_a= 0 \mbox{ on } \{y_{n+1}=0\}.
\end{align*}
\item[(ii)] $\tilde{v}_a$ is a solution to the fully nonlinear equation
\begin{align*}
0=F_a(u, y) := F(u\circ \Phi_a^{-1}(z),z)|_{z=\Phi_a(y)}.
\end{align*}
\end{itemize}
\end{prop}

\begin{proof}
To show (i) we compute the derivatives of $\tilde{v}_a$ in terms of the ones of $v$: 
\begin{align*}
\p_i \tilde{v}_a(y) &= \p_j v|_{\Phi_a(y)} \p_i \Phi^{j}_a(y),\\
\p_{ik} \tilde{v}_a(y) & = \p_{j \ell} v|_{\Phi_a(y)} \p_i \Phi^j_a(y) \p_k \Phi^\ell_a (y) + \p_j v(y) \p_{ik}\Phi^j_a(y).
\end{align*}
We first note that in the domain in which $|(y_n,y_{n+1})|\leq 1/4$, the tangential and normal derivatives are not mixed (c.f. Lemma \ref{lem:diffeo} (v)). Thus, as all tangential derivatives are treated homogeneously, around $P$ the function $\tilde{v}_a$ satisfies the decomposition from Proposition~\ref{prop:decomp}, if $|a|<1/4$. Away from $P$ we invoke the radial assumption on $\eta$:
In order to conclude that $\tilde{v}_a \in X_{\alpha,\epsilon}(\mathcal{B}_1^+)$, it remains to discuss the region in which $1/4\leq |(y_n,y_{n+1})|\leq 1/2$. As $r(y) \sim 1$ in this region, it suffices to study the behavior of $\p_i v$ at the boundaries $\{y_n=0\}$ and $\{y_{n+1}=0\}$. Due to the radial dependence of $\eta$ (and by considering the ODE from Lemma \ref{lem:diffeo}), we however have 
\begin{align*}
|\p_n \Phi^j_a(y) | \leq C y_{n} \mbox{ and }  |\p_{n+1} \Phi^j_a(y) | \leq C y_{n+1} \mbox{ for } j\in\{1,\dots,n-1\}.
\end{align*}
Hence, the asymptotics at the boundary (and the Dirichlet-Neumann boundary conditions) follow in this region as well. \\
Additionally, choosing $|a|\leq \eta_0$ for some sufficiently small $\eta_0(n,s)>0$, we further infer that $\tilde{v}_a\in \mathcal{U}_{r_0}(v_0)$ with $r_0=\frac{1}{4}\|v_0\|_{X_{\alpha,\epsilon}(\mathcal{B}_1^+)}$ (this follows from \eqref{eq:nonlinear} and the estimate $\|\tilde{v}_a-v\|_{X_{\alpha,\epsilon}(\mathcal{B}_1^+)}\lesssim |a|$).\\
In order to compute the equation satisfied by $v_a$, we set $\Psi_a := \Phi_a^{-1}$ and observe that $\tilde{v}_a(\Psi_a(y)) = v(y)$. This then yields:
\begin{align*}
\p_i v(y) &= \p_j \tilde{v}_a|_{\Psi_a(y)} \p_i \Psi^j_a(y),\\
\p_{ik} v(y) & = \p_{j \ell} \tilde{v}_a|_{\Psi_a(y)} \p_i \Psi^j_a(y) \p_k \Psi^\ell_a (y) + \p_j \tilde{v}_a|_{\Psi_a(y)} \p_{ik}\Psi^j_a(y),
\end{align*}
from which we infer the equation for $\tilde{v}_a$.
\end{proof}

\begin{rmk}
The linearization $D_v F_a(\cdot, y)$ of the nonlinear function $F_a(\cdot, y)$ still satisfies local a priori estimates in the spaces $X_{\alpha,\epsilon}, Y_{\alpha,\epsilon}$. This is a consequence of the existence of a one-to-one correspondence between the solutions of $F_a$ and $F$ by means of the diffeomorphism $\Phi_a$ and by the discussion in the preceding Proposition \ref{prop:comp}. 
\end{rmk}

We finally prepare the application of the implicit function theorem by extending our problem to a problem on the whole quarter space $Q_+$ and by working with the function $\tilde{w}_a:= \tilde{v}_a-v$ rather than with $\tilde{v}_a$. We point out that
as $\tilde{w}_a$ is compactly supported in $\mathcal{B}_{1/2}^+$, we can avoid dealing with artificially created boundaries $\p \mathcal{B}_1^+\cap \{y_n>0,y_{n+1}>0\}$. 

\begin{prop}
\label{prop:final_nonlin}
Let $v, \tilde{v}_a\in X_{\alpha,\epsilon}(\mathcal{B}_1^+)$ be as \eqref{eq:v_a}. Let $\tilde{w}_a:= \tilde{v}_a-v$ in $\mathcal{B}_1^+$ and set $\tilde{w}_a=0$ in $Q_+\setminus \mathcal{B}_1^+$. Let $\epsilon, \alpha\in (0,1)$ with $\epsilon\leq \alpha$. 
Then, 
\begin{itemize}
\item[(i)] $\supp  (\tilde{w}_a)\subset \mathcal{B}_{3/4}^+$. Moreover, $\tilde{w}_a$ satisfies the equation
\begin{align*}
G_a(\tilde{w}_a,y):= \bar{\eta}(y) \tilde{F}_a(\tilde{w}_a,y) + (1-\bar{\eta}(y)) \D_{G,s}\tilde{w}_a=0 \text{ in } Q_+,
\end{align*}
where $\bar{\eta}$ is a smooth cut-off function that is one on $\mathcal{B}_{3  /4}^+$ and zero outside $\mathcal{B}_1^+$, $\tilde{F}_a(\tilde{w}_a,y)= F_a(v+\tilde{w}_a,y)$, and $\D_{G,s}$ is the fractional Baouendi-Grushin Laplacian in \eqref{eq:frac_1tional_Gru}.
\item[(ii)] For $\eta_0=\eta_0(n,s)>0$ as in Proposition \ref{prop:comp} and for $a\in B''_{\eta_0}$, the map
\begin{align*}
 w \mapsto G_a(w)\in Y_{\alpha,\epsilon}
\end{align*}
is analytic in $\mathcal{U}_{r_0/2}(0)\subset X_{\alpha,\epsilon}$ with $r_0=\frac{1}{4}\|v_0\|_{X_{\alpha,\epsilon}}$.
\item[(iii)] For any $w\in \mathcal{U}_{r_0/2}(0)$, 
\begin{align*}
B''_{\eta_0} \ni a \mapsto G_a(w)\in Y_{\alpha,\epsilon}
\end{align*}
is analytic. 
\item[(iv)] There exists $\mu_0=\mu_0(s,n)>0$ such that if $\|f\|_{C^{1,\epsilon}(B_1^+)}<\mu_0$, then 
$$D_w G|_{(w,a)=(0,0)}: X_{\alpha,\epsilon} \rightarrow Y_{\alpha,\epsilon}$$ 
is an invertible map. 
\end{itemize}
\end{prop}

\begin{proof}
The proof of (i) follows immediately by noticing that $\tilde{w}_a=0$ outside $\mathcal{B}_{3/4}^+$ and by rewriting $\tilde{v}_a = \tilde{w}_a + v$. \\
To prove (ii), we first note that $w+v\in \mathcal{U}_{r_0}(v_0)$ for each $w\in \mathcal{U}_{r_0/2}(0)\subset X_{\alpha,\epsilon}$.  
Applying Proposition \ref{prop:analyticity} we obtain that $w\mapsto F(w+v,y)$ is analytic in $\mathcal{U}_{r_0/2}(0)\subset X_{\alpha,\epsilon}(\mathcal{B}_1^+)$. The analyticity of $w\mapsto F_a(w+v,y)$ for fixed $a$ (recall $F_a$ is defined Proposition~\ref{prop:comp} (ii)) follows from the properties (ii), (v) (in Lemma~\ref{lem:diffeo}) of the diffeomorphism $\Phi_a$ and the analyticity for $F$. Thus by the definition of $G_a$, the map $w\mapsto G_a(w)$ is analytic in $\mathcal{U}_{r_0/2}(0)$ as well.\\
 The statement (iii) follows from the analytic dependence of $\Phi_a$ and $\Psi_a$ on $a$ (c.f. Lemma~\ref{lem:diffeo} (i), (ii)). Finally, using Lemma~\ref{lem:diffeo} (ii), we can directly compute that $D_wG\big|_{(w,a)=(0,0)}=\bar \eta L_v + (1-\bar \eta)\D_{G,s}$, where $L_v=D_vF$. Since $\D_{G,s}:X_{\alpha,\epsilon}\rightarrow Y_{\alpha,\epsilon}$ is invertible, Proposition~\ref{prop:perturb} implies that the linearization $D_w G\big|_{(0,0)}:X_{\alpha,\epsilon}\rightarrow Y_{\alpha,\epsilon}$ is also invertible, if $\|f\|_{C^{1,\epsilon}(B_1^+)}$ is sufficiently small (e.g. by rewriting $D_w G|_{(0,0)}= \D_{G,s}(Id + \D_{G,s}^{-1} (\bar{\eta}\mathcal{P}))$ with $\mathcal{P}$ being the operator from Proposition \ref{prop:perturb} and using the norm bounds from Proposition \ref{prop:perturb}). 
\end{proof}

With this at hand, we can finally prove our main theorem:

\begin{thm}[Analyticity]
\label{thm:ana}
Let $v$ be a Legendre function associated to a solution of the fractional thin obstacle problem (\ref{eq:fracLa_2}) with smooth or analytic inhomogeneity $f$. 
Then there exists a constant $\eta_0>0$ such that the mapping
\begin{align*}
B_{\eta_0}'' \ni y''\mapsto -\frac{1}{2s}y_n^{1-2s}\p_{n}v(y)|_{y=(y'',0,0)},
\end{align*}
is smooth if $f$ is smooth and real analytic if $f$ is real analytic. In particular, the regular free boundary $\Gamma_{s+1}(w)$ is locally smooth if $f$ is smooth and locally real analytic if $f$ is real analytic.
\end{thm}

\begin{proof}
The proof follows by an application of the smooth/analytic implicit function theorem (c.f. \cite{Dei10}). We only show the real analytic case. The arguments for the smooth case are analogous.

By Proposition \ref{prop:final_nonlin} (ii) and (iii) the mapping $G: B''_{\eta_0}\times \mathcal{U}_{r_0}(0)\rightarrow Y_{\alpha,\epsilon}$ is analytic for $r_0=\frac{1}{4}\|v_0\|_{X_{\alpha,\epsilon}}$. By Proposition~\ref{prop:final_nonlin} (iv) the operator $D_w G|_{(0,0)}$ is invertible from $X_{\alpha,\epsilon}$ to $Y_{\alpha,\epsilon}$. Due to the implicit function theorem there exists a neighborhood $(-\tilde{\epsilon}_0,\tilde{\epsilon}_0)^{n-1}\times \mathcal{U}_{\tilde{r}}(0)$ of $(0,0)$, such that for each $a\in (-\tilde{\epsilon}_0,\tilde{\epsilon}_0)^{n-1}$ there exists a unique function $w_a\in \mathcal{U}_{\tilde{r}}(0)\subset X_{\alpha,\epsilon}$ satisfying
\begin{align}
\label{eq:non_G}
G_a(w_a)=0.
\end{align}
Moreover, this solution $w_a$ depends analytically on the parameter $a$.
As the function $\tilde{w}_a = \tilde{v}_a -v\in X_{\alpha,\epsilon}$ (defined in Proposition~\ref{prop:final_nonlin}) also satisfies the nonlinear equation $G_a(\tilde{w}_a)=0$, and as $\|\tilde{w}_a\|_{X_{\alpha,\epsilon}}\lesssim |a|<\tilde{r}$ for a small choice of $|a|$, the local uniqueness result of the implicit function theorem asserts that $w_a = \tilde{w}_a$. Hence, as a function in $X_{\alpha,\epsilon}$, $\tilde{w}_a$ depends analytically on $a$. Thus, by definition of the norm of $X_{\alpha,\epsilon}$ the function $y_n^{1-2s}\p_n \tilde{w}_a$ also depends analytically on $a$. As a consequence, this remains true for $y_n^{1-2s}\p_n \tilde{v}_a$. Recalling that $\Phi_a$ infinitesimally corresponds to a (tangential) translation at $P$, this implies that the function $y_n^{1-2s}\p_n v$ depends analytically on the tangential variables. This yields the desired result.
\end{proof}

\begin{rmk}
\label{rmk:inhom}
We briefly comment on generalizations of Theorem \ref{thm:ana} for inhomogeneities with less regularity.
It is clear that by carefully tracking our arguments, the set-up of analytic and smooth inhomogeneities can also be extended to that of Hölder inhomogeneities.
\end{rmk}

\section{Appendix A}
\label{sec:AppA}

The following three sections contain auxiliary regularity results for the fractional Laplacian with Dirichlet, Neumann and mixed Dirichlet-Neumann data. These are deduced by relying on the compactness method similar as in \cite{Wa03} and build on approximation results in terms of eigenfunctions of the respective operator. On the one hand this is reminiscent of Campanato type arguments of proving regularity \cite{Ca64}, on the other hand it also reminds us of the methods used in obtaining the up to the corner (or edge) asymptotics of the solutions for elliptic equations in conical domains by means of eigenfunction approximations (c.f. \cite{Gr11}, \cite{Ko97}).

The Section is structured as follows: First we formulate and prove the up to the boundary regularity results for solutions of $L_sw=f$ with Dirichlet and Neumann boundary data in Section \ref{sec:reg1}. For these we use the compactness argument in e.g. \cite{Wa03} to approximate the solution $w$ by (a linear combination of) homogeneous global solutions to $L_su=0$ (c.f. Section \ref{sec:approx_poly}). A key step involves characterizing the homogeneous global solutions with Dirichlet, Neumann or mixed Dirichlet-Neumann boundary data. For this we use the eigenfunction approximations which we then describe in Section \ref{sec:Eigenf}.  \\
Although we deal with linear problems, the results in this section might be of independent interest. For instance, we obtain a full classification of all eigenvalues and eigenfunctions of the mixed Dirichlet-Neumann problem in the slit domain. This situation can be regarded as the linearization of the thin obstacle problem.\\
For us the results in this section play an important role in deducing the asymptotic expansions and mapping properties for the fractional Baouendi-Grushin operator (e.g. in Propositions \ref{prop:asymp2} and \ref{prop:map}). In particular in Section \ref{sec:open} we transfer the results from the Laplacian to the Baouendi-Grushin Laplacian by means of the square root mapping.

\subsection{Up to the boundary a priori estimates}
\label{sec:reg1}
The following sections are dedicated to regularity properties of the fractional Laplacian with various (linear) boundary conditions (c.f. in particular Sections \ref{sec:reg1} and Section \ref{sec:approx_poly}). Here a central ingredient is given by the approximation results in terms of ``eigenpolynomials'' (i.e. homogeneous global solutions with symmetry given by the Dirichlet and Neumann boundary condition) (c.f. Section \ref{sec:approx_poly}). To characterize the eigenpolynomials we use polar coordinates and study the form of the eigenfunctions associated with the respective operators on the sphere. These are computed in Section \ref{sec:Eigenf} and might be of independent interest. \\

In the following two sections we first discuss the desired regularity results and then provide the necessary auxiliary results (c.f. Propositions \ref{prop:Dirichlet}, \ref{prop:Neumann}).\\

In this section, we consider the equation
\begin{align*}
L_s u= f \text{ in } B_1^+,\quad u=0  \text{ on } B'_1,
\end{align*}
and the equation
\begin{align*}
L_s u = x_{n+1}^{1-2s} f \text{ in } B_1^+,\quad 
x_{n+1}^{1-2s}\p_{n+1} u=0 \text{ on } B'_1,
\end{align*}
where 
$$L_s:=\nabla \cdot x_{n+1}^{1-2s}\nabla,\quad s\in (0,1).$$
In both cases, we assume that the inhomogeneity $f\in C^{0,\epsilon}(B_1^+)$ for some $\epsilon\in (0,1)$. For solutions of the above equations, we show Schauder a priori estimates which hold up to $B'_1$. In the sequel, for simplicity, we let $\bar \omega$ denote the weight $x_{n+1}^{1-2s}$.

\begin{prop}[Dirichlet data]
\label{prop:Dirichlet}
Let $u\in L^\infty(B_1^+)\cap H^{1}_{\bar{\omega}}(B_1^+)$ be a solution of
\begin{align*}
L_s u = f \mbox{ in } B_1^+,\quad 
u = 0 \mbox{ on } B'_{1} ,
\end{align*}
where $f\in C^{0,\epsilon}(B_1^+)$ for some $\epsilon\in (0,1)$.
Then there exists a constant $C=C(n,s,\epsilon)>0$ such that
\begin{align*}
&\|x_{n+1}^{-2s}u\|_{C^{0,\epsilon}(B_{1/2}^+)}+ \sum_{i=1}^{n}\|x_{n+1}^{-2s}\p_iu\|_{C^{0,\epsilon}(B_{1/2}^+)}+ \|x_{n+1}^{1-2s}\p_{n+1}u\|_{C^{0,\epsilon}(B_{1/2}^+)}\\
&+ \sum_{i,j=1}^{n+1}\|\p_ix_{n+1}^{1-2s}\p_ju\|_{C^{0,\epsilon}(B_{1/2}^+)} \leq C\left(\|f\|_{C^{0,\epsilon}(B_1^+)}+\| u\|_{L^2_{\bar \omega}(B_1^+)}\right).
\end{align*}
\end{prop}

\begin{proof}
The result follows from the approximation result at $B'_{1/2}$ which will be shown in Proposition~\ref{prop:approx_main} and  a scaling argument.\\
More precisely, let $x_0=(x_0',0)\in B'_{1/2}$. Without loss of generality we may assume that $f(x_0)=0$, as we can always subtract the function $\frac{1}{1+2s}f(x_0)x_{n+1}^{1+2s}$.
Then, by the polynomial approximation result at $x_0$ (c.f. Proposition \ref{prop:approx_main} in Section \ref{sec:approx_poly}), there exists a polynomial $P_{x_0}(x)$ with the properties that
$$x_{n+1}^{2s}P_{x_0}(x):=x_{n+1}^{2s}(a(x_0)+\sum\limits_{j=1}^{n}b_j(x_0) x_{j}), $$ 
$L_s (x_{n+1}^{2s}P_{x_0})=0$ and $\|P_{x_0}\|_{L^\infty(B_1^+)}\leq C(\|f\|_{L^\infty(B_1^+)}+\|u\|_{L^2_{\bar\omega}(B_1^+)})$, such that
\begin{align}
\label{eq:approx}
\| u - x_{n+1}^{2s}P_{x_0} \|_{\tilde{L}^2_{\bar{\omega}}(B_r^+)} \leq C r^{1+2s+\epsilon}\left([f]_{C^{0,\epsilon}(B_{1}^+)} + \| u\|_{\tilde{L}^2_{\bar{\omega}}(B_{1}^+)}\right) \text{ for all }r\in (0,1/2).
\end{align}
Here $\| u\|_{\tilde{L}^2_{\bar{\omega}}(\Omega)}:= \frac{1}{\bar{\omega}(\Omega)^{1/2}}\|x_{n+1}^{\frac{1-2s}{2}} u\|_{L^2(\Omega)}$ with $\bar{\omega}(\Omega):=\int\limits_{\Omega}x_{n+1}^{1-2s}dx$, and $C=C(n,s,\epsilon)$.
Given $\lambda \in (0,1/2)$ , we consider
\begin{align*}
\tilde{u}_\lambda(z):= \frac{(u-x_{n+1}^{2s}P_{x_0})(x_0+\lambda z)}{\lambda^{1+2s+\epsilon}}.
\end{align*}
By \eqref{eq:approx} we infer that $\|\tilde{u}_\lambda\|_{L^2_{\bar{\omega}}(B_2^+)}\leq C([f]_{C^{0,\epsilon}(B_{1}^+)} + \| u\|_{\tilde{L}^2_{\bar{\omega}}(B_{1}^+)})$. Moreover,
\begin{align*}
L_s \tilde{u}_{\lambda}(z)&= f_{\lambda}(z),\quad f_{\lambda}(z):=\lambda^{-\epsilon}f(x_0+\lambda z).
\end{align*}
Since $f\in C^{0,\epsilon}(B_1^+)$, we have that $f_\lambda\in C^{0,\epsilon}(B_2^+)$ with $\|f_\lambda\|_{C^{0,\epsilon}(B_2^+)}\leq C \|f\|_{C^{0,\epsilon}(B_1^+)}$. 
We notice that in $B_{3/4}(e_{n+1})$ the equation for $\tilde{u}_\lambda$ is uniformly elliptic with a $C^{0,\epsilon}$ inhomogeneity. Hence, by classical elliptic estimates, we have
\begin{align*}
\|\tilde{u}_\lambda\|_{C^{2,\epsilon}(B_{1/2}(e_{n+1}))}&\leq C \left([f_\lambda]_{\dot{C}^{0,\epsilon}(B_2^+)}+ \|\tilde{u}_\lambda\|_{L^2(B_2^+)}\right).
\end{align*}
Scaling back, yields that in the non-tangential balls $B_{\lambda/2}(x_0+\lambda e_{n+1})$ it holds
\begin{equation}\label{eq:scale1}
\begin{split}
&\quad \lambda^{-(1+2s+\epsilon)}\|u-x_{n+1}^{2s}P_{x_0}\|_{L^\infty(B_{\lambda/2}(x_0+\lambda e_{n+1}))}\\
&+\lambda^{-(2s+\epsilon)}\|\p_i (u-x_{n+1}^{2s}P_{x_0})\|_{L^\infty(B_{\lambda/2}(x_0+\lambda e_{n+1}))}\\
&+\lambda^{-(-1+2s+\epsilon)}\|\p_{ij}(u-x_{n+1}^{2s}P_{x_0})\|_{L^\infty(B_{\lambda/2}(x_0+\lambda e_{n+1}))}\\
&+\lambda^{-(-1+2s)}[\p_{ij}(u-x_{n+1}^{2s}P_{x_0})]_{\dot{C}^{0,\epsilon}(B_{\lambda/2}(x_0+\lambda e_{n+1}))}\\
&\leq C([f]_{\dot{C}^{0,\epsilon}(B_{1}^+)} + \| u\|_{\tilde{L}^2_{\bar{\omega}}(B_{1}^+)}).
\end{split}
\end{equation}
Repeating the above procedure at every $x_0\in B'_{1/2}$, leads to \eqref{eq:scale1} for each $x_0\in B'_{1/2}$ and each $\lambda \in (0,1/2)$.\\

Based on this, a triangle inequality argument implies that $x_0\mapsto a(x_0)$ is in $C^{1,\epsilon}(B'_{1/2})$ and $x_0\mapsto b_i(x_0)$, $i\in\{1,\dots,n\}$, is in $C^{0,\epsilon}(B'_{1/2})$ with norm bounded by the right hand side of \eqref{eq:scale1}. More precisely, let $x_0$, $\hat x_0\in B'_{1/2}$. Let $\tilde{x}$ be the mid point of $x_0 $ and $\hat x_0$. We apply \eqref{eq:scale1} with $\lambda =2|x_0-\hat x_0|$ at $x_0$ and $\hat x_0$. By a triangle inequality (note that $B_{\lambda/4}(\tilde{x}+\lambda e_{n+1})\subset B_{\lambda/2}(x_0+\lambda e_{n+1})\cap B_{\lambda/2}(x_1+\lambda e_{n+1})$), 
\begin{align*}
\lambda^{-(1+2s+\epsilon)}\|x_{n+1}^{2s}P_{x_0}-x_{n+1}^{2s}P_{x_1}\|_{L^\infty(B_{\lambda/4}(\tilde{x}+\lambda e_{n+1}))}\\
+\lambda^{-(2s+\epsilon)}\|\p_i(x_{n+1}^{2s}P_{x_0}-x_{n+1}^{2s}P_{x_1})\|_{L^\infty(B_{\lambda/4}(\tilde{x}+\lambda e_{n+1}))}\leq C.
\end{align*}
Using that in $x_{n+1}\sim \lambda$ in $B_{\lambda/4}(\tilde{x}+\lambda e_{n+1})$, we have $|b_i(x_0)-b_i(\hat x_0)|\leq C\lambda^\epsilon$, $|a(x_0)-a(\hat x_0)|\leq C\lambda$, $|\nabla a(x_0)-\nabla a(\hat x_0)|\leq C\lambda^\epsilon$. Recalling the definition of $\lambda$, we obtain the desired estimate.\\
A further triangle inequality argument combined with a covering argument gives the up to $B'_{1/2}$ H\"older regularity of the weighted derivatives:
\begin{align*}
x_{n+1}^{-2s}u, \ x_{n+1}^{1-2s}\p_{n+1}u \in C^{1,\epsilon}(B_{1/2}^+),\\
x_{n+1}^{-2s}\p_iu, \ x_{n+1}^{1-2s}\p_{i,n+1}u, \ \p_{n+1}(x_{n+1}^{1-2s}\p_{n+1}u)\in C^{0,\epsilon}(B_{1/2}^+),
\end{align*}
and for each $x_0\in B'_{1/2}$
\begin{align*}
&x_{n+1}^{-2s}u \big|_{x=x_0}=a(x_0),\\
&x_{n+1}^{-2s}\p_i u \big |_{x=x_0}=\frac{1}{2s}x_{n+1}^{1-2s}\p_{i,n+1}u\big|_{x=x_0}=b_i(x_0), \quad i\in \{1,\dots, n\},\\
&\p_{n+1}(x_{n+1}^{1-2s}\p_{n+1}u)\big|_{x=x_0}=f(x_0).
\end{align*}
This yields the desired estimates. 
\end{proof}

Similarly as the Dirichlet case, we can treat the Neumann case.

\begin{prop}[Neumann data]
\label{prop:Neumann}
Let $u\in L^\infty(B_1^+)\cap H^{1}_{\bar\omega}(B_1^+)$ be a solution of
\begin{align*}
L_s u = x_{n+1}^{1-2s} f \mbox{ in } B_1^+,\quad \lim\limits_{x_{n+1}\rightarrow 0_+} x_{n+1}^{1-2s}\p_{n+1} u = 0 \mbox{ on } B'_1,
\end{align*}
where $f\in C^{0,\epsilon}(B_1^+)$ for some $\epsilon\in (0,1)$.
Then there exists $C=C(n,s,\epsilon)>0$ such that
\begin{align*}
&\|u\|_{C^{0,\epsilon}(B_{1/2}^+)}+\|\p_iu\|_{C^{0,\epsilon}(B_{1/2}^+)}+\|x_{n+1}^{-1}\p_{n+1}u\|_{C^{0,\epsilon}(B_{1/2}^+)}+\|\p_{ij}u\|_{C^{0,\epsilon}(B_{1/2}^+)}  \\
&\leq C\left(\|f\|_{C^{0,\epsilon}(B_1^+)} + \|u\|_{L^{\infty}(B_1^{+})}\right).
\end{align*}
\end{prop}

\begin{proof}
Again the proof follows by approximation at the set $\{x_{n+1}=0\}$ (where the Muckenhoupt weight $x_{n+1}^{1-2s}$ degenerates) and by rescaling. Let $x_0\in B'_{1/2}$. Without loss of generality we may assume that $f(x_0)=0$ (as we can always subtract the polynomial $\frac{1}{4(2-2s)}f(x_0)x_{n+1}^2$). By  Proposition~\ref{prop:approx_main} we have
\begin{align*}
\|( u-Q_{x_0})\|_{\tilde{L}^2_{\bar{\omega}}(B_{r}(y_0))} \leq Cr^{2+\epsilon }([f]_{C^{0,\epsilon}(B_1^+)} + \|u\|_{\tilde{L}^2_{\bar{\omega}}(B_1^+)}),
\end{align*}
for all $r\in(0,1/2)$. 
Here, 
$$Q_{x_0}(x)= c(x_0) + \sum\limits_{j=1}^{n}a_j(x_0) x_j + \sum\limits_{i,j=1}^{n}d_{ij}(x_0)x_i x_j,$$ 
which satisfies $L_sQ_{x_0}(x)=0$.
With the approximation result at hand, we argue similarly as in the previous proposition and rescale.
Let $\lambda \in (0,1/2)$ and consider
\begin{align*}
\tilde{v}_{\lambda}(z) : = \frac{(u-Q_{x_0})(x_0+\lambda z)}{\lambda^{2+\epsilon}}.
\end{align*} 
This makes the equation uniformly elliptic in $B_{3/4}(e_{n+1})$ and thus yields
\begin{align*}
\|\p_i \p_j \tilde{v}_{\lambda}\|_{C^{0,\epsilon}(B_{1/2}(e_{n+1}))} \leq C\left(\|f_{\lambda}\|_{C^{0,\epsilon}(B_{3/4}(e_{n+1}))} + \|\tilde{v}_{\lambda}\|_{\tilde{L}^2(B_{3/4}(e_{n+1}))}\right).
\end{align*}
Here $f_\lambda(z)= z_{n+1}^{1-2s}\lambda^{-\epsilon}f(x_0+\lambda z)$, with $\|f_\lambda\|_{C^{0,\epsilon}(B_{3/4}(e_{n+1}))}\leq C[f]_{\dot{C}^{0,\epsilon}(B_1^+)}$, and
$
\left\| \tilde{v}_{\lambda} \right\|_{\LL_{\bar{\omega}}(B_{3/4}(e_{n+1}))}\leq C([f]_{C^{0,\epsilon}(B_1^+)} + \|u\|_{\tilde{L}^2_{\bar{\omega}}(B_1^+)}).
$
Undoing the rescaling therefore yields the desired result in a non-tangential neighborhood of $x_0$. Applying this at each $x_0\in B'_{1/2}$ and using a triangle inequality argument as in Proposition~\ref{prop:Dirichlet}, we obtain the Hölder continuity of the coefficients of $Q_{x_0}$ in terms of $x_0$. This then implies the estimate up to the boundary. Since this part of argument is similar as in the proof of Proposition~\ref{prop:Dirichlet}, we do not repeat it here.
\end{proof}

\subsection{Approximation results for the fractional Laplacian}
\label{sec:approx_poly}

In this section, we prove approximation results for the fractional Laplacian in the upper half-space with Dirichlet, Neumann and mixed Dirichlet-Neumann data. The proofs of these results are based on the compactness method as for instance in \cite{Wa03}, \cite{Wang92}. Here a key step is the characterization of the homogeneous global solutions (with corresponding Dirichlet, Neumann boundary data). These rely on the computations of the eigenvalues and eigenfunctions in the following Section \ref{sec:Eigenf}. As shown in Section \ref{sec:open}, these results for the fractional Laplacian also suffice to prove a corresponding approximation result for the fractional Grushin Laplacian $\D_{G,s}$ (c.f. Proposition \ref{prop:gr_approx}).\\

The main result of this section is the following approximation statement:

\begin{prop}[Approximation]
\label{prop:approx_main}
Let $u\in H^1_{\bar\omega}(B_1^+)$ be a solution of
\begin{equation}
\label{eq:fra_DNM}
\begin{split}
L_s u & = g \mbox{ in } B_1^+,\\
B u &= 0 \mbox{ on } B_1',
\end{split}
\end{equation}
where $B \in \{B_D,B_N, B_{DN}\}$ is one of the following operators:
\begin{itemize}
\item (Dirichlet) $B_D u := u $,
\item (Neumann) $B_N u : = \p_{n+1}u $,
\item (mixed Dirichlet-Neumann) $B_{DN} u := u$ on $\{x_n\leq 0\}$ and $B_{DN} u = \p_{n+1}u$ on $\{x_{n}\geq 0\}$.
\end{itemize}
Assume that the inhomogeneity $g$ is of the following form:
\begin{itemize}
\item $g=f$ in the case of Dirichlet data,
\item $g= x_{n+1}^{1-2s} f$ for Neumann data,
\item and 
\begin{align*}
g(x)&=
x_{n+1}^{1-2s} (x_n^2+x_{n+1}^2)^{-1/2}  w_{0,s}(x_n,x_{n+1})f_0(x)+(x_n^2+x_{n+1}^2)^{-1/2} f_1(x)
\end{align*}
in the case of mixed Dirichlet-Neumann data.
\end{itemize}
Further suppose that 
\begin{itemize}
\item
in the Dirichlet and Neumann cases $f$ is $C^{0,\alpha}$ at $0$ in the sense that for all $x\in B_1^+$
\begin{align*}
|f(x)-f(0)|\leq C |x|^{\alpha},
\end{align*}
\item in the case of mixed Dirichlet-Neumann data $f_0$ is $C^{0,\alpha}$ at $0$, and for all $x\in B_1^+$, $f_1$ satisfies
\begin{align*}
|f_1(x)|\leq C |x|^{1+\alpha-s}.
\end{align*}
\end{itemize}
Then, there exist a constant $C=C(n,s,\alpha)>0$ and functions $h_{\beta_B}(x)$ such that for all $r\in(0,1/2)$
\begin{align*}
\|u-h_{\beta_B} \|_{\LLw(B_{r}^+)} \leq C r^{\beta_{B}+\alpha}\left( \left\| f\right\| +\|  u\|_{\LLw(B_1^+)} \right).
\end{align*}
Here $\|f\|$ denotes $\|f\|_{C^{0,\alpha}(0)}$ in the Dirichlet and Neumann cases, or $\|f_0\|_{C^{0,\alpha}(0)}+\sup\limits_{B_1^+}||x|^{-1-\alpha+s}f_1(x)|$ in the mixed Dirichlet-Neumann case, and 
\begin{align*}
\beta_{B}:= \left\{
\begin{array}{ll}
1+2s \mbox{ in the Dirichlet case},\\
2 \mbox{ in the Neumann case},\\
1+s \mbox{ in the mixed Dirichlet-Neumann case},
\end{array} \right.
\end{align*}
and
\begin{equation}
\label{eq:model_poly}
\begin{split}
h_{\beta_B}(x):= 
 \left\{
\begin{array}{ll}
x_{n+1}^{2s}\left( a + \sum\limits_{j=1}^{n} b_j x_{j} \right) + \frac{f(0)}{1+2s}x_{n+1}^{1+2s} \mbox{ in the Dirichlet case},\\
c + \sum\limits_{j=1}^{n} a_j x_j  + \sum\limits_{i,j=1}^{n}d_{ij} x_i x_j + \frac{f(0)}{2(2-2s)}x_{n+1}^{2}   \mbox{ in the Neumann case},\\
w_{0,s}(x_n,x_{n+1})\left(a_0 + a_1(s|x|-x_n) \right) + \frac{f_0(0)}{2(2+2s)}w_{0,s}^{1+1/s}(x_n,x_{n+1}) \\
\quad \quad \quad \mbox{in the mixed Dirichlet-Neumann case}.
\end{array} \right.
\end{split}
\end{equation}
All coefficients of $h_{\beta_B}$ are bounded in terms of $\| u\|_{\LLw(B_1^+)}$ and $\|f\|$.
\end{prop}

\begin{rmk}[Inhomogeneity for the Dirichlet-Neumann data]
The specific form of the inhomogeneity in the mixed Dirichlet-Neumann case stems from our definition of the spaces ${Y}_{\alpha,\epsilon}$ and the transformation behavior under the opening of the domain transformation described in Section \ref{sec:open}. Carrying out this transformation carefully leads to an inhomogeneity of the form
$$g(x)=x_{n+1}^{1-2s}r^{-1}\left(w_{0,s}(x_n,x_{n+1})f_0(x)+w_{0,s}(x_n,x_{n+1})^{\frac{2s-1}{2s}}r^{1/2+\alpha-\epsilon/2} f_1(x)\right),$$ 
where $r=(x_n^2+x_{n+1}^2)^{1/2}$, $f_0\in C^{0,\alpha}(0)$ and $f_1\in C^{0,\epsilon/2}(0)$ with $f_1(0)=0$.
This inhomogeneity however falls into the class of the inhomogeneities from Proposition \ref{prop:approx_main}. 
\end{rmk}

\begin{rmk}
In the sequel, we assume that $f(0)=0$ in the Dirichlet and Neumann cases and that $f_0(0)=0$ in the mixed Dirichlet-Neumann case. This can be achieved by subtracting the profiles $c_s f(0) x_{n+1}^{1+2s}$, $c_s f(0) x_{n+1}^{2}$ and $c_s f_0(0) w_{0,s}^{1+1/s}$.
\end{rmk}

\begin{rmk}
We point out that in the Dirichlet and Neumann boundary data cases, the approximation result of Proposition \ref{prop:approx_main} holds at all points $x_0\in \{x_{n+1}=0\}$, while in the mixed Dirichlet-Neumann case, it holds at all points $x_0\in P:=\{x_n=0=x_{n+1}\}$ (if the inhomogeneities $f$ satisfy suitable regularity assumptions at these points). This observation follows immediately from translation invariance of the problem in the corresponding directions.
\end{rmk}

In order to infer this result, we first approximate the inhomogeneous problem by the corresponding homogeneous one and then use the fact that homogeneous solutions are well-approximated by ``eigenpolynomials'' (i.e. homogeneous global solutions with corresponding Dirichlet Neumann boundary data).

\begin{lem}
\label{lem:homo}
Let $u$ be a solution of (\ref{eq:fra_DNM}) with 
\begin{align}
\label{eq:small}
\|u \|_{\LLw(B_1^+)} \leq 1, \ \|  x_{n+1}^{\frac{2s-1}{2}} g \|_{\LL(B_1^+)} \leq \delta.
\end{align}
For each $\epsilon>0$ there exists $\delta= \delta(\epsilon, n,s)>0$ such that if (\ref{eq:small}) is satisfied, then there exists a solution $h$ of the homogeneous equation
\begin{equation}
\label{eq:fra_DNM_h}
\begin{split}
L_s h & = 0 \mbox{ in } B_1^+,\\
B h &= 0 \mbox{ on } B_1',
\end{split}
\end{equation}
such that
\begin{align*}
\| u-h\|_{\LLw(B_\frac{1}{2}^+)} \leq \epsilon.
\end{align*}
Here $B\cdot$ denotes the Dirichlet, Neumann or mixed Dirichlet-Neumann operators from Proposition \ref{prop:approx_main}.
\end{lem}

\begin{proof}
In order to infer this result, we argue by contradiction. Assuming the statement were wrong, there existed $\bar{\epsilon}>0$ and sequences of solutions $u_k$ of (\ref{eq:fra_DNM}) with inhomogeneities $g_k$ such that
\begin{align*}
\| u_k\|_{\LLw(B_1^+)}\leq 1, \ \|x_{n+1}^{\frac{2s-1}{2}} g_k\|_{\LL(B_1^+)} \leq k^{-1},
\end{align*}
but
\begin{align}
\label{eq:contra}
\| u_k - h \|_{\LLw(B_{\frac{1}{2}}^+)} \geq \bar{\epsilon},
\end{align}
for any solution $h$ of the homogeneous problem (\ref{eq:fra_DNM_h}).
However, by energy estimates, all these solutions $u_k$ satisfy
\begin{align*}
\|\nabla u_k\|_{L^2_{\bar\omega}(B_{\frac{1}{2}}^+)}\leq C(\|u_k\|_{L^2_{\bar\omega}(B_1^+)} + \|x_{n+1}^{\frac{2s-1}{2}} g_k\|_{L^2(B_1^+)}) \leq C <\infty.
\end{align*}
Hence, on the one hand $u_k \rightharpoonup \bar{u}$ in $H^1_{\bar\omega}(B_{\frac{1}{2}}^+)$, with $\bar{u}$ being a weak solution of the corresponding homogeneous problem (\ref{eq:fra_DNM_h}). As such it enjoys higher regularity properties and in particular satisfies the boundary conditions in a pointwise sense. On the other hand, by compactness up to a subsequence the functions $u_k$ converge to $\bar{u}$ strongly in $L^2_{\bar \omega}(B_{\frac{1}{2}}^+)$. This contradicts the assumption (\ref{eq:contra}).
\end{proof}

As the next step towards the proof of Proposition \ref{prop:approx_main}, we approximate solutions of (\ref{eq:fra_DNM_h}) by ``eigenpolynomials'':

\begin{lem}[Eigenpolynomial approximation]
\label{lem:eigen}
Let $h$ be a solution of (\ref{eq:fra_DNM_h}) with $\| h\|_{\LLw(B_1^+)}\leq \bar{c}$. Then there exist solutions $h_{\beta_B}$ of (\ref{eq:fra_DNM_h}), which are of the form (\ref{eq:model_poly}) (c.f. Proposition \ref{prop:approx_main}), such that for all $r\in(0,1/2)$
\begin{align*}
\|h-h_{\beta_B}\|_{\LLw(B_r^+)} \leq C(\bar{c})r^{\beta_B + 1}.
\end{align*}
All coefficients of $h_{\beta_B}$ are bounded by $C \bar{c}$ with $C=C(n,s)$.
\end{lem}

\begin{proof}
We first prove that $h$ can be decomposed as
\begin{align}
\label{eq:decomp_eig_a}
h(x)= \sum\limits_{k=0}^{\infty} \alpha_k h_k(x) \mbox{ with } \sum\limits_{k=0}^{\infty}|\alpha_k|^2 \leq \bar{c}^2,
\end{align}
where the functions $h_k(x)$ denote the homogeneous solutions from Section \ref{sec:Eigenf}. Indeed, rewriting the equation (\ref{eq:fra_DNM_h}) in (standard) polar coordinates $(r,\theta)$ with $\theta_n:= \frac{y_{n+1}}{|y|}$ yields
\begin{align*}
\theta_n^{1-2s} r^{-n} \p_r (r^{n+1-2s}\p_ r) h +  r^{-1-2s} \nabla_{S^n} \theta_n^{1-2s} \nabla_{S^n} h & =0 \mbox{ in } S^{n}_+ \times \R_+, \\
B_{S^{n-1}} h & =0 \mbox{ on } S^{n-1}\times \R_+,
\end{align*}
where $B_{S^{n-1}} $ denotes the suitably transformed boundary data operator $B$ from (\ref{eq:fra_DNM_h}). Due to the compact embedding $H^1(\theta_n^{1-2s}d\theta, S^{n}_+)\hookrightarrow L^2(\theta_n^{1-2s}d\theta, S^n_+)$ (and due to the form of the boundary data), there is an orthonormal basis of $L^2(\theta_n^{1-2s}d\theta, S^n_+)$ consisting of eigenfunctions $\{h_m\}_{m\in\N}$ of the spherical operator as well as an associated discrete set of eigenvalues $\lambda_m$, i.e., the functions $h_m$ and the values $\lambda_m$ satisfy
\begin{align*}
\nabla_{S^n} \theta_n^{1-2s} \nabla_{S^n} h_m  & = -\lambda^2_m \theta_n^{1-2s} h_m \text{ in } S^n_+,\\
B_{S^{n-1}} h &= 0 \mbox{ on } S^{n-1}.
\end{align*}
Thus, $h$ can be expanded into these eigenfunctions:
\begin{align}
\label{eq:decomp_eig1}
h(r,\theta) = \sum\limits_{m} \alpha_m(r) h_m(\theta).
\end{align}
By orthogonality the functions $\alpha_m(r)$ satisfy
\begin{align*}
r^{-n} \p_r (r^{n+1-2s}\p_ r) \alpha_m  -\lambda_m^2 r^{-1-2s} \alpha_m = 0,
\end{align*}
and are hence homogeneous. As a consequence the functions $\alpha_m(r)h_m(\theta)$ are homogeneous solutions to (\ref{eq:fra_DNM_h}) which satisfy the boundary conditions (which implies that the homogeneity $\kappa$ is larger than or equal to zero). Homogeneous solutions to (\ref{eq:fra_DNM_h}) are however exactly the ones which are classified in Propositions \ref{prop:nD_mDN} and \ref{prop:DNnd}. Combining this with (\ref{eq:decomp_eig1}) shows the existence of the claimed decomposition (\ref{eq:decomp_eig_a}).\\
Building on this decomposition we prove the claim of the lemma: As the functions $h_k$ are homogeneous, orthogonal with respect to the $L^2_{\bar\omega}(B_1)$ scalar product and normalized on $B_1^+$, the result follows by setting $h_{\beta_B}(x):= \sum\limits_{k=0}^{\beta_B} \alpha_k h_k(x)$:
\begin{align*}
\|h-h_{\beta_B}\|_{L^2_{\bar\omega}(B_r^+)}^2 &= \sum\limits_{k=\beta_B+1}^{\infty} |\alpha_k|^2 \| h_k\|_{L^2_{\bar\omega}(B_r^+)}^2 \\
&\leq r^{2(\beta_B+1) + 2(n+1)+1-2s}  C(\bar{c}).
\end{align*}
This concludes the proof.
\end{proof}

Combining the results of the previous two lemmas, we obtain the following key approximation lemma. An iteration of it yields the proof of Proposition \ref{prop:approx_main}.

\begin{lem}[Iteration]
\label{lem:it}
There exist $\delta>0$ and a radius $r_0\in (0,1)$ such that for any solution $u$ of (\ref{eq:fra_DNM}) with $\| u \|_{\LLw(B_1^+)}\leq 1$ and 
\begin{align*}
\| x_{n+1}^{\frac{2s-1}{2}}g \|_{\LL(B_1^+)} \leq \delta,
\end{align*}
there exists a sum of homogeneous functions $h_{\beta_B}$ as in (\ref{eq:model_poly}) in Proposition \ref{prop:approx_main} with all coefficients bounded by a uniform constant $C=C(n,s,\alpha)$ such that
\begin{align*}
\| u - h_{\beta_B} \|_{\LLw(B_{r_0}^+)} \leq r_0^{\beta_{B}+\alpha}.
\end{align*}
Here $\delta_0$ and $r_0$ are constants depending only on $n,s,\alpha$.
\end{lem}

\begin{proof}
For a sufficiently small $\epsilon>0$ which will be determined later, Lemma \ref{lem:homo} implies the existence of $\delta=\delta(\epsilon,n,s)>0$ and a homogeneous solution $h$ of (\ref{eq:fra_DNM_h}) such that if $\|x_{n+1}^{\frac{2s-1}{2}}g\|_{\LL(B_1^+)}\leq \delta$, we have
\begin{align*}
\|u-h\|_{\LLw(B_{\frac{1}{2}}^+)} \leq \epsilon.
\end{align*}
Moreover, Lemma \ref{lem:eigen} then yields a sum of homogeneous solutions of the form (\ref{eq:model_poly}) from Proposition \ref{prop:approx_main} such that for all $r\in(0,1/2)$
\begin{align*}
\| h- h_{\beta_B}\|_{\LLw(B_{r}^+)} \leq C r^{\beta_B + 1 },
\end{align*}
where $C$ is an absolute constant (since by the triangle inequality $\|h\|_{L^2_{\bar{\omega}}(B_{1/2}^+)}\leq 2$). Thus, the triangle inequality leads to
\begin{align*}
\|  u-h_{\beta_B}\|_{\LLw(B_r^+)} &\leq \| u-h\|_{\LLw(B_r^+)} + \|h-h_{\beta_B}\|_{\LLw(B_r^+)}\\
& \leq \epsilon + C r^{\beta_B+1 },\quad r\in (0,1/2).
\end{align*}
Choosing first $r_0\in(0,1/2)$ such that $C r_0^{\beta_B+1 } \leq \frac{1}{2} r_0^{\beta_B + \alpha }$ and then $\epsilon$ such that $ \epsilon \leq \frac{1}{2}r_0^{\beta_B + \alpha }$ gives the desired estimate. Here the constants $\delta$ and $r_0$ depend on $n,s$ and $\alpha$.
\end{proof}

Iterating this result and exploiting the structure of the right hand side, then yields the proof of Proposition \ref{prop:approx_main}:

\begin{proof}[Proof of Proposition \ref{prop:approx_main}]
It suffices to prove the following iteration statement: If $\| u\|_{\LLw(B_1^+)}\leq 1 $ and if for some sufficiently small $\delta>0$ (which can be chosen as in Lemma~\ref{lem:it})
\begin{align}
\label{eq:assume}
\|x_{n+1}^{\frac{2s-1}{2}}g\|_{\LL(B_{r_0^k})} \leq \delta r_0^{k\left( \beta_B+ \alpha - \frac{3}{2}-s \right)},
\end{align}
where $\beta_B$ is the exponent from Proposition \ref{prop:approx_main},
then there exist solutions $h_{\beta_B,k}$ which satisfy the corresponding boundary data and which are as in (\ref{eq:model_poly}) in Proposition \ref{prop:approx_main} such that
\begin{align}
\label{eq:ind_assu}
\|u-h_{\beta_B,k} \|_{\LLw(B_{r_0^k})} \leq r_0^{k(\beta_B + \alpha)}.
\end{align}
Once this is shown the remainder of the proof is similar as in \cite{Wa03} or \cite{KRSIII}. In order to derive this claim, we argue by induction. As the base case corresponds to the statement of Lemma \ref{lem:it}, it suffices to prove the step from $k$ to $k+1$. To this end, let $h_{\beta_B,k}$ be the approximating solution at step $k$. We consider
\begin{align*}
u_{k}(x):= \frac{(u-h_{\beta_B,k})(r_0^k x)}{r_0^{k(\beta_B+\alpha)}}.
\end{align*}
By the inductive assumption assumption (\ref{eq:ind_assu}) we have that $\| u_k\|_{\LLw(B_1^+)}\leq 1$. Moreover, 
\begin{align*}
\nabla \cdot x_{n+1}^{1-2s} \nabla u_k = r_0^{-k(\beta_B+\alpha)} r_0^{(1+2s)k} g(r_0^k x)=: g_{r_0}(x).
\end{align*}
By (\ref{eq:assume})
\begin{align*}
\|x_{n+1}^{\frac{2s-1}{2}}g_{r_0}\|_{\LL(B_1^+)} = r_0^{-k(\beta_B+ \alpha)} r_0^{\left(\frac{3}{2}+s\right)k}\|x_{n+1}^{\frac{2s-1}{2}} g\|_{\LL(B_{r_0^k})} \leq \delta.
\end{align*}
Hence, Lemma \ref{lem:it} is applicable and yields a solution $\tilde{h}_{\beta_B,k+1}$ which is of the form of (\ref{eq:model_poly}) such that
\begin{align*}
\|u_k - \tilde{h}_{\beta_B,k+1}\|_{\LLw(B_1^+)} \leq r^{\beta_B + \alpha }.
\end{align*}
Rescaling and setting 
\begin{align*}
h_{\beta_B,k+1}(x):= h_k(x) + r^{\beta_B+\alpha}_0 \tilde{h}_{\beta_B,k+1}\left( \frac{x}{r_0^k}\right),
\end{align*}
yields the claim. Using the geometric decay of the coefficients of $\tilde{h}_{\beta_B,k+1}$ hence allows us to find a limiting function $h_{\beta_B,\infty}(x):= \lim\limits_{k\rightarrow \infty}h_{\beta_B,k}(x)$ which is still a solution and of the desired form (\ref{eq:model_poly}) and satisfies the right boundary conditions. We remark that this iteration procedure is applicable in the setting of Proposition \ref{prop:approx_main} as scaling allows us to assume that the inductive hypotheses are satisfied.
\end{proof}

\subsection{Eigenfunctions and eigenvalues}
\label{sec:Eigenf}
 In the following two sections we discuss the form of the global homogeneous solutions to $L_sw=0$ in the upper half-plane with Dirichlet, Neumann and mixed Dirichlet-Neumann data on the boundary. We recall that these results played a crucial role for our approximation arguments in the regularity statements of Section \ref{sec:reg1} and also of Section \ref{sec:open}.\\

\subsubsection{Mixed Dirichlet-Neumann data}
\label{sec:2Deigen}
In this section we compute homogeneous solutions to the mixed Dirichlet-Neumann problem. In the following section we then deal with the Dirichlet and Neumann problems, respectively. In deducing the approximation result for the mixed Dirichlet-Neumann problem, we argue in two steps: We first compute the solutions in the two-dimensional set-up (c.f. Section \ref{sec:2Deigen}) and then exploit the translation invariance in tangential directions of our problem to infer an analogous $(n+1)$-dimensional result from that.\\

We begin by considering the mixed Dirichlet-Neumann problem for the fractional Laplacian with $s\in(0,1)$ in the two-dimensional upper half plane:
\begin{equation}
\label{eq:fracLa_lapl}
\begin{split}
(\p_1 x_{2}^{1-2s} \p_1 + \p_2 x_{2}^{1-2s} \p_2) u &= 0 \mbox{ in } \R^{2}_+,\\
u(x_1,0)&=0 \mbox{ on } \{x_1 \leq 0\} \cap (\R\times\{0\}),\\
\lim\limits_{x_2 \rightarrow 0_+} x_2^{1-2s} \p_{2} u(x_1,x_2) &= 0 \mbox{ on } \{x_1 > 0\}\cap (\R\times \{0\}). 
\end{split}
\end{equation}

In the polar coordinates $(x_1,x_2)=(r\cos\va, r\sin \va)$, $\va\in [0,\pi]$, and with a separation of variables ansatz $u(r,\va)=u_1(r)u_2(\va)$ the bulk equation in spherical variables reads
\begin{equation}
\label{eq:sep}
\begin{split}
(\p_r^2 + (2-2s)r^{-1}\p_r) u_1(r) &= \lambda^2 r^{-2} u_1(r),\\
\sin(\va)^{2s-1} \dv (\sin(\va)^{1-2s}\dv)u_2(\va) & = - \lambda^2 u_2(\va),
\end{split}
\end{equation}
(we stress that the separation ansatz is justified here, as the spherical operator forms a basis of $L^2(S^1_+, \sin(\va)^{1-2s})$ and as the separation ansatz essentially corresponds to an expansion into these eigenfunctions). The Neumann condition becomes
\begin{align*}
\lim\limits_{\va \rightarrow 0} (\sin(\va))^{1-2s} \dv u_2(\va) = 0.
\end{align*}
In the sequel we focus on the spherical part of the problem and determine the corresponding spherical eigenfunctions:

\begin{lem}[$2D$ spherical eigenfunctions]
\label{lem:2Deigen}
Let $u_2(\va)$ be a solution of 
\begin{equation*}
\begin{split}
\sin(\va)^{2s-1} \dv (\sin(\va)^{1-2s}\dv)u_2(\va) & = - \lambda^2 u_2(\va) \mbox{ for } \va \in (0,\pi),\\
u_2(\pi)&=0,\\
\lim\limits_{\va \rightarrow 0} (\sin(\va))^{1-2s} \dv u_2(\va) &= 0.
\end{split}
\end{equation*}
Then, the eigenvalue $\lambda^2$ has the form
\begin{align*}
\lambda^2 = k(k+1)-s(s-1) \mbox{ for some } k \in \N.
\end{align*}
The associated spherical eigenfunction is given as
\begin{align*}
 u_{2}(\va)=C \left(\frac{1+\cos(\va)}{2}\right)^s F\left(1-s-k,k+s,1+s;\frac{\cos(\va)+1}{2}\right),
\end{align*}
where $F$ is a hypergeometric function. Moreover, the hypergeometric function $F(1-s-k,k+s,1+s;z)$ is a polynomial of degree $k$ in $z$.
\end{lem}

\begin{proof}
In order to prove the lemma, we consider the following change of variables: We set $x=\cos(\va)$ and define $u_2(\va)=:v(\cos(\va))$. In these coordinates the equations become
\begin{align*}
(1-x^2)v''(x) + (2s-2)xv'(x) + \lambda^2 v(x)&=0 \mbox{ for } x\in[-1,1],\\
v(-1)&=0,\\
\lim\limits_{x\rightarrow 1_-}(1-x^2)^{1-s}v'(x)&=0.
\end{align*}
This equation has three regular singular points at $x=\pm 1, \infty$.
By further defining $z=\frac{x+1}{2}$ and $v(x)=:w(z)$, we transform this into a standard hypergeometric equation with the three regular singular points $z=0,1,\infty$:
\begin{equation}
\label{eq:eqtrafo}
\begin{split}
z(1-z)w''(z) + ((1-s)+2(s-1)z)w'(z) + \lambda^2 w(z) & = 0 \mbox{ for } z\in [0,1],\\
w(0) &= 0,\\
\lim\limits_{z\rightarrow 1_-}(z(1-z))^{1-s}\p_z w(z) & = 0.
\end{split}
\end{equation}
The general solution of the bulk equation is given as
\begin{align}
\label{eq:sol}
w(z) & = A F(a,b,c;z) + B z^{1-c} F(a+1-c,b+1-c,2-c;z), \ A,B \in \R,
\end{align}
where $F$ denotes the hypergeometric function and where
\begin{align*}
a&= \frac{1}{2}(1-2s + \sqrt{4\lambda^2 + 4s^2 - 4s+1}),\\
b&= \frac{1}{2}(1-2s - \sqrt{4\lambda^2 + 4s^2 - 4s+1}),\\
c&=1-s.
\end{align*}
As $F(a,b,c;0)=1$, the Dirichlet boundary conditions immediately imply that $A=0$ and thus, for some $B\in \R$,
\begin{align*}
w(z) & = B z^{1-c} F(a+1-c,b+1-c,2-c;z) = B z^{s}F(a+s,b+s,1+s;z).
\end{align*}
In order to determine the possible values of $\lambda$, we now use the Neumann condition. To this end, we recall the following relations for hypergeometric functions:
\begin{equation}
\label{eq:hyp}
\begin{split}
\p_z F(a,b,c;z)&= \frac{ab}{c}F(a+1,b+1,c+1;z),\\
F(a,b,c;z) &= \frac{\Gamma(c)\Gamma(c-a-b)}{\Gamma(c-a)\Gamma(c-b)} F(a,b,a+b-c+1;1-z) \\
& \quad + (1-z)^{c-a-b} \frac{\Gamma(c)\Gamma(a+b-c)}{\Gamma(a)\Gamma(b)} F(c-a,c-b,c-a-b+1;1-z).
\end{split}
\end{equation}
Thus, $\p_z w(z)$ turns into
\begin{align}
\label{eq:Neu}
\p_z w(z) = s z^{s-1} F(a+s,b+s,1+s;z) + z^{s}\frac{ab}{1+s}F(a+s+1,b+s+1,2+s;z).
\end{align}
We consider the two terms separately. For the first contribution we note that
\begin{align*}
F(a+s,b+s,1+s;z) & = \frac{\Gamma(1+s)\Gamma(s)}{\Gamma(1-a)\Gamma(1-b)} F(a+s,b+s,1-s;1-z)\\
& \quad  + (1-z)^{s}\frac{\Gamma(1+s)\Gamma(1-s)}{\Gamma(a+s)\Gamma(b+s)}F(1-a,1-b,1+s;1-z).
\end{align*}
In the relevant Neumann derivative (coming from the equation) the previous expression is weighted with the vanishing factor $(1-z)^{1-s}$ (we recall that $s\in(0,1)$). Since the prefactors in the expression for $F(a+s,b+s,1+s;z)$ are all finite and as $F(a,b,c;0)=1$, the first part in (\ref{eq:Neu}) always satisfies the boundary conditions. As a consequence, we turn to the second contribution from (\ref{eq:Neu}). We have:
\begin{align*}
& F(a+s+1,b+s+1,2+s;z) = \frac{\Gamma(2+s)\Gamma(1+s)}{\Gamma(1-a)\Gamma(1-b)} F(1+a+s,1+b+s,2-s;1-z)\\
& \quad  + (1-z)^{s-1}\frac{\Gamma(2+s)\Gamma(2-s)}{\Gamma(1+a+s)\Gamma(1+b+s)}F(2-a,2-b,2+s;1-z).
\end{align*}
Similarly as above, the first summand vanishes in the limit $z\rightarrow 0$ if it is multiplied with the weight $(1-z)^{1-s}$. Therefore, it suffices to consider the second term which does not vanish in the limit $z\rightarrow 0$ unless the prefactor
\begin{align*}
\frac{\Gamma(2+s)\Gamma(2-s)}{\Gamma(1+a+s)\Gamma(1+b+s)}
\end{align*}
vanishes. This is the case iff at least one of the $\Gamma$-functions in the denominator explodes (i.e. iff at least one of the arguments of the $\Gamma$-functions in the denominator is a negative integer). Plugging in the definition of $a,b$, this is the case iff 
\begin{align*}
\frac{1}{2} \pm \frac{1}{2} \sqrt{1+4 \lambda^2 + 4s^2 - 4s} = - k, \ k\in \N.
\end{align*}
This however is equivalent to 
\begin{align*}
\lambda^2 = k^2 + k +s-s^2.
\end{align*}
The result on the eigenfunction representation in terms of the corresponding hypergeometric functions therefore follows from inserting these values of $\lambda$ into the expressions for $a,b$.\\
We prove that $F(1-s-k,k+s,1+s,z)$ is a polynomial of degree $k$: To this end, we first set $w(z)=z^{s}h(z)$ (where $w(z)$ is the solution of the transformed equation (\ref{eq:eqtrafo})). Inserting this into the equation for $w$, we deduce that $h(z)$ satisfies
\begin{align*}
z(1-z)h''(z) +(1+s-2z)h'(z)+k(k+1)h(z)=0.
\end{align*}
Making a series ansatz, $h(z)=\sum\limits_{m=0}^{\infty} a_m z^m$ thus yields
\begin{align*}
0&=\sum\limits_{m=0}^{\infty}[\chi_{\{m\geq 1\}}m(m+1)a_{m+1} - \chi_{\{m\geq 2\}}m(m-1)a_m+(1+s)(m+1)a_{m+1} \\
& \quad - 2\chi_{\{m\geq 1\}}m a_m + k(k+1)a_m]z^m.
\end{align*}
For a prescribed non-zero value of $a_0$ this corresponds to the following system of equations for the coefficients $a_m$:
\begin{align*}
(1+s)a_1 + k(k+1)a_0 & = 0,\\
(4+2s)a_2+(k(k+1)-2)a_1 & = 0,\\
(m^2+2m+ms+1+s)a_{m+1} - (m(m+1)-k(k+1))a_m & =0. 
\end{align*}
As $m^2+2m+ms+1+s\neq 0$ for $m\in \N$, $s\in (0,1)$, this system is up to order $m=k$ uniquely solvable for given $a_0$. Moreover, we note that it is possible to choose $a_m=0$ for all $m\geq k+1$. This yields the claimed polynomial form of the hypergeometric function $F(1-s-k,k+s,1+s,z)$.
\end{proof}

As a corollary of Lemma \ref{lem:2Deigen} we obtain the following result on the structure of 2D homogeneous solutions to (\ref{eq:fracLa_lapl}):

\begin{cor}[$2D$ homogeneous solutions]
\label{cor:2D_eigen}
Let $u:\R^{2}_+ \rightarrow \R$ be a $\kappa$-homogeneous solution of (\ref{eq:fracLa_lapl}) with $\kappa\geq 0$. Then,
\begin{align*}
\kappa = k+s, \mbox{ for some } k\in \N,
\end{align*}
and $u$ has the form
\begin{align*}
u(x) = C |x|^{k+s} \left(\frac{1+\frac{x_1}{|x|}}{2}\right)^s F\left(1-s-k,k+s,1+s;\frac{\frac{x_1}{|x|}+1}{2}\right),
\end{align*}
where $F$ denotes the hypergeometric function from \ref{lem:2Deigen}.
By the observation on the polynomial structure of the relevant hypergeometric function $F$, this can also be rewritten as
\begin{align*}
u(x)= w_{0,s}(x_1,x_2)P_k(x_1,|x|),
\end{align*}
where $P_k$ is a polynomial and $w_{0,s}$ is the function from Section \ref{sec:asymp}.
\end{cor}

\begin{rmk}
\label{rmk:DNm}
For later reference, we note that  for instance for $k=0,1$ we have
\begin{align*}
u_0(x) &= c_0 w_{0,s}(x_1,x_2), \\
u_1(x) & = c_1 w_{0,s}(x_1,x_2)\left(s|x|-x_n\right).
\end{align*}
Here we used the series approach from the proof of Lemma \ref{lem:2Deigen} to compute the coefficients of $u_1$. We note that these functions correspond to the ones from the asymptotic expansion in Proposition \ref{prop:asymp1}.
\end{rmk}

\begin{rmk}[Orthogonality]
We note that the spherical eigenfunctions are pairwise orthogonal with respect to the $L^2((\sin(\varphi))^{1-2s} d \varphi, [0,\pi])$ scalar product. This entails that the homogeneous solutions from Corollary \ref{cor:2D_eigen} are orthogonal with respect to the $L^2(x_{2}^{1-2s}dx)$ scalar product on $B_1^+\subset \R^{2}_+$.
\end{rmk}

\begin{proof}
The corollary is an immediate consequence of the form of the solutions in Lemma \ref{lem:2Deigen}, the fact that $x_1=r\cos(\varphi)$ and of the equation for the radial component of $u_k$.
\end{proof}

Relying on the two-dimensional result from above, we now proceed to determining the full set of $(n+1)$-dimensional homogeneous solutions for the mixed Dirichlet-Neumann problem.

\begin{prop}[Homogeneous solutions in $\R^{n+1}_+$]
\label{prop:nD_mDN}
Let $u: \R^{n+1}_+ \rightarrow \R$ be a $\kappa$-homogeneous solution of
\begin{equation}
\label{eq:fr_eig}
\begin{split}
\nabla \cdot x_{n+1}^{1-2s} \nabla u & = 0 \mbox{ in } \R^{n+1}_+,\\
u & = 0 \mbox{ on } \{x_{n}\leq 0\}\cap (\R^{n}\times \{0\}),\\
\lim\limits_{x_{n+1}\rightarrow 0} x_{n+1}^{1-2s} \p_{n+1} u & = 0 \mbox{ on } \{x_{n}\geq 0\}\cap (\R^{n}\times \{0\}),
\end{split}
\end{equation} 
with $\kappa \geq 0$.
Then, the possible homogeneities are of the form
\begin{align*}
\kappa= s+ \ell \mbox{ with } \ell \in \N.
\end{align*}
The corresponding $\kappa$ homogeneous solutions are
\begin{align*}
u_{\kappa}(x)= \sum\limits_{\substack{\kappa=m+d+s\\ d-2k \geq 0}} |(x_n,x_{n+1})|^{2k }P_{d-2k}(x'')u_{m}(x),
\end{align*}
where $d,k\in \N\cup\{0\}$,
\begin{align*}
u_m(x)=|(x_n,x_{n+1})|^{m}\left(\frac{|(x_n,x_{n+1})|+x_n}{2}\right)^s F\left(1-s-m,m+s,1+s;\frac{\frac{x_n}{|x|}+1}{2}\right),
\end{align*} 
and $P_{l}(x'')$ denotes a $l$-homogeneous polynomial.
In particular, a general $(n+1)$-dimensional eigenfunction can also be represented as
\begin{align*}
\tilde{u}_{\kappa}(x) = \sum\limits_{k+d=\kappa} w_{s,0}(x_n,x_{n+1}) P_k(x_n,\sqrt{x_n^2+x_{n+1}^2})P_d(x),
\end{align*}
where $P_k$ is one of the $k$-homogeneous polynomials from Corollary \ref{cor:2D_eigen} and $P_d(x'')$ denotes a polynomial of degree $d$.
\end{prop}

\begin{proof}
We begin by introducing new coordinates $(x'',r,\va)$ which are defined as
\begin{align*}
(x'',x_n,x_{n+1}) = (x'',r\cos(\va), r\sin(\va)).
\end{align*}
Dividing our equation (\ref{eq:fr_eig}) by $x_{n+1}^{1-2s}$ and rewriting it in the new variables leads to
\begin{align}
\label{eq:fr_eig1}
(\Delta''+  (\p_r^2 + (2-2s)r^{-1}\p_r)  +  r^{-2}(\sin(\varphi)^{2s-1} \partial_{\va}(\sin(\varphi)^{1-2s})\p_{\varphi})) u=0.
\end{align}
Let $u_m(\va)$ denote the functions from Lemma \ref{lem:2Deigen}. As they form an orthogonal basis of $L^2(S^1_+, \sin(\va)^{1-2s}d\va)$ (as they are eigenfunctions of an associated Sturm-Liouville operator), we obtain an expansion
\begin{align}
\label{eq:series_m}
u(x'',r,\va) = \sum\limits_{m=1}^{\infty} c_m(x'',r)u_m(\va).
\end{align}
By orthogonality of the functions $u_m(\va)$, each of the functions $c_m(x'',r)$ solves
\begin{align}
\label{eq:cm0}
(-\D'' + r^2(\p_r^2 + (2-2s) r^{-1}\p_r ) - \lambda_m^2) c_m(x'',r) = 0,
\end{align}
with $\lambda_m^2 = m^2 +m-s^2+s$. Moreover, again by orthogonality and by the homogeneity of $u$ each of the functions $c_m$ is $\kappa$ homogeneous.\\
Since the problem is translation invariant in tangential variables $x''$, we carry out a Fourier transform of $c_m(x'',r)$ (interpreted as a Fourier transform on tempered distributions) in the tangential variables $x''$ and denote the partial Fourier transform of $c_m$ by 
\begin{align*}
\hat{c}_m(\xi'',r).
\end{align*}
We note that by the $\kappa$-homogeneity the functions $c_m$, we obtain that
\begin{align}
\label{eq:hom5}
\hat{c}_m(\lambda^{-1}\xi'',\lambda r) = \lambda^{\kappa+n-1} \hat{c}_m(\xi'', r).
\end{align}
Therefore,
\begin{align}
\label{eq:hom4}
|\hat{c}_{m}(\xi'',r)| \leq C \max\{r^{\kappa+n-1},|\xi''|^{-\kappa-n+1}\}.
\end{align}
After the Fourier transform the equation (\ref{eq:cm0}) for $c_m$ reads
\begin{align}
\label{eq:cm}
(-r^2|\xi''|^2 + r^2(\p_r^2 + (2-2s) r^{-1}\p_r ) - \lambda_m^2) \hat{c}_m(\xi'',r) = 0,
\end{align}
with $\lambda_m^2 = m^2 +m-s^2+s$.  \\
Considering the ansatz, $\hat{c}_m(\xi'',r)= r^{\frac{2s-1}{2}} f_m(\xi'',r)$, we deduce that $f_m(\xi'',r)$ satisfies a modified Bessel equation:
\begin{align*}
r^2 f''_m(\xi'',r) + r f'_m(\xi'',r) - f_m(\xi'',r) \left(\lambda_m^2 + \frac{(1-2s)^2}{4} +r^2 |\xi''|^2 \right)  = 0.
\end{align*}
A fundamental system of this ODE is given by the modified Bessel functions 
\begin{align*}
f_m(\xi'',r) = d_1(\xi'') I_{\frac{1+2m}{2}}(|\xi''|r) + d_2(\xi'') K_{\frac{1+2m}{2}}(|\xi''|r) .
\end{align*}
These functions satisfy the following asymptotics:
\begin{align*}
I_{\nu}(x) &\sim (\Gamma(\nu+1))^{-1} \left( \frac{x}{2} \right)^{\nu}, \ K_{\nu}(x) \sim \frac{1}{2} \Gamma(\nu)\left( \frac{x}{2} \right)^{-\nu} \mbox{ for } x \rightarrow 0 \mbox{ and } \nu\geq 0,\\
I_{\nu}(x) &\sim \frac{e^{x}}{\sqrt{2\pi x}}, \ K_{\nu}(x) \sim \frac{e^{-x}\sqrt{\pi}}{\sqrt{2x}} \mbox{ for } |x| \rightarrow \infty. 
\end{align*}
Combining this asymptotic behavior with the bounds from (\ref{eq:hom4}), we infer that $f_m(\xi'',r)$ is only supported in $|\xi''|=0$ (this follows by considering the limits $r\rightarrow 0$, and $r \rightarrow \infty$ for fixed $\xi''\neq 0$). 
Thus, 
\begin{align}
\label{eq:sumDN}
\hat c_{m}(\xi'',r) = \sum\limits_{k=0}^{\infty}\sum\limits_{|\alpha|=k}  c_{m,\alpha}(r)\delta^{(\alpha)}_{\{\xi''=0\}}(\xi'').
\end{align}
Here $\alpha= (\alpha_1,\dots,\alpha_{n-1})$ is a multi-index with $\alpha_j\in \N\cup\{0\}$ for all $j\in\{1,\dots,n-1\}$, and $\delta^{(\alpha)}_{\{\xi''=0\}}(\xi'')$ denotes the distribution which is obtained by taking $\alpha$ distributional derivatives of the delta distribution $\delta_{\{\xi''=0\}}$. Using the homogeneity of $\hat{c}_m$ (c.f. (\ref{eq:hom5})), we deduce that $c_{m,\alpha}( r) $ is $\kappa+n-1-k$ homogeneous. However the Dirichlet data (which hold on part of the domain) require that $k\leq \kappa+n-1$, whence we observe that the series in (\ref{eq:sumDN}) is a finite sum. We use $K_0(m)$ to denote the largest positive integer less than $\kappa+n-1$ with $c_{m,K_0(m)}(r)\neq 0$. Then,
\begin{align}
\label{eq:sumDN}
\hat{c}_{m}(\xi'',r) = \sum\limits_{k=0}^{K_0(m)}\sum\limits_{|\alpha|=k} c_{m,\alpha}(r) \delta^{(\alpha)}_{\{\xi''=0\}}(\xi'').
\end{align}
Here $[\cdot]$ denotes the floor function.\\
We plug this expression back into (\ref{eq:cm}) and test it with a smooth, compactly supported function $\va$. This yields
\begin{equation}
\label{eq:recur_DN}
\begin{split}
&\sum\limits_{k=0}^{K_0(m)}\sum\limits_{|\alpha|=k} [- [(D^{\alpha}_{\xi''}(|\xi''|^2 \va))|_{\xi''=0}] r^2  c_{m,\alpha}(r) \\
& \quad +[(D^{\alpha}_{\xi''} \va)(0)] (r^2 \p_r^2 + (2-2s)r^{-1}\p_r - \lambda_m^2)c_{m,\alpha}] =0.
\end{split}
\end{equation}
Successively inserting the test functions $\va_j(\xi'')=|\xi''|^{j}$ for $j\in\{1,\dots,K_0(m)\}$ with $\varphi_j$ being extended away from zero to have compact support in $\xi''$ (and beginning the testing with large $j$ first and then decreasing the power in each step), we obtain the following equations for $c_{m,\alpha}(r)$: For $\tilde{\alpha}$ with $|\tilde{\alpha}|\in\{K_0(m),K_0(m)-1\}$ we have
\begin{align*}
(r^2 \p_r^2 + (2-2s)r^{-1}\p_r - \lambda_m^2)c_{m,\tilde{\alpha}}(r) =0,
\end{align*}
from which we obtain $c_{m,\tilde{\alpha}}(r) = \tilde{c}_{m,\tilde{\alpha}} r^{m+s}$ with an absolute constant $\tilde{c}_{m,\tilde{\alpha}}$ in both cases. As $c_{m}$ is however homogeneous (in the sense of (\ref{eq:hom5})), this implies that either $c_{m,\tilde{\alpha}}(r)=0$ for all $\tilde{\alpha}$ with $|\tilde{\alpha}|=K_0(m)$ or $c_{m,\bar{\alpha}}(r)=0$ for all $\bar{\alpha}$ with $|\bar{\alpha}|=K_0(m)-1$ (as the functions $c_{m,\tilde{\alpha}}\delta^{(\tilde{\alpha})}_{\{\xi''=0\}}$ with $|\tilde{\alpha}|=K_0(m)$ and $c_{m,\bar{\alpha}}\delta^{\bar{\alpha}}_{\{\xi''=0\}}$ with $|\bar{\alpha}| =K_0(m)-1$ can else not have the same homogeneity). We assume that the second case holds (the other one is analogous). \\
Again invoking (\ref{eq:hom5}), we obtain a condition for $K_0(m)$ depending on $m,\kappa,n$:  
\begin{align*}
K_0(m)=\kappa+n-1-s-m\geq 0.
\end{align*}
This in particular entails that $\kappa+s\in \N$. Further we note that with increasing $m$ the value of $K_0(m)$ decreases. Hence $\hat{c}_m\neq 0$ only for finitely many values of $m$ (depending on $\kappa$). In particular the sum in (\ref{eq:series_m}) is finite.\\
We return to the condition on the coefficients $c_{m,\alpha}(r)$: Evaluating (\ref{eq:recur_DN}) for $\va_k$ with $k\in\{1,\dots,K_0(m)-2\}$, we then have
\begin{align*}
(r^2 \p_r^2 + (2-2s)r^{-1}\p_r - \lambda_m^2)c_{m,\alpha}(r) = d_{m,\Tilde{\alpha}} c_{m,\tilde{\alpha}}(r),
\end{align*}
where $|\alpha|=k$ and $\tilde{\alpha}=k+2$ and $d_{m,\alpha}$ is an absolute constant. Integrating this iteratively and recalling that $c_{m}$ obeys the homogeneity condition (\ref{eq:hom5}), we thus deduce that for $|\alpha|=K_0(m)-l$  we have $c_{m,\alpha}(r)= d_{m,l} r^{m+s+2(l-1+[l/2])}$. By homogeneity of $c_{m}$, we conclude that $d_{m,2j+1}=0$ for all $j\in\{1,\dots,[(K_0(m)-2)/2]\}$. 
Therefore, (\ref{eq:sumDN}) turns into
\begin{align*}
\hat{c}_m(\xi'',r) = \sum\limits_{\substack{k\in 2\N \cup\{0\},\\ k\leq K_0(m)/2}}\sum\limits_{|\alpha|=K_0(m)-2k} r^{m+s+2k} \delta^{(\alpha)}_{\{\xi''=0\}}(\xi'').
\end{align*}
Inserting this information, using that the inverse Fourier transform of $\delta^{(\alpha)}_{\{\xi''=0\}}$ is a homogeneous polynomial of degree $|\alpha|$, and transforming back into $(x'',r,\va)$ coordinates then yields
\begin{align*}
c_m(x'',r) =   \sum\limits_{\substack{k\in 2\N \cup\{0\},\\ k\leq (K_0(m)-n+1)/2}} r^{m+s+2k }P_{K_0(m)-2k-n+1}(x'')u_{m}(\va).
\end{align*}
Here $P_{K_0(m)-2k-n+1}$ is a polynomial of degree $K_0(m)-2k-n+1\geq 0$.
As 
\begin{align*}
u(x'',r,\va) = \sum\limits_{\substack{m,\\ K_0(m)>0}} \sum\limits_{\substack{k\in 2\N \cup\{0\},\\ k\leq (K_0(m)-n+1)/2}} r^{m+s+2k }P_{K_0(m)-2k-n+1}(x'')u_{m}(\va),
\end{align*}
where $P_{[\kappa+n+1]+2k}(x'')$ is a polynomial of degree $2k$ and $\kappa-[\kappa+n+1]-s\in \N$, this concludes the proof of the proposition.
\end{proof}

\subsubsection{Dirichlet and Neumann data}
In this section we determine all homogeneous solutions to the $(n+1)$-dimensional Dirichlet and Neumann problems for the fractional Laplacian. This is slightly less involved than the argument for the mixed Dirichlet-Neumann problem, as the problem has only one ``broken symmetry'' originating from the operator (which is inhomogeneously weighted in the $x'$ and $x_{n+1}$ directions). In contrast in the mixed Dirichlet-Neumann case we had to deal with two directions of ``symmetry loss'' as the boundary data caused an additional direction with loss of symmetry.\\

For the Dirichlet and Neumann problems our main result is:

\begin{prop}[Dirichlet, Neumann homogeneous solutions in $\R^{n+1}_+$] 
\label{prop:DNnd}
Let $u_{D,N}:\R^{n+1}_+ \rightarrow \R$ be a $\kappa$-homogeneous solution of
\begin{align*}
\nabla \cdot x_2^{1-2s} \nabla u_{D,N} &= 0 \mbox{ in } \R^{2}_+,\\
B_{D,N} u & = 0 \mbox{ in } \R \times \{0\}.
\end{align*}
with $\kappa\geq 0$.
Here $B_{D}u_D = u_D $ and $B_N u_N = \lim\limits_{x_{n+1} \rightarrow 0} x_{n+1}^{1-2s} \p_{n+1} u_N$, respectively.
Then,
\begin{align*}
u_D(x) &=  x_{n+1}^{2s}\sum\limits_{k=0}^{[m/2]} x_{n+1}^{2k} P_{m-2k}(x'),\\
u_N(x) & =  \sum\limits_{k=0}^{[m/2]}x_{n+1}^{2k} P_{m-2k}(x'),
\end{align*}
with $m\in \N\cup\{0\}$ and $P_{m-2k}$ being a polynomial of degree $m-2k$.
\end{prop}

\begin{proof}
We begin with the Dirichlet case. Carrying out a Fourier transform (interpreted as a tempered distribution) and dividing by $x_{n+1}^{1-2s}$ yields an ODE for $\hat{u}(\xi',x_{n+1})$:
\begin{equation}
\label{eq:ODE1}
\begin{split}
\left(\p_{n+1}^2 + \frac{1-2s}{x_{n+1}}\p_{n+1} - |\xi'|^2\right) \hat{u} &= 0 \mbox{ for } x_{n+1}\in(0,\infty),\\
\hat{u}(\xi',0) & = 0.
\end{split}
\end{equation}
Similarly as in the proof of Proposition \ref{prop:nD_mDN}, homogeneity further implies that
\begin{align}
\label{eq:bHom}
|\hat{u}(\xi',x_{n+1})| \leq C \max\{|\xi'|^{-1},x_{n+1}\}^{\kappa+n}.
\end{align}
Making the ansatz $\hat{u}(\xi',x_{n+1}) = x_{n+1}^{s}v(\xi',x_{n+1})$, the ODE from (\ref{eq:ODE1}) is transformed into a modified Bessel equation
\begin{align*}
x_{n+1}^2 v'' + x_{n+1}v' - (|\xi'|^2x_{n+1}^2 +s^2)v = 0,
\end{align*}
where differentiation with respect to $x_{n+1}$ has been abbreviated with the dashes and $\xi'$ plays the role of a parameter.
The general solution of this ODE is of the form
\begin{align*}
v(\xi',x_{n+1}) = C_1(\xi')I_s(|\xi'|x_{n+1}) + C_2(\xi')K_s(|\xi'|x_{n+1}).
\end{align*}
The asymptotics for $I_s, K_s$ (c.f. the proof of Proposition \ref{prop:nD_mDN}) and the Dirichlet data (i.e. the limit $x_{n+1}\rightarrow 0$) imply that for $|\xi'|\neq 0$, $C_2=0$. Considering the asymptotics $x_{n+1}\rightarrow \infty$ in combination with the bound (\ref{eq:bHom}) and the exponential growth of $I_s$ then also results in $C_1 = 0$. Hence, $v$ is supported in $\xi'=0$. Thus $\hat{u}$ can be written as
\begin{align*}
\hat{u}(\xi',x_{n+1}) = \sum\limits_{k=0}^{\infty}\sum\limits_{|\alpha|=k} c_{\alpha}(x_{n+1}) \delta^{(\alpha)}_{\{\xi'=0\}}(\xi'),
\end{align*}
where $\alpha=(\alpha_1,\dots,\alpha_n)\in (\N\cup \{0\})^n$ and $\delta^{(\alpha)}_{\{\xi'=0\}}(\xi')$ denotes an $\alpha$-fold distributional derivative of the delta-distribution $\delta_{\{\xi'=0\}}(\xi')$.
Using the homogeneity of $\hat{u}$, we further obtain that $c_{\alpha}(x_{n+1})$ is $\kappa-|\alpha|-n$ homogeneous. As $\hat{u}$ has to satisfy Dirichlet boundary conditions, we thus have that $|\alpha|\leq \kappa-n$, which yields that the series is a finite sum:
\begin{align}
\label{eq:decomp_D}
\hat{u}(\xi',x_{n+1}) = \sum\limits_{k=0}^{[\kappa-n]}\sum\limits_{|\alpha|=k} c_{\alpha}(x_{n+1}) \delta^{(\alpha)}_{\{\xi'=0\}}(\xi').
\end{align}
As in the proof of Proposition \ref{prop:nD_mDN}, we can further compute the functions $c_{\alpha}$ iteratively, by plugging it into (\ref{eq:ODE1}) and by testing with test functions which vanish of sufficiently high order. More precisely, we obtain that for $\alpha$ with $|\alpha|=\kappa-n$
\begin{align*}
c_{\alpha}'' + \frac{1-2s}{x_{n+1}}c_{\alpha}' = 0,
\end{align*}
i.e. $c_{\alpha}(x_{n+1})=\tilde{c}_{\alpha} x_{n+1}^{2s}$ with $\tilde{c}_{\alpha}\in \R$ (where we used the Dirichlet data). The functions $c_{\beta}(x_{n+1})$ with $\beta=[\kappa-n]-1$ also satisfy this equation and are hence of the same form. Inductively, for $l\in\{2,\dots,[\kappa-n]-1\}$, $\tilde{\alpha}$ with $|\tilde{\alpha}|=[\kappa-n]-l$ and $\tilde{\beta}$ with $|\tilde{\beta}|=[\kappa-n]-l+2$ we have
\begin{align*}
c_{\tilde{\alpha}}'' + \frac{1-2s}{x_{n+1}}c_{\tilde{\alpha}}'= d_l c_{\tilde{\beta}}.
\end{align*}
Inductively and by invoking the Dirichlet data, we thus infer that
\begin{align*}
c_{\tilde{\alpha}}(x_{n+1}) = x_{n+1}^{2s} \sum\limits_{j=1}^{l-1+[l/2]} \tilde{c}_{l,j,\tilde{\alpha}} x_{n+1}^{2j},
\end{align*}
for some constants $\tilde{c}_{l,j,\tilde{\alpha}}\in \R$. Due to the homogeneity condition on $c_{\tilde{\alpha}}$ from above, this further simplifies to
\begin{align*}
c_{\tilde{\alpha}}(x_{n+1}) = \tilde{c}_{\tilde{\alpha}} x_{n+1}^{2s+2(l-1+[l/2])} .
\end{align*}
for $\tilde{\alpha}$ with $|\tilde{\alpha}|=[\kappa-n]-l$ and $\tilde{c}_{\tilde{\alpha}}\in \R$.
Therefore transforming (\ref{eq:decomp_D}) back into $x$-coordinates, leads to
\begin{align*}
u(x',x_{n+1}) = x_{n+1}^{2s}\sum\limits_{k=0}^{[m/2]} x_{n+1}^{2k} P_{m-2k}(x'),
\end{align*}
where $m\in \N$ and $P_{m-2d}(x')$ denotes a polynomial of degree $m-2k$ depending on the $x'$ variables and $[\cdot]$ denotes the floor function. This concludes the proof for the case with Dirichlet data.\\
For the Neumann problem we argue analogously. As in the Dirichlet case, the Neumann boundary condition implies that the transformed function $\hat{u}(\xi',x_{n+1})$ has the form
\begin{align*}
\hat{u}(\xi',x_{n+1}) = \sum\limits_{k=0}^{[\kappa-n]} \sum\limits_{|\alpha|=k} c_{\alpha}(x_{n+1}) \delta^{(\alpha)}_{\{\xi'=0\}}(\xi').
\end{align*}
As before the finiteness of the sum is a consequence of homogeneity. We obtain the same recurrence relation as above for the coefficient functions $c_{[\kappa-n]-l}$. However, as these now satisfy Neumann data, we have that $c_{[\kappa-n]}(x_{n+1})=\tilde{c}\in \R$. Thus,
\begin{align*}
u(x',x_{n+1}) = \sum\limits_{k=0}^{[m/2]}x_{n+1}^{2k} P_{m-2k}(x'),
\end{align*}
which concludes the proof.
\end{proof}

\section{Appendix B}
\label{sec:AppB}

In this second appendix, we study the fractional Baouendi-Grushin Laplacian $\Delta_{G,s}$, which is related to the operator $L_s$ by a square root transformation (c.f. Section \ref{sec:open}). As the main result in Section \ref{sec:reg2} we provide the argument for Proposition \ref{prop:map} and show that $\Delta_{G,s}$ is invertible as a map from $X_{\alpha,\epsilon}$ to $Y_{\alpha,\epsilon}$. Here $X_{\alpha,\epsilon}$ and $Y_{\alpha,\epsilon}$ denote the function spaces which were introduced in Section~\ref{sec:function_spaces}. To this end, we prove a Schauder type apriori estimate (in our function spaces) by using a similar compactness argument as in Appendix A (c.f. Section \ref{sec:Schauder}) and a Schwartz kernel estimate (c.f. Section \ref{sec:kernel}).  Last but not least, in Sections \ref{sec:prove} and \ref{sec:Banach} we also deduce the characterization results of Proposition \ref{prop:decomp} for our function spaces $X_{\alpha,\epsilon}$ and $Y_{\alpha,\epsilon}$ and give the argument that they form Banach spaces which had been claimed in Proposition \ref{prop:Banach}.

\subsection{Fractional Baouendi-Grushin Laplacian, open up the domain}
\label{sec:open}
We first recall the definition of the fractional Baouendi-Grushin Laplacian 
$$\Delta_{G,s}=\sum_{i=1}^{2n}Y_i\omega(y)Y_i,$$ 
where $\{Y_i\}$ are the Baouendi-Grushin vector fields from Defnition~\ref{defi:BGgeo}, and $\omega(y)=|y_ny_{n+1}|^{1-2s}$ is the associated Muckenhoupt weight. As seen in the Example~\ref{ex:model} $\Delta_{G,s}$ is the push-forward operator of the operator $L_s=\nabla\cdot x_{n+1}^{1-2s}\nabla$ by means of the square root mapping.\\

Next we introduce some notation and recall some known results which are related to our set-up. We define the weighted $L^2$ Sobolev space associated with the Baouendi-Grushin vector fields $\{Y_i\}$: For $\Omega\subset \R^{n+1}$, we set
\begin{align*}
M^1_\omega(\Omega):=\{u: u\in L^2_\omega(\Omega), \quad Y_i u \in L^2_\omega(\Omega)\}.
\end{align*}
Here $L^2_w(\Omega):=L^2(\Omega, \omega(y)dy)$.  If $\Omega=\R^{n+1}$ we omit the domain dependence and write $M^1_\omega$ and $L^2_\omega$ for simplicity.  
Given $f$ such that $\omega^{-1}f\in L^2_\omega(\Omega)$, we say that $v\in M^1_\omega(\Omega)$ is a weak solution to $\Delta_{G,s}v=f$ iff for all $\phi\in C^\infty_c(\Omega)$
\begin{align*}
\int_{\Omega} Y_ivY_i\phi\ \omega (y)dy = \int _{\Omega}f\phi\ dy.
\end{align*}
From this the energy estimate is immediate: Suppose that $w\in M^1_\omega(\Omega)$ is a weak solution to $\Delta_{G,s}v=f$, then for any $\Omega'\Subset \Omega$, 
\begin{align*}
\sum_{i=1}^{2n}\|Y_i v\|_{L^2_\omega(\Omega')}\leq C \left(\|\omega^{-1}f\|_{L^2_\omega(\Omega)}+\|v\|_{L^2_\omega(\Omega)}\right),
\end{align*}
where $C=C(n,s,\Omega',\Omega)$. 
\\

In the sequel we will consider weak solutions to 
$$\Delta_{G,s}v=f \text{ in } \inte(\mathcal{B}_1^+), \quad\mathcal{B}_1^+=\mathcal{B}_1\cap \{y_n\geq 0, y_{n+1}\geq 0\},$$
with mixed Dirichlet and Neumann boundary conditions: 
$$w=0 \text{ on } \mathcal{B}_1\cap \{y_n=0\},\quad \lim_{y_{n+1}\rightarrow 0_+}\omega(y)\p_{n+1}w(y)=0\text{ on } \mathcal{B}_1\cap \{y_{n+1}=0\}.$$ 

Similarly as in the setting for the Laplacian it is possible to switch between the fractional Laplacian $L_s$ and the fractional Baouendi-Grushin Laplacian $\Delta_{G,s}$ by ``opening up the domain''. This allows us to transfer the approximation results for $L_s$ in Section~\ref{sec:approx_poly} to the corresponding approximation result for $\Delta_{G,s}$ along the set $\{y_n=y_{n+1}=0\}$ where the operator degenerates.

\begin{prop}[Opening up the domain]
 \label{prop:gr_approx}
Let $\Delta_{G,s}$ be the fractional Baouendi-Grushin Laplacian. Assume that $v\in  M^1_\omega(\mathcal{B}_1^+)$ is a weak solution to
\begin{equation}
\label{eq:fracLa_Gr}
\begin{split}
\Delta_{G,s}v&=f\mbox{ in } \inte(\mathcal{B}_1^+),\\
 v&=0 \mbox{ on } \mathcal{B}_1^+\cap\{y_{n}=0\},\\
 \p_{n+1} v  &= 0 \mbox{ on } \mathcal{B}_1^+\cap \{y_{n+1}=0\},
\end{split}
\end{equation}
where for all $r\in(0,1/2)$ the function $f$ satisfies 
\begin{align*}
\|\omega^{-1}f\|_{\tilde{L}_\omega^2(\mathcal{B}_r^+)}:=\left(\frac{1}{\omega(\mathcal{B}_r^+)}\int_{\mathcal{B}_r^+}|\omega^{-1}f|^2\omega(y) dy\right)^{\frac{1}{2}}\leq C_0 r^{2s+2\alpha}.
\end{align*}
Here $\omega(\mathcal{B}_r^+):=\int_{\mathcal{B}_r^+} \omega(y) dy$. Then, there exist a constant $C=C(n,s,\alpha)>0$ and a function 
\begin{align*}
h_{2+2s}(y):= y_{n}^{2s}\left( a_0 + \sum\limits_{i=1}^{n-1} a_i y_i + a_n y_{n}^2 + a_{n+1}y_{n+1}^2 \right)
\end{align*}
with coefficients
\begin{align*}
\sum\limits_{i=1}^{n+1}|a_i| \leq C\left( C_0+\| v\|_{\LL_{\omega}(B_1^+)} \right),
\end{align*}
such that 
\begin{align*}
&\|v-h_{2+2s}\|_{\LL_{\omega}(\mathcal{B}_r^+)} \leq C r^{2+2s +2\alpha} \left( C_0 + \| v\|_{\LL_{\omega}(\mathcal{B}_1^+)} \right)
\end{align*}
for all $r\in(0,1/2)$.
\end{prop}

\begin{proof}
\emph{Step 1: Square root transformation.}
First we observe the following relation between $\Delta_{G,s}$ and $L_s$ which is established by means of the square mapping: Suppose that $h$ is a solution to 
\begin{align*}
\Delta_{G,s}h&=0 \mbox{ in } Q_+, \\
h=0 \mbox{ on } \{y_n=0\}, \lim_{y_{n+1}\rightarrow 0_+} \omega(y)\p_{n+1}h&=0 \mbox{ on } \{y_{n+1}=0\}. 
\end{align*}
Let
\begin{align*}
\mathcal{T}:Q_+&\rightarrow \R^{n+1},\quad x=\mathcal{T}(y),\\
x_i &= y_i \mbox{ for } i\in \{1,...,n-1\}, \ x_{n} = \frac{1}{2}(y_n^2- y_{n+1}^2), \ x_{n+1}= y_n y_{n+1}.
\end{align*}
We define $\tilde{h}(x):= h(\mathcal{T}^{-1}(x))$. Then $\tilde{h}$ solves 
$$
L_s \tilde{h}=0 \mbox{ in } \R^{n+1}_+,
$$
with the mixed Dirichlet-Neumann boundary condition
\begin{align*}
\tilde{h}  = 0 \mbox{ on } \R^n\times\{0\}\cap \{x_{n}\leq 0\},\ \lim_{x_{n+1}\rightarrow 0_+}x_{n+1}^{1-2s}\p_{n+1}\tilde{h}=0 \text{ on } \R^n\times\{0\}\cap \{x_n>0\}.
\end{align*}
Thus, global homogeneous solutions to $\Delta_{G,s}h=0$ with Dirichlet-Neumann boundary data are characterized by invoking the characterization result for $L_s\tilde{h}=0$ from Proposition~\ref{prop:nD_mDN}. \\

\emph{Step 2: Approximation.}
With this at hand, we argue by compactness as in the proof for Proposition~\ref{prop:approx_main}. Since the proofs are very similar we only sketch the argument here. First one can normalize such that $\|v\|_{\tilde{L}^2_\omega(\mathcal{B}_1^+)}\leq 1$ and 
$$\sup_{r}r^{-(2s+2\alpha)}\|\omega^{-1}f\|_{\tilde{L}^2_\omega(\mathcal{B}_r^+)}\leq \delta$$ 
with $\delta$ small. Then by using similar arguments as in Lemma~\ref{lem:homo}--\ref{lem:it} we show that there exists a radius $r_0>0$ and an ``eigenpolynomial'' 
$$h_{2+2s,0}(y)=y_n^{2s}(a_0+\sum_{i=1}^{n-1}a_iy_i+a_ny_{n}^2+a_{n+1}y_{n+1}^2),$$ 
which solves $\Delta_{G,s}h=0$ in $Q_+$ with Dirichlet and Neumann boundary condition, such that $\|v-h_{2+2s,0}\|_{\tilde{L}^2_\omega(\mathcal{B}_{r_0}^+)}\leq C r_0^{2+2s+2\alpha}$. In the end, an inductive argument as in the proof for Proposition~\ref{prop:approx_main} gives the desired result.
\end{proof}

\begin{rmk}
By translation invariance of the operator, this approximation result holds at every point $y\in \mathcal{B}_1\cap\{y_n=y_{n+1}=0\}$.
\end{rmk}

\subsection{Invertibility of $\Delta_{G,s}$ from $X_{\alpha,\epsilon}$ to $Y_{\alpha,\epsilon}$}
\label{sec:reg2}
In this section we show that the fractional Baouendi-Grushin operator $\Delta_{G,s}$ is invertible as a map from the generalized H\"older space $X_{\alpha,\epsilon}$ to $Y_{\alpha,\epsilon}$ (c.f. Definition~\ref{defi:XY}). In Section \ref{sec:Schauder} we begin with the proof for the apriori estimate stated in Proposition~\ref{prop:map}. Then in Section \ref{sec:kernel} we provide the argument for the invertibility properties based on kernel estimates.

\subsubsection{Apriori Schauder estimate (Proof of Proposition \ref{prop:map})} 
\label{sec:Schauder}

As our main result in this subsection we prove the apriori estimate from Proposition \ref{prop:map}.

\begin{prop}
\label{prop:map2}
Suppose that $v\in L^\infty(\mathcal{B}_1^+)\cap M^1_\omega(\mathcal{B}_1)$ is a weak solution to $\Delta_{G,s}v=f$ in $\mathcal{B}_1^+$ with $f\in Y_{\alpha,\epsilon}$, and that it satisfies the following mixed Dirichlet and Neumann boundary condition: $v=0$ on $\mathcal{B}_1\cap \{y_n=0\}$ and $\lim_{y_{n+1}\rightarrow 0_+}\omega(y)\p_{n+1}v(y)=0$ on $\mathcal{B}_1\cap \{y_{n+1}=0\}$. Then, $v\in X_{\alpha,\epsilon}(\mathcal{B}_{1/2}^+)$ and it satisfies
\begin{align*}
\|v\|_{X_{\alpha, \epsilon}(\mathcal{B}_{1/2}^+)}\leq C\left(\|f\|_{Y_{\alpha, \epsilon}(\mathcal{B}_1^+)}+\|v\|_{L^\infty(\mathcal{B}_1^+)}\right).
\end{align*}
\end{prop}

\begin{proof}
\emph{Step 1: ``Eigenpolynomial'' approximation at $P$.}
Given $f\in Y_{\alpha,\epsilon}$, Proposition~\ref{prop:decomp} implies that $$f(y)=y_ny_{n+1}^{1-2s}f_0(y'')+y_{n+1}^{1-2s}r^{1+2\alpha-\epsilon}f_1(y)$$ 
with $\supp(f_0), \supp(f_1)\subset B_1'' \times \R^{2}$ and with $\|f_0\|_{ C^{0,\alpha}(\mathcal{B}_1\cap P)}+\|f_1\|_{ C^{0,\epsilon}_\ast(\mathcal{B}_1^+)}\leq C\|f\|_{Y_{\alpha,\epsilon}}$.  Let $\bar{y}\in \mathcal{B}_{1/2}\cap P$ be an arbitrary point.
Without loss of generality we may assume that $f_0(\bar{y}'')=0$ (as else we can subtract the correction $\frac{1}{1+2s}y_{n+1}^{1-2s}f_0(\bar{y}'')$ from $v$). Note that $f$ satisfies $\sup_{r\in (0,1/2)}r^{-(2s+2\alpha)}\|\omega^{-1}f\|_{\tilde{L}_\omega(\mathcal{B}_r^+)}\leq C\|f\|_{Y_{\alpha,\epsilon}}$. 
By virtue of Proposition \ref{prop:gr_approx}, there exists a function 
\begin{align*}
y_n^{2s}P_{\bar y}(y)= y_n^{2s}\left(a_0+ \sum_{i=1}^{n-1}a_iy_i+a_ny_n^{2} + a_{n+1}y_{n+1}^2\right),
\end{align*}
with coefficients $a_k$ depending on $\bar y$ and $\sum\limits_{k}|a_k|\leq C\|f\|_{Y_{\alpha, \epsilon}(\mathcal{B}_1^+)}$, such that
\begin{align*}
\|v-y_n^{2s}P_{\bar y}\|_{\tilde{L}^2_{\omega}(\mathcal{B}_r^+(\bar y))} \leq Cr^{2+2s+2\alpha}\left( \|f\|_{Y_{\alpha, \epsilon}(\mathcal{B}_1^+)}+\| v\|_{\tilde{L}^2_{\omega}(\mathcal{B}_1^+)}\right).
\end{align*}
for all $r\in (0,1/4)$.  \\

\emph{Step 2: Interpolation.} We consider
$$\tilde{v}_\lambda(y):=\frac{(v-y_n^{2s}P_{\bar y})(\bar y+ \delta_\lambda(y))}{\lambda^{2+2s+2\alpha}}, \quad \lambda \in (0,1/4),$$
where $\delta_\lambda (y)= (\lambda ^2 y'',\lambda y_n, \lambda y_{n+1})$. By Step 1, 
$$\|\tilde{v}_\lambda\|_{\tilde{L}^2_{\omega}(\mathcal{B}_4^+)} \leq C\left( \|f\|_{Y_{\alpha, \epsilon}(\mathcal{B}_1^+)}+\|v\|_{\LL_{\omega}(\mathcal{B}_1^+)}\right).$$ 
Moreover, $\tilde{v}_\lambda$ solves
\begin{align*}
\D_{G,s}\tilde{v}_\lambda(y)&=\lambda^{-(2-2s+2\alpha)} f(\bar y+\delta_\lambda y)\\
&= y_{n+1}^{1-2s} \lambda^{-2\alpha} f_0 + y_{n+1}^{1-2s}r(y)^{1+2\alpha-\epsilon}\lambda^{-\epsilon}f_1(\bar y+\delta_\lambda y)=:f_\lambda(y).
\end{align*}
In the region $y\in \mathcal{C}^+_2:=\{y: |y''|< 1, \frac{1}{16}< y_n^2+y_{n+1}^2< 4, y_n>0\}$, the operator $\D_{G,s}$ can be viewed as
\begin{align*} 
\mathcal{L}_s:= (y_n y_{n+1})^{1-2s}\Delta'' + \p_n ((y_n y_{n+1}) ^{1-2s}\p_n) +\p_{n+1}((y_n y_{n+1})^{1-2s}\p_{n+1}),
\end{align*}
i.e. the weight $(y_n^2 + y_{n+1}^2)$ in front of the tangential Laplacian can be ignored. 
Moreover, since $f_0\in C^{0,\alpha}$, $f_1\in C^{0,\epsilon}_\ast$ and since both vanish at $P$, the function $f_\lambda$ satisfies $y_{n+1}^{2s-1}f_\lambda(y)\in C^{0,\epsilon}(\mathcal{C}^+_2)$. \\
We now distinguish between three regions in which the operator $\mathcal{L}_s$ behaves differently:
\begin{align*}
C_{D}&:= \mathcal{C}^+_2\cap \{0\leq y_n<1/8\},\\
C_{N}&:=\mathcal{C}^+_2\cap \{|y_{n+1}|<1/8\}, \\
C_{E}&:= \mathcal{C}^+_2 \setminus (C_N \cup C_{D}).
\end{align*}
These correspond to the regions in which the equation is governed by a fractional Laplacian with Dirichlet ($C_D$) or Neumann data ($C_N$) or where the equation becomes uniformly elliptic ($C_E$). We discuss these cases separately:
\begin{itemize}
\item[(i)] We observe that in the region $C_D$ we have $y_{n+1}>1/8$. Hence, $y_{n+1}^{1-2s}$ is smooth in this region. Thus, the operator $\mathcal{L}_s$ can be viewed as a simple variation of $\nabla \cdot y_{n}^{1-2s}\nabla$. Furthermore, we note that $\tilde{v}_\lambda=0$ vanishes continuously on $\mathcal{C}_2^+\cap \{y_n=0\}$. Therefore, the up to the boundary apriori estimate from Proposition~\ref{prop:Dirichlet} applies. 
\item[(ii)] In the region $C_N$ we have $y_n>1/8$. Thus, we can invoke the apriori estimate with the Neumann data from Proposition~\ref{prop:Neumann}. 
\item[(iii)] In the remaining region $C_E$ the operator $\mathcal{L}_s$ is uniformly elliptic, therefore classical Schauder estimates hold. 
\end{itemize}
Combining the three cases from above, therefore leads to the following apriori estimate for $\tilde{v}_\lambda$:
\begin{equation}
\label{eq:apriori2}
\begin{split}
&\|y_n^{1-2s}\p_n \tilde{v}_\lambda\|_{C^{0,\epsilon}(\mathcal{C}^+_1)}+ \sum_{i=1}^{n-1}\|y_n^{-2s}\p_i\tilde{v}_\lambda\|_{C^{0,\epsilon}(\mathcal{C}^+_1)}+\|y_n^{-2s}y_{n+1}^{-1}\p_{n+1}\tilde{v}_\lambda\|_{C^{0,\epsilon}(\mathcal{C}_1^+)}\\
&+ \|y_n^{-2s} \tilde{v}_\lambda \|_{C^{0,\epsilon}(\mathcal{C}^+_1)} +\sum\limits_{i,j=1}^{n+1}\|\p_{i}y_n^{1-2s}\p_j\tilde{v}_\lambda\|_{C^{0,\epsilon}(\mathcal{C}^+_1)}\\
&\leq C\left( \|y_{n+1}^{2s-1}f_\lambda\|_{C^{0,\epsilon}(\mathcal{C}^+_2)}+ \| \tilde{v}_\lambda\|_{L^2_{\omega}(\mathcal{C}^+_2)}\right).
\end{split}
\end{equation}
Scaling \eqref{eq:apriori2} back results in
\begin{align*}
\left[ Y_i y_n^{1-2s}Y_j (v-y_n^{2s}P_{\bar y})\right]_{\dot{C}^{0,\epsilon}_{\ast}(\mathcal{C}^+_\lambda(\bar y))}\leq \lambda ^{1+2\alpha-\epsilon}\left( \|f\|_{Y_{\alpha, \epsilon}(\mathcal{B}_1^+)}+\|v\|_{\LL_{\omega}(\mathcal{B}_1^+)}\right),
\end{align*}
where $\mathcal{C}_\lambda^+(\bar y)=\{y:|y-\bar y|<\lambda^2, \lambda^2/4 \leq y_n^2+y_{n+1}^2\leq \lambda^2, y_n>0\}$.\\

Applying this to every $\bar y\in P\cap \mathcal{B}_{1/2}$ and every $\lambda \in (0,1/4)$, yields the boundedness of $\|v\|_{X_{\alpha,\epsilon}(\mathcal{B}_{1/2}^+)}$, i.e. a local version of the estimate from Proposition \ref{prop:map}. We note that from the specific expression of the approximation eigenpolynomials $y_n^{2s}P_{\bar y}(y)$, the boundary condition of $v$ on $\mathcal{B}_{1/2}\cap P$ are satisfied.  
\end{proof}

\subsubsection{Invertibility} 
\label{sec:kernel}

The main result of this section is the proof of the inverbility result:

\begin{prop}[Invertibility]
\label{prop:invert_global}
Let $\D_{G,s}$ be the fractional Baouendi-Grushin Laplacian and let $X_{\alpha,\epsilon}, Y_{\alpha,\epsilon}$ be the function spaces from Definition \ref{defi:XY} with $\epsilon\leq \alpha$. Then the operator $\Delta_{G,s}:X_{\alpha,\epsilon}\rightarrow Y_{\alpha,\epsilon}$ is invertible. Moreover, 
$$\|v\|_{X_{\alpha,\epsilon}}\leq C\|\Delta_{G,s}v\|_{Y_{\alpha,\epsilon}}.$$
\end{prop}

To show the above result, we first prove a global $L^\infty $ estimate for the solution by using a kernel estimate (c.f. Lemma~\ref{lem:ker}).\\

We recall several auxiliary results which we will use in the sequel.
As a central tool, we will Sobolev embedding for the space $M^1_\omega$, which we briefly recall here: We note that the Baouendi-Grushin vector fields $\{Y_i\}$ satisfy the H\"ormander vector field condition. Furthermore, for all $s\in (0,1)$ the weight $\omega(y)=|y_ny_{n+1}|^{1-2s}$ satisfies the $A_2$ condition:
\begin{align*}
\left(\int_{\mathcal{B}_R} \omega dy\right)\left(\int_{\mathcal{B}_R} \omega^{-1} dy\right)\leq C_s \left(\int_{\mathcal{B_R}} dy\right)^2
\end{align*}
for all Grushin balls $\mathcal{B}_R\subset \R^{n+1}$. Under these conditions the following Sobolev inequality holds true (c.f. \cite{Lu92}): 
Let $u\in M^1_\omega$, then $u\in L^p_\omega$ with $\frac{1}{p}+\frac{1}{Q}=\frac{1}{2}$, where $Q=2(n+1-2s)$, and moreover
\begin{align}
\label{eq:sob}
\|u\|_{L^p_\omega}\leq C_{n,s} \|Y_i u\|_{L^2_\omega}.
\end{align}

We also recall the definition of the homogeneous Sobolev spaces: Let $\dot{M}^1_\omega$ denote the homogeneous Sobolev space, which is the completion of $C^{\infty}_{0}$ with respect to the homogeneous norm $\|v\|_{\dot{M}^{1}_\omega}:= \sum_i\|Y_i v\|_{L^2_\omega}$.
Let $\dot{M}^{-1}_\omega:=(\dot{M}^1_\omega)^\ast$ denote the dual space of $\dot{M}^1_\omega$. By the Riesz representation theorem, for each $F\in \dot{M}^{-1}_\omega$, there exists $F^i\in L^2_\omega$, $i=1,\dots, 2n$, such that $(F,v)=\sum_i\int F^i Y_i v\ \omega $ for any $v\in \dot{M}^1_\omega$. Moreover, $\|F\|_{\dot{M}^{-1}_\omega}=\inf_{F^i} \sum_i \|F^i\|_{L^2_\omega} $.\\

Now we show a global $L^\infty$ estimate for the weak solution to $\Delta_{G,s}u=f \mbox{ in } \R^{n+1}$ with $f$ compactly supported and $\omega^{-1}f$ in some Morrey type space.

\begin{lem}
\label{lem:ker}
Given $f$ with $\omega^{-1}f\in L^2_\omega$ and $(supp f)\subset \mathcal{B}_1$. Suppose that $$\sup _{\mathcal{B}_r(y)}\left(\frac{1}{\omega(\mathcal{B}_r(y))}\int_{\mathcal{B}_r(y)} (\omega^{-1}f)^2 \ \omega dy\right)^{\frac{1}{2}} \leq C_0 r^{-\gamma}\text{ for some }\gamma\in [0,2).$$ 
Then there exists a unique weak solution $u\in \dot{M}^1_\omega\cap L^p$ which solves $\D_{G,s}u= f$ in $\R^{n+1}_+$.
Moreover, $u\in L^\infty$ and it satisfies $\|u\|_{L^\infty(\R^{n+1})}\leq c_\gamma C_0$ for some $c_\gamma>0$.
\end{lem}

\begin{proof}
\emph{Step 1: Existence and representation.}
Consider the Hilbert space $(\dot{M}^1_\omega\cap L^p, \|\cdot\|_{\dot{M}^1_\omega})$ for $p$ satisfying $\frac{1}{p}+\frac{1}{Q}=\frac{1}{2}$. 
By the Lax-Milgram theorem, for any $F \in \dot{M}^{-1}_w$, there exists a unique function $u\in \dot{M}^1_w\cap L^p$ such that $(F,\phi)=\int F^i Y_i\phi \omega dy =\int Y_iu Y_i\phi \omega dy $ for any $\phi\in C^\infty_c$. 
Moreover, by energy estimates, 
$
\|Y_iu\|_{L^2_\omega}\leq  \|F\|_{\dot{M}^{-1}_\omega}.
$ We will write $u:=\Delta_{G,s}^{-1}F$.\\

By the Sobolev embedding (\ref{eq:sob}) and by duality, we have the embedding $L^{p'}_\omega\hookrightarrow \dot{M}^{-1}_\omega$, where $p'$ is the H\"older conjugate of $p$. Thus, for any $g\in L^{p'}_\omega$, there exists a unique function $u:=\Delta_{G,s}^{-1}g\in \dot{M}^1_\omega\cap L^p$ such that $\int g\phi\omega dy=\int Y_i u Y_i\phi \omega dy$ for any $\phi \in C^\infty_c$, with $\|Y_iu\|_{L^2_\omega}\leq C\|g\|_{\dot{M}^{-1}_\omega}$. By Sobolev embedding \eqref{eq:sob},
\begin{align}\label{eq:emb}
\|\Delta_{G,s}^{-1}g\|_{L^p_\omega}\leq C\|Y_i\Delta_{G,s}^{-1}g\|_{L^2_\omega}\leq C \|g\|_{\dot{M}^{-1}_\omega}\leq C\|g\|_{L^{p'}_\omega}.
\end{align}
Thus, by the Schwartz kernel theorem there exists a distribution kernel $k(z,y)\in \mathcal{D}'(\R^{n+1}\times\R^{n+1})$ such that the right inverse of $\D_{G,s}$ is represented as
\begin{align*}
u(z)=(\Delta_{G,s}^{-1}g)(z)=\int_{\R^{n+1}} k(z,y) g(y)\omega(y) dy.
\end{align*}

\emph{Step 2: Kernel estimates.}
We estimate the distribution kernel.
Suppose that $\supp (g) \subset \mathcal{B}_1(y)$ for some $y$, then outside $\mathcal{B}_1(y)$, $u$ is a solution to the homogeneous equation $\Delta_{G,s}u=0$. By Moser's inequality (c.f. equation (2.2) in \cite{Lu92}), for any $z\in \R^{n+1}\setminus \mathcal{B}_3(y)$ 
\begin{align*}
\sup_{\tilde{z}\in \mathcal{B}_1(z)}|u(\tilde{z})|\leq C \|u\|_{L^2(\mathcal{B}_2(z))}.
\end{align*}
By H\"older's inequality and \eqref{eq:emb} and the fact that $g$ is compactly supported, we infer that
\begin{align*}
\sup_{\tilde{z}\in \mathcal{B}_1(z)}|u(\tilde{z})|\leq C \|g\|_{L^{2}_w(\mathcal{B}_1(y))}.
\end{align*}
 Hence, for fixed $z$, the mapping $g\mapsto u(z)=\Delta_{G,s}^{-1}g$ is a continuous linear functional from $L^2_w(\mathcal{B}_1(y))$ to $\R$, which is represented by $k(z,\cdot)\in L^2_w(\mathcal{B}_1(y))$ with $\|k(z,\cdot)\|_{L^2_w(\mathcal{B}_1(y))}\leq C$. 
Scaling thus yields that 
$$\|k(z,\cdot)\|_{L^2_w(\mathcal{B}_\lambda (y))}\leq C \lambda^{2}\omega(\mathcal{B}_\lambda(y))^{-\frac{1}{2}}$$ for any $y$ and $z$ with $d_G(y,z)\geq 3\lambda$. Here $\lambda>0$ and $C$ is independent of $\lambda$.\\

\emph{Step 3: $L^\infty$ estimate.} With the previous considerations at hand we proceed to the claimed $L^\infty $ estimate. Let $f$ satisfy the assumptions of the lemma, and let $u=\Delta_{G,s}^{-1}(\omega^{-1}f)=\int_{\R^{n+1}}k(z,y)(\omega(y)^{-1}f(y))\omega(y)dy \in \dot{M}^1_{\omega}\cap L^p$. For any $z\in \R^{n+1}$, by H\"older's inequality,
\begin{align*}
|u(z)|&\leq \sum_{j=0}^{\infty}\|k(z,\cdot)\|_{L^2_\omega(A_j)}\|\omega^{-1}f\|_{L^2_\omega(A_j)},
\end{align*}
where $A_j:=\mathcal{B}_{2^{-j}}(z)\setminus \mathcal{B}_{2^{-j-1}}(z)$ are dyadic annuli centered at $z$. 
Using the estimate on the kernel $k(z,\cdot)$ from step 2, as well as our growth assumption on $f$, we further obtain 
\begin{align*}
|u(z)|&\leq C\sum_j (2^{-j})^{2}\omega(\mathcal{B}_{2^{-j}}(z))^{-\frac{1}{2}}C_0(2^{-j})^{-\gamma}\omega(\mathcal{B}_{2^{-j}}(z))^{\frac{1}{2}}\\
&\leq C C_0\sum_j(2^{-j})^{2-\gamma}= c_\gamma C_0.
\end{align*}
This completes the proof. 
\end{proof}

In the end, we combine Proposition~\ref{prop:map2} and Lemma~\ref{lem:ker} to conclude the statement of Proposition~\ref{prop:invert_global}:

\begin{proof}[Proof of Proposition~\ref{prop:invert_global}]
Given $f\in Y_{\alpha,\epsilon}$, we extend $f$ to the whole space by reflection about $y_n$ and $y_{n+1}$. By Proposition~\ref{prop:decomp} it is not hard to see that $\omega^{-1} f\in L^2_\omega$, $\supp (\omega^{-1}f)\subset \mathcal{B}_1$, and moreover,
\begin{align*}
\sup _{\mathcal{B}_r(y)}\left(\frac{1}{\omega(\mathcal{B}_r(y))}\int_{\mathcal{B}_r(y)} (\omega^{-1}f)^2 \ \omega dy\right)^{\frac{1}{2}} \leq \|f\|_{Y_{\alpha,\epsilon}} r^{ \alpha-(1-2s)}.
\end{align*}
Thus $\omega^{-1}f$ satisfies the assumption of Lemma~\ref{lem:ker}. 
Then, by Lemma~\ref{lem:ker}, there exists a unique weak solution $v\in \dot{M}^1_\omega\cap L^p$ to $\D_{G,s}v=f$ in the sense that
\begin{align*}
\int\limits_{\R^{n+1}} Y_i v  Y_i \varphi\ \omega(y) dy = \int \limits_{\R^{n+1}} (\omega^{-1}f) \varphi\ \omega(y)dy, \quad \forall \varphi \in C_0^\infty.
\end{align*}
Moreover, 
\begin{align}
\label{eq:Linfty}
\|v\|_{L^\infty(\R^{n+1})}\leq C\|f\|_{Y_{\alpha,\epsilon}}.
\end{align}
By Proposition~\ref{prop:map2}, $v\in X_{\alpha,\epsilon}(\mathcal{B}_R^+)$ for each $R>0$. Moreover, it satisfies $\|v\|_{X_{\alpha,\epsilon}(\mathcal{B}_R^+)}\leq C(\|f\|_{Y_{\alpha,\epsilon}}+R^{-1-2s}\|v\|_{L^\infty(\mathcal{B}_{2R}^+)})$ with a constant $C>0$ which is independent of $R$ (here we used that the norm of $X_{\alpha,\epsilon}$ is a homogeneous norm). Combining this with the $L^\infty$ estimate in \eqref{eq:Linfty}, we infer that $v\in X_{\alpha,\epsilon}$ and that it satisfies $\|v\|_{X_{\alpha,\epsilon}}\leq C\|f\|_{Y_{\alpha,\epsilon}}$. Note that this estimate also implies the uniqueness of the solution.
\end{proof}

\subsection{Characterization and Banach property of the function spaces}
In this section we show the characterization of the functions spaces $X_{\alpha,\epsilon}$ and $Y_{\alpha,\epsilon}$ stated in Proposition~\ref{prop:decomp}, and the Banach property of the function spaces stated in Proposition~\ref{prop:Banach}. As these follow from ideas which are similar to those presented in \cite{KRSIII} and as they would have obscured the structure of the main argument, we decided to present them separately in this appendix.

\subsubsection{Proof of Proposition \ref{prop:decomp}}
\label{sec:prove}
In this section we provide the proof of Proposition \ref{prop:decomp}

We begin with the discussion of the decomposition and regularity of $f$. We show that if $f\in X_{\alpha,\epsilon}$, the decomposition and the estimates from the Proposition hold.\\ 
We define $f_0(y''):= \frac{Q_{y,1}(y)}{y_n}$, where $Q_{y,1}(y)$ denotes the first order approximating polynomial of $\tilde{f}(y):= y_{n+1}^{2s-1} f(y)$. As $\tilde{f}$ was assumed to be $C^{1,\alpha}_{\ast}$ at $P$, these limits exist. We further define
\begin{align*}
f_1(y):= y_{n+1}^{2s-1} r(y)^{-1-2\alpha+\epsilon}(f(y)-f_0(y'')y_n y_{n+1}^{1-2s}).
\end{align*}
By Remark \ref{rmk:supp} and the definition of the norm on $Y_{\alpha,\epsilon}$ this quantity is finite. Hence, it remains to prove the claimed regularity properties for these two functions. We begin with the estimate for $f_0$: Let $y_1,y_2 \in P$ be given. Define $\bar{y}\in P$ such that $y_1, y_2 \in \mathcal{B}_1(\bar{y})$. Further let $y\in \mathcal{B}_1(\bar{y})$ be another point with the property that $y_n=d_G(\bar{y},y)=|y_1-y_2|^{1/2}$, $y_{n+1}=0$ and with $d_G(y,y_1) \sim d_G(y,y_2) \sim d_G(y,\bar{y})$.
By to Remark \ref{rmk:supp} we infer that
\begin{align*}
|d_G(y,y_i)^{-1-2\alpha}y_{n+1}^{2s-1}(f(y) - y_{n+1}^{1-2s}Q_{y_i,1})| \leq C \mbox{ for } i\in\{1,2\}.
\end{align*}
Thus, the triangle inequality and the choice of $y,\bar{y}$ yield
\begin{align*}
|Q_{y_1,1}(y)-Q_{y_2,1}(y)| \leq C (d_G(y,y_1)^{1+2\alpha} + d_G(y,y_2)^{1+2\alpha} ) \leq C d_G(y, \bar{y})^{1+2\alpha} .
\end{align*}
Using the form of $Q_{\cdot,1}$ and the choice of $y$ again by dividing by $y_n$, we obtain that
\begin{align*}
| f_0(y_1) - f_0(y_2)| \leq C d_G(y,\bar{y})^{2\alpha} = C |y_1-y_2|^{\alpha}.
\end{align*} 
This proves the claimed regularity of $f_0$. We proceed by discussing the regularity of $f_1(y)$. We first note that we can always bound 
\begin{align*}
|f_1(y)| \leq C r(y)^{\epsilon}.
\end{align*}
Therefore, for points $y_1,y_2\in Q_+$ with $\max\{r(y_1),r(y_2)\}\leq 10 d_G(y_1,y_2)$ we infer
\begin{align*}
|f_1(y_1)-f_1(y_2)| \leq |f_1(y_1)| + |f_1(y_2)| \leq C(r(y_1)^{\epsilon} + r(y_2)^{\epsilon}) \leq Cd_G(y_1,y_2)^{\epsilon}.
\end{align*}
In the case that $y_1,y_2$ are such that $\max\{r(y_1),r(y_2)\}\geq 10 d_G(y_1,y_2)$, there always exists a point $\bar{y}\in P$ with $y_1,y_2\in C_1^+(\bar{y})$. In this case the estimate follows by an application of the triangle inequality and the result for $f_0$.  Indeed, setting $r_i:=r(y_i)$, for $i=1,2$, we note that in this case $d_G(y_2,y''_1)\sim r_1\sim r_2$.  By the triangle inequality,
\begin{align*}
&|f_1(y_1)-f_1(y_2)| \\
& = \left|r_1^{-1-2\alpha+\epsilon}[(y_1)_{n+1}^{2s-1}f(y_1)-f_0(y''_1)(y_1)_n] - r_2^{-1-2\alpha+\epsilon}[(y_2)_{n+1}^{2s-1}f(y_2)-f_0(y''_2)(y_2)_n] \right|\\
& \leq \left|r_1^{-1-2\alpha+\epsilon}[(y_1)_{n+1}^{2s-1}f(y_1)-f_0(y_1'')(y_1)_n]\right. \\
&\quad  \left.- d_G(y_2,y''_1)^{-1-2\alpha+\epsilon}[(y_2)_{n+1}^{2s-1}f(y_2)- f_0(y_1'')(y_2)_n]\right|\\
& \quad + \left|(d_G(y_2,y''_1)^{-1-2\alpha+\epsilon}-r_2^{-1-2\alpha+\epsilon})[(y_2)_{n+1}^{2s-1}f(y_2)-f_0(y_1'')(y_2)_n]\right|\\
& \quad + \left|r_2^{-1-2\alpha+\epsilon}(f_0(y_1'')-f_0(y_2''))(y_2)_n\right|.
\end{align*}
Using the bound on the norm of $Y_{\alpha,\epsilon}$, the estimate for $f_0$ and the Hölder continuity of $x\mapsto x^{\epsilon}$ in combination with the form of the $Y_{\alpha,\epsilon}$ norm, then bounds the three terms.
Conversely, we show that, if $f$ is of the form stated in Proposition \ref{prop:decomp}, then $f\in Y_{\alpha,\epsilon}$. Indeed, for any $\bar y\in P$, let $P_{\bar y, 1}(y):=f_0(\bar y)y_n$. Therefore,
\begin{align*}
|y_{n+1}^{2s-1}f(y)-P_{\bar y,1}(y)|&\leq |f_0(\bar y)-f_0(y'')|y_n+r^{1+2\alpha-\epsilon}|f_1(y)|\\
&\leq [f_0]_{\dot{C}^{0,\alpha}}|y''-\bar y''|^{\alpha}y_n+r(y)^{1+2\alpha-\epsilon}|f_1(y)|\\
&\leq C[f_0]_{\dot{C}^{0,\alpha}}d_G(y,\bar y)^{1+2\alpha}.
\end{align*}
Thus, $y_{n+1}^{2s-1}f(y)$ is $C^{1,\alpha}_\ast$ at $\bar y\in P$. It hence remains to discuss the boundedness of $\|f\|_{Y_{\alpha,\epsilon}}$. This however follows from the regularity of $f_0$ and $f_1$.\\
Conversely, if $f$ is of the form stated in Proposition \ref{prop:decomp}, then $y_{n+1}^{2s-1}f$ is $C^{1,\alpha}_{\ast}$ at $P$ and $\|f\|_{Y_{\alpha,\epsilon}}\leq C([f_0]_{\dot{C}^{0,\alpha}} + [f_1]_{\dot{C}^{0,\epsilon}})$. The remaining properties satisfied by functions in the space $Y_{\alpha,\epsilon}$ follow by assumption.\\

We proceed with the characterization of the space $X_{\alpha,\epsilon}$. First we note that if $v\in X_{\alpha,\epsilon}$, then the boundary and pointwise conditions which are imposed in the definition of $X_{\alpha,\epsilon}$ imply that the (homogeneous) approximating polynomial $P_{\bar{y},2}^s(y)$ of $y_{n}^{-2s}v$ at $\bar{y}\in P$ is of the form stated in the decomposition in (a). The regularity results for the functions $c_0, a_0, a_1$ and $C_{i}, C_{ij}$ follow as in the proof for the space $Y_{\alpha,\epsilon}$.
\\
In the end, we show that if a function $v$ satisfies the conditions (a)-(d) in Proposition~\ref{prop:decomp}, then $v\in X_{\alpha,\epsilon}$. It is not hard to see that the boundary conditions are satisfied. We claim that $v\in C^{2,2\alpha}_{\ast}$: For each $\bar y\in P$, we set 
\begin{align*}
P_{\bar y, 2}(y):=c_0(\bar y)+\sum_{i=1}^{n-1}\p_ic_0(\bar y)(y_i-\bar y_i)+a_0(\bar y)y_n^2+a_1(\bar y)y_{n+1}^2.
\end{align*}
Then, 
\begin{align*}
&|y_n^{-2s}v(y)-P_{\bar y, 2}(y)|\\
&\leq |c_0(y'')-c_0(\bar y)-\sum_{i=1}^{n-1}\p_ic_0(\bar y)(y_i-\bar y_i)|+|a_0(y'')-a_0(\bar y)|y_n^2+|a_1(y'')-a_1(\bar y)|y_{n+1}^2\\
& \quad +r^{2+2\alpha-\epsilon}|C_0(y)|\\
&\leq [\nabla c_0]_{\dot{C}^{0,\alpha}}|y''-\bar y|^{1+\alpha}+[a_0]_{\dot{C}^{0,\alpha}}|y''-\bar y|^{\alpha}y_n^2+[a_1]_{\dot{C}^{0,\alpha}}|y''-\bar y|^{\alpha}|y_{n+1}|^2\\
& \quad+[C_0]_{\dot{C}^{0,\epsilon}_\ast}r^{2+2\alpha},\\
&\leq C\left([\nabla c_0]_{\dot{C}^{0,\alpha}}+[a_0]_{\dot{C}^{0,\alpha}}+[a_1]_{\dot{C}^{0,\alpha}}+[C_0]_{\dot{C}^{0,\epsilon}_\ast}\right)d_G(y,\bar y)^{2+2\alpha},
\end{align*}
with $C$ independent of $\bar y$. Hence, $y_n^{-2s}v$ is $C^{2,2\alpha}_\ast$ at each $\bar y\in P$. To show the boundedness of the remaining terms in the norm $\|v\|_{X_{\alpha,\epsilon}}$, we argue similarly as for the space $Y_{\alpha,\epsilon}$.

\subsubsection{Proof of Proposition \ref{prop:Banach}}
\label{sec:Banach}
In this section we show that $X_{\alpha,\epsilon}$ and $Y_{\alpha,\epsilon}$ are Banach spaces.

\begin{proof}[Proof of Proposition \ref{prop:Banach}]
For the space $Y_{\alpha,\epsilon}$ the Banach property follows from the compact support assumption: Using the characterization from Proposition \ref{prop:decomp}, we have that for a given function $f\in Y_{\alpha,\epsilon}$ the functions $f_0, f_1$ are supported only in $B''_1 \times \R^2$. Thus, the homogeneous Hölder norms control the lower order $L^{\infty}$ norms. This yields the Banach property for $Y_{\alpha,\epsilon}$.\\

Next we show that $X_{\alpha, \epsilon}$ is complete under the homogeneous norm. Indeed, for any $v\in X_{\alpha,\epsilon}$,  it is not hard to check that $\D_{G,s}v\in Y_{\alpha,\epsilon}$ and $\|\D_{G,s}v\|_{Y_{\alpha,\epsilon}}\leq \|v\|_{X_{\alpha,\epsilon}}$. By Lemma~\ref{lem:ker} we have $\|v\|_{L^\infty}\leq C\|\D_{G,s}v\|_{Y_{\alpha,\epsilon}}\leq C\|v\|_{X_{\alpha,\epsilon}}$. This then also implies $L^{\infty}$ bounds for the functions $c_0, a_0, a_1, C_i, C_{ij}$ in terms of $\|v\|_{X_{\alpha,\epsilon}}$: Indeed, recalling that $C_0(\bar{y})=0$ for $\bar{y}\in P$ and recalling the Hölder bound for $C_0$ in terms of $\|v\|_{X_{\alpha,\epsilon}}$, implies that
\begin{align*}
\bar{v}(y):= c_0(y'')y_n^{2s} + a_0(y'')y_n^{2+2s} + a_1(y'') y_n^{2s}y_{n+1}^2 \in L^{\infty}(\{y: \dist(y,P)\leq 2\})
\end{align*}
and $\|\bar{v}\|_{L^{\infty}(\{y: \dist(y,P)\leq 2\})} \leq C \|v\|_{X_{\alpha,\epsilon}}$. Now varying the values of $y_n, y_{n+1} \in \{y: \dist(y,P)\leq 2\}$ yields the desired bounds 
\begin{align*}
\|c_0\|_{L^{\infty}(\R^{n-1})} + \|a_0\|_{L^{\infty}(\R^{n-1})}+ \| a_1\|_{L^{\infty}(\R^{n-1})} \leq C \|v\|_{X_{\alpha,\epsilon}}.
\end{align*}
This then also implies the global $L^{\infty}$ bound for $C_0$ in terms of $\|v\|_{X_{\alpha,\epsilon}}$. We note that the $L^{\infty}$ bounds for $C_i, C_{ij}$, $i,j\in\{1,\dots,n+1\}$ follow from the a priori estimates for $\D_{G,s}$, e.g. a slight modification of the arguments in Section \ref{sec:reg1} also yields
\begin{align*}
\|d_G(\cdot,\bar{y})^{-(1+2\alpha)}Y_i y_n^{1-2s} Y_j (v- y_n^{2s}P_{\bar{y},2}^s)\|_{L^{\infty}(Q_+)} \leq C\left(\|\D_{G,s}v\|_{Y_{\alpha,\epsilon}}+\|v\|_{L^\infty}\right).
\end{align*}
Similar estimates hold for the properly weighted first derivatives.
\end{proof}

\bibliography{citations1}

\begin{thebibliography}{GPPG15}

\bibitem[Ang90]{An90}
Sigurd~B. Angenent.
\newblock Nonlinear analytic semiflows.
\newblock {\em Proceedings of the Royal Society of Edinburgh: Section A
  Mathematics}, 115(1-2):91--107, 1990.

\bibitem[BFRO15]{BFRO}
Bego{\~n}a Barrios, Alessio Figalli, and Xavier Ros-Oton.
\newblock Global regularity for the free boundary in the obstacle problem for
  the fractional {L}aplacian.
\newblock {\em arXiv preprint arXiv:1506.04684}, 2015.

\bibitem[Cam64]{Ca64}
Sergio Campanato.
\newblock Proprieta di una famiglia di spazi funzionali.
\newblock {\em Annali della Scuola Normale Superiore di Pisa-Classe di
  Scienze}, 18(1):137--160, 1964.

\bibitem[CROS16]{CROS16}
Luis Caffarelli, Xavier Ros-Oton, and Joaquim Serra.
\newblock Obstacle problems for integro-differential operators: Regularity of
  solutions and free boundaries.
\newblock {\em arXiv preprint arXiv:1601.05843}, 2016.

\bibitem[CS07]{CaS}
Luis~A. Caffarelli and Luis Silvestre.
\newblock An {E}xtension {P}roblem {R}elated to the {F}ractional {L}aplacian.
\newblock {\em Communications in Partial Differential Equations},
  32(8):1245--1260, 2007.

\bibitem[CSS08]{CSS}
Luis~A. Caffarelli, Sandro Salsa, and Luis Silvestre.
\newblock Regularity estimates for the solution and the free boundary of the
  obstacle problem for the fractional {L}aplacian.
\newblock {\em Inventiones mathematicae}, 171(2):425--461, 2008.

\bibitem[Dei10]{Dei10}
Klaus Deimling.
\newblock {\em Nonlinear functional analysis}.
\newblock Courier Corporation, 2010.

\bibitem[DSS14]{DSS14}
Daniela De~Silva and Ovidiu Savin.
\newblock Boundary {H}arnack estimates in slit domains and applications to thin
  free boundary problems.
\newblock {\em arXiv preprint arXiv:1406.6039}, 2014.

\bibitem[GPPG15]{GPPSVG15}
Nicola Garofalo, Arshak Petrosyan, Camelia~A Pop, and Mariana Smit~Vega Garcia.
\newblock Regularity of the free boundary for the obstacle problem for the
  fractional {L}aplacian with drift.
\newblock {\em arXiv preprint arXiv:1509.06228}, 2015.

\bibitem[Gri11]{Gr11}
Pierre Grisvard.
\newblock {\em Elliptic problems in nonsmooth domains}, volume~69.
\newblock SIAM, 2011.

\bibitem[KL12]{KL12}
Herbert Koch and Tobias Lamm.
\newblock Geometric flows with rough initial data.
\newblock {\em Asian Journal of Mathematics}, 16(2):209--235, 2012.

\bibitem[KMR97]{Ko97}
Vladimir~A Kozlov, VG~Mazia, and J{\"u}rgen Rossmann.
\newblock {\em Elliptic boundary value problems in domains with point
  singularities}, volume~52.
\newblock American Mathematical Soc., 1997.

\bibitem[KN77]{KN77}
David Kinderlehrer and Louis Nirenberg.
\newblock Regularity in free boundary problems.
\newblock {\em Annali della Scuola Normale Superiore di Pisa-Classe di
  Scienze}, 4(2):373--391, 1977.

\bibitem[KPS14]{KPS14}
Herbert Koch, Arshak Petrosyan, and Wenhui Shi.
\newblock Higher regularity of the free boundary in the elliptic {S}ignorini
  problem.
\newblock {\em Arxiv preprint arXiv:1406.5011}, 2014.

\bibitem[KRS15]{KRSI}
Herbert Koch, Angkana R{\"u}land, and Wenhui Shi.
\newblock The variable coefficient thin obstacle problem: Optimal regularity,
  free boundary regularity and first order asymptotics.
\newblock {\em Arxiv preprint, arXiv:1504.03525}, 2015.

\bibitem[KRS16]{KRSIII}
Herbert Koch, Angkana R{\"u}land, and Wenhui Shi.
\newblock The variable coefficient thin obstacle problem: Higher regularity.
\newblock {\em ArXiv preprint, arXiv:1605.02002}, 2016.

\bibitem[Lu92]{Lu92}
Guozhen Lu.
\newblock Existence and size estimates for the {G}reen's functions of
  differential operators constructed from degenerate vector fields.
\newblock {\em Communications in partial differential equations},
  17(7-8):143--160, 1992.

\bibitem[PP15]{PP15}
Arshak Petrosyan and Camelia~A Pop.
\newblock Optimal regularity of solutions to the obstacle problem for the
  fractional {L}aplacian with drift.
\newblock {\em Journal of Functional Analysis}, 268(2):417--472, 2015.

\bibitem[Sil07]{Si}
Luis Silvestre.
\newblock Regularity of the obstacle problem for a fractional power of the
  {L}aplace operator.
\newblock {\em Communications on Pure and Applied Mathematics}, 60(1):67--112,
  2007.

\bibitem[Ura87]{U87}
Nina~Nikolaevna Ural'tseva.
\newblock Regularity of solutions of variational inequalities.
\newblock {\em Russian Mathematical Surveys}, 42(6):191--219, 1987.

\bibitem[Wan92]{Wang92}
Lihe Wang.
\newblock On the regularity theory of fully nonlinear parabolic equations.
  {II}.
\newblock {\em Comm. Pure Appl. Math.}, 45(2):141--178, 1992.

\bibitem[Wan03]{Wa03}
Lihe Wang.
\newblock H{\"o}lder estimates for subelliptic operators.
\newblock {\em Journal of functional Analysis}, 199(1):228--242, 2003.

\end{thebibliography}
\bibliographystyle{alpha}
\end{document}